%% file: arxiv_version.tex
\documentclass[11pt,onecolumn]{article}
\usepackage[bottom=1.0in,left=1.1in,right=1.1in]{geometry}

\usepackage{caption} 
\usepackage{subcaption}
\usepackage{abstract}

\usepackage{amsmath,amsthm,epsfig,amssymb,amsbsy}
\usepackage{amsfonts}
\usepackage{tikz}
\usepackage{pgfplots}
\pgfplotsset{compat=1.16}
\usepackage{enumerate}
\usepackage{comment}
\usepackage{stmaryrd}
\usepackage[normalem]{ulem} 

\usepackage{incl_commands}
\usepackage{amsthm}
\usepackage{amsmath}

\usepackage{graphicx}
\usepackage{grffile}

\usepackage[outdir=./]{epstopdf}

\newenvironment{keywords}{\begin{paragraph}{Keywords:}
}
{
\end{paragraph}
}
    
\newenvironment{AMS}{\begin{paragraph}{AMS Subject Classification:}
}
{\end{paragraph}
}

\newtheorem{theorem}{Theorem}[section]
\newtheorem{corollary}[theorem]{Corollary}
\newtheorem{lemma}[theorem]{Lemma}
\newtheorem{proposition}[theorem]{Proposition}
\newtheorem{definition}[theorem]{Definition}

\theoremstyle{remark}
\newtheorem{remark}[theorem]{Remark}
\newtheorem{example}[theorem]{Example}

\usepackage{hyperref}
\hypersetup{
    colorlinks=true,
    linkcolor=blue,
    filecolor=magenta,
    citecolor =magenta,  
    urlcolor=magenta,
    pdftitle={Lifting the Convex Conjugate in Lagrangian Relaxations: A Tractable Approach for Continuous Markov Random Fields}
    }
\usepackage[capitalize,nameinlink]{cleveref}[0.19]

\crefname{section}{section}{sections}
\crefname{subsection}{subsection}{subsections}
\Crefname{section}{Section}{Sections}
\Crefname{subsection}{Subsection}{Subsections}

\Crefname{figure}{Figure}{Figures}

\crefformat{equation}{\textup{#2(#1)#3}}
\crefrangeformat{equation}{\textup{#3(#1)#4--#5(#2)#6}}
\crefmultiformat{equation}{\textup{#2(#1)#3}}{ and \textup{#2(#1)#3}}
{, \textup{#2(#1)#3}}{, and \textup{#2(#1)#3}}
\crefrangemultiformat{equation}{\textup{#3(#1)#4--#5(#2)#6}}%
{ and \textup{#3(#1)#4--#5(#2)#6}}{, \textup{#3(#1)#4--#5(#2)#6}}{, and \textup{#3(#1)#4--#5(#2)#6}}

\Crefformat{equation}{#2Equation~\textup{(#1)}#3}
\Crefrangeformat{equation}{Equations~\textup{#3(#1)#4--#5(#2)#6}}
\Crefmultiformat{equation}{Equations~\textup{#2(#1)#3}}{ and \textup{#2(#1)#3}}
{, \textup{#2(#1)#3}}{, and \textup{#2(#1)#3}}
\Crefrangemultiformat{equation}{Equations~\textup{#3(#1)#4--#5(#2)#6}}%
{ and \textup{#3(#1)#4--#5(#2)#6}}{, \textup{#3(#1)#4--#5(#2)#6}}{, and \textup{#3(#1)#4--#5(#2)#6}}

\title{Lifting the Convex Conjugate in Lagrangian Relaxations: \\A Tractable Approach for Continuous Markov Random Fields}

\begin{document}
	
\author{Hartmut Bauermeister\thanks{equal contribution} \footnotemark[4] \and Emanuel Laude\footnotemark[1] \thanks{KU Leuven, Department of Electrical Engineering (ESAT-STADIUS), Leuven, Belgium (\texttt{emanuel.laude@east.kuleuven.be}). This work was conducted while at Technical University of
  Munich} \and Thomas M\"ollenhoff\thanks{RIKEN Center for Advanced Intelligence Project, Tokyo, Japan (\texttt{thomas.moellenhoff@riken.jp})} \and \newline Michael Moeller\thanks{University of Siegen, Department of Electrical Engineering and Computer Science, Siegen, Germany (\texttt{\{hartmut.bauermeister,michael.moeller\}@uni-siegen.de})} \and Daniel Cremers\thanks{Technical University of
  Munich, Department of Informatics, Munich, Germany (\texttt{cremers@tum.de})}}	

\date{May 12, 2022}

\maketitle
\thispagestyle{empty}

\input{siims_content.tex}

\bibliographystyle{abbrv}
\bibliography{references}
\end{document}

%% file: siims_content.tex
\begin{abstract}
Dual decomposition approaches in nonconvex optimization
may suffer from a duality gap. This poses a challenge when applying them directly to nonconvex problems such as MAP-inference in a Markov random field (MRF) with continuous state spaces.
To eliminate such gaps, this paper considers a reformulation of the original nonconvex task in the space of measures.
This infinite-dimensional reformulation is then approximated by a semi-infinite one, which is obtained via a piecewise polynomial discretization in the dual.
We provide a geometric intuition behind the primal problem induced by the dual discretization and draw connections to optimization over moment spaces. In contrast to existing discretizations which suffer from a grid bias, we show that a piecewise polynomial discretization better preserves the continuous nature of our problem.
Invoking results from optimal transport theory and convex algebraic geometry we reduce the semi-infinite program to a finite one and provide a practical implementation based on semidefinite programming. We show, experimentally and in theory, that the approach successfully reduces the duality gap. To showcase the scalability of our approach, we apply it to the stereo matching problem between two images.
\end{abstract}



\tableofcontents

\section{Introduction and Motivation} \label{sec:intro}
\subsection{Lagrangian Relaxations}
In this paper, our goal is to develop a convex optimization framework \`a la dual decomposition for the MAP-inference problem in a continuous \emph{Markov random field} (MRF). Continuous MRFs are versatile and therefore widely used as a model in image processing, computer vision and machine learning \cite{trinh2008particle,yamaguchi2012continuous,Zach-Kohli-eccv12,chen2016full,domokos2018mrf,weinmann2014total,bergmann2015restoration,weinmann2016mumford,crandall2012sfm,bernard2017combinatorial,peng2011convex,salzmann2013continuous,wang2014efficient}. Due to the continuous nature of the state space, inference amounts to a continuous nonlinear optimization problem. For an undirected graph $\cG=(\cV, \cE)$ it takes the following form:
\begin{align} 
\min_{x \in \Omega^\cV}\; \left\{F(x):=\sum_{u \in \cV} f_u(x_u) + \sum_{uv \in \cE} f_{uv}(x_u, x_v) \right\}.
  \tag{P}
    \label{eq:mrf}
\end{align}
The objective function $F$ is an additive composition of a separable part, $\sum_{u \in \cV} f_u(x_u)$, with unary functions ${f_u : (\Omega \subset \bR^m) \to \bR}$ and a coupling part $\sum_{uv \in \cE} f_{uv}(x_u, x_v)$ with pairwise functions ${f_{uv} : \Omega \times \Omega \to \bR}$ both of which are lower semicontinuous (lsc) and $\emptyset \neq \Omega \subset \bR^m$ compact.

The coupling term $\sum_{uv \in \cE} f_{uv}(x_u, x_v)$ introduces a major challenge for efficient optimization in particular when $\cV$ is large and therefore the problem is high-dimensional. For tractability, one therefore seeks to find a decomposable reformulation of the problem.
In continuous optimization a viable approach is to derive a \emph{Lagrangian relaxation} of the problem: The key idea is to introduce auxiliary variables $x_{uv} \in \bR^m \times \bR^m$ for each edge $uv \in \cE$ and linear constraints $x_{uv} = (x_u, x_v)$. Dualizing the linear constraints with Lagrange multipliers $\lambda_{uv} \in \bR^m \times \bR^m$ one arrives at the Lagrangian dual problem which falls within the regime of convex optimization algorithms that can exploit its separable dual structure.
Examples include subgradient ascent or the \emph{primal-dual hybrid gradient} (PDHG) method \cite{chambolle2011first}.

However, the approach suffers from potentially large duality gaps since problem~\cref{eq:mrf} is nonconvex in general.
Indeed, it can be shown that a direct Lagrangian relaxation of \cref{eq:mrf} leads to the following ``naive'' convexification of the original problem:
\begin{align} \label{eq:classical_lagr_relaxation}
\min_{x \in (\bR^m)^\cV} \sum_{u \in \cV} f_u^{**}(x_u) + \sum_{uv \in \cE} f_{uv}^{**}(x_u, x_v).
\end{align}
With some abuse of notation, here, $f_u$ and $f_{uv}$ attain the value $+\infty$ whenever $x_u \not\in \Omega$ or $(x_u, x_v) \not\in \Omega^2$. Then $f_u^{**}$ are the convex biconjugates which correspond to the largest convex lower semicontinuous under-approximations to $f_u$. Such component-wise convex envelopes can produce inaccurate or even trivial convex under-approximations to the global objective $F$.

\subsection{Contributions and Overview}
To remedy duality gaps, in \cref{sec:local_marginal-polytope}, we consider a reformulation of the problem in terms of a linear one over probability measures and perform the Lagrangian relaxation afterwards.
Instead of a naive formulation $\min_{\mu \in \cP(\Omega^\cV)} \int_{\Omega^\cV} F(x) \dd \mu(x)$ over $\cP(\Omega^\cV)$, which is intractable for large $\cV$ we consider a formulation over $\cP(\Omega)^\cV$ which exploits the partially separable structure of our problem. This is called the \emph{local marginal polytope relaxation} (the former being called \emph{marginal polytope relaxation}).
For univariate $\Omega$ and submodular pairwise terms $f_{uv}$ we prove tightness of the relaxation regardless of nonconvexity of $f_u$.

Since the Lagrange multipliers will be continuous functions, in \cref{sec:dual_discretization}, we consider hierarchies of dual programs \cref{eq:hierarchy} obtained by subspace approximations for which we show in \cref{sec:cone_program} that the duality gap vanishes in the limit for a piecewise polynomial discretization.

In \cref{sec:lifting} we derive and study the primal optimization problem corresponding to the dual discretization. In particular we draw connections to optimization over moment spaces and aspects from variational analysis.
We show that under a certain extremality condition, satisfied by a polynomial discretization with degree at least $2$, the generalized biconjugate of a potentially discontinuous, nonconvex function equals (up to closure) the original function. As a consequence, in contrast to a piecewise linear approximation, our piecewise polynomial approximation conserves the original nonconvex cost including concavities when restricted to moment vectors of Diracs.

In \cref{sec:cone_program}, based on the above developments, we derive a piecewise polynomial discretization of the infinite-dimensional local marginal polytope relaxation and show that for univariate $\Omega$ and submodular pairwise terms $f_{uv}$ the duality gap vanishes in the limit at rate $\mathcal{O}(1 / (K \cdot \sqrt{\deg}))$ where $K$ is the number of pieces and $\deg$ the degree of the pieces.
After discretization the problem amounts to a semi-infinite program which can be transformed into a separable finite-dimensional semidefinite program applying concepts from algebraic geometry, such as nonnegativity certificates of polynomials. This allows us to derive an efficient first-order primal-dual algorithm which, due to the separable problem structure, can be parallelized on a GPU to handle large problems.

In \cref{sec:num} we provide numerical evidence for the strict reduction of the duality-gap and demonstrate the advantages of the nonlinear approximation over existing liftings in the literature. We implement our algorithm on a GPU and apply it to the nonconvex problem of large-scale stereo matching. The results show that increasing the dual subspace improves both, the dual energy, and the quality of a recovered primal solution from a discretized measure.

\subsection{Related Work} \label{sec:related_work}
\paragraph{MRFs}
Fix and Agarwal~\cite{fix2014duality} are the first to propose a dual subspace approximation of the infinite-dimensional local marginal polytope relaxation. This allows for a generalization and unified treatment of convex piecewise linear formulations for MRFs with continuous state spaces proposed in
\cite{Zach-Kohli-eccv12,zach2013dual}. However, for the polynomial case, due to the semi-infinite problem structure, the existence of an implementable algorithm beyond subgradient ascent is posed as an open question by \cite{fix2014duality}. As a consequence, no numerical results are presented. 
Dual subspace approximations for related models in a spatially continuous setting have been considered by \cite{moellenhoff-iccv-2017} 
and further employed for the global optimization of vector- or manifold-valued optimization problems in imaging and vision
\cite{mollenhoff2019lifting,vogt2020connection,vogt2020lifting}. In these papers, piecewise linear dual approximations are used and
the techniques developed in the present work allow one to go beyond the piecewise linear case in a tractable way.
Besides applications in imaging and vision, MRFs with continuous state spaces have for example been used in protein folding~\cite{peng2011convex}.

\paragraph{Lifting}
Lifting to measures for global constrained polynomial optimization is used in \cite{lasserre2001global,lasserre2002semidefinite} to reformulate the infinite-dimensional \emph{linear program} (LP) in terms of a linear objective over the \emph{semidefinite programming} (SDP) characterization of the finite-dimensional space of moments. Its SDP-dual is connected to nonnegativity certificates of polynomials based on sum-of-squares (SOS) and the Positivstellensatz \cite{krivine1964anneaux,stengle1974nullstellensatz,schmudgen1991thek}.
However, for high-dimensional problems the relaxations used in \cite{lasserre2001global,lasserre2002semidefinite} turn out intractable.
As a remedy, \cite{waki2006sums,weisser2018sparse} consider sparse SOS-approaches and in particular sparse versions of the Positivstellensatz \cite{weisser2018sparse}. This is closely related to the local marginal polytope relaxation considered in this work.
As a key difference to \cite{lasserre2001global,lasserre2002semidefinite,waki2006sums,weisser2018sparse} we consider possibly nonpolynomial objective functions which results in a generally nonlinear formulation over the space of moments. 
In contrast to \cite{weisser2018sparse}, we apply optimal transport duality theory to further reduce the formulation using a sum-of-squares characterization of Lipschitz continuity of piecewise polynomials.
Similar techniques have been considered recently in \cite{latorre2020lipschitz,chen2020semialgebraic} to estimate Lipschitz constants in neural networks.

\paragraph{Lagrangian relaxation and decoupling by lower relaxations}
A component-wise lifting in problems with a partially separable structure results in a generalized dual decomposition approach. Dual decomposition and Lagrangian relaxation are general principles in optimization and appear as a useful tool across many disciplines, see \cite[Ch.~5--6]{Ber99}, \cite{lemarechal2001lagrangian} and the references therein for an overview. Traditionally, decomposition methods are applied directly in the nonconvex setting without lifting, unlike the generalized scheme we present in this paper which is based on a lifted reformulation. Without lifting this typically results in a component-wise convex lower envelope. More recently, \cite{simoes2021lasry} consider a homotopy method based on component-wise Lasry--Lions envelopes which specializes to a component-wise convex envelope in a certain limit case.

\paragraph{Generalized conjugacy and duality}
Our notion of a lifted convex conjugate is closely related to generalized conjugate functions originally due to \cite{moreau1970inf}. It was utilized by \cite{balder1977extension} to study nonconvex dualities,
expanding upon the work of \cite{rockafellar1974augmented} on augmented Lagrangians for nonconvex optimization. The specific case of quadratic conjugate functions was developed by \cite{poliquin1990subgradient} as an analytic tool in a seminal proof that the proximal subgradient map of a lower semicontinuous extended real-valued function is monotone if and only if the function is convex. 
Typically, generalized conjugate functions appear in the context of eliminating duality gaps \cite{bui2021zero,rockafellar1974augmented} in nonconvex and nonsmooth optimization. In that sense our work shares the goal with the aforementioned works. We offer a somewhat complementary approach through a primal-type lifting view which allows us to establish a connection to optimization in spaces of measures.

\subsection{Notation}
For a compact nonempty set $X \subset \bR^m$, denote by $\ccP(X)$ the space of Borel probability measures on $X$ and by $\ccM(X)$ the space of Radon measures on $X$. The convex cone of nonnegative Radon measures is denoted by $\ccM_+(X)$. For a Radon measure $\mu \in \ccM(X_1)$ and a measurable mapping $T:X_1 \to X_2$, $T \sharp \mu$ is the pushforward of $\mu$ w.r.t. $T$ defined by: $(T \sharp \mu)(A) = \mu(T^{-1}(A))$ for all $A\subset X_2$ in the corresponding $\sigma$-algebra.
Furthermore let $\cC(X)$ be the space of continuous functions on $X$. We will write $\langle \mu, f \rangle=\int_{X} f(x) \dd \mu(x)$, for $f \in \cC(X)$ and $\mu \in \ccM(X)$. For a set $C\subset \bR^m$ we denote by $\ind_C:\bR^m \to \exR = \bR \cup \{\pm \infty\}$ the indicator function with $\ind_C(x) =0$ if $x \in C$ and $\ind_C(x) =\infty$ if $x\notin C$ and by $\sigma_C(x) = \sup_{y\in C} \langle x,y \rangle$ the support function of $C$ at $x \in \bR^m$. We denote by $C^*=\{y \in \bR^m : \langle y, x \rangle \geq 0, \forall x \in C \}$ the dual cone of $C$. These notions are defined analogously for the topologically paired spaces $\ccM(X)$ and $\cC(X)$.
The convex hull $\con C$ of a set $C\subset \bR^m$ is the smallest convex set that contains $C$. Equivalently, $\con C$ is the set of all finite convex combinations of points in $C$. With some abuse of notation we write $\con f$ for a function $f$ to denote the largest convex function below $f$. Let $\delta_x$ denote the Dirac measure centered at $x \in \bR^m$.
We write lsc for lower semicontinuous and denote the epigraphical closure of some function $f$ by $\cl f$. Let $\llbracket \cdot \rrbracket$ denote the Iverson bracket, where $ \llbracket P \rrbracket=1$
if $P$ is true and $ \llbracket P \rrbracket =0$, otherwise. For some extended real-valued function $f:\bR^m \to \exR$ let $f^*(y)=\sup_{x \in \bR^m} \langle x, y\rangle -f(x)$ denote the Fenchel conjugate of $f$ at $y$ and $f^{**} = (f^*)^*$ the Fenchel biconjugate. 

\section{The Local Marginal Polytope Relaxation} \label{sec:local_marginal-polytope}
To overcome the duality gap in \cref{eq:classical_lagr_relaxation} for nonconvex $f_u$, we propose to reformulate the original nonlinear problem in terms of an infinite-dimensional linear program over the space of Radon measures and apply the Lagrangian relaxation afterwards.

The following key lemma reveals, that every minimization problem can be equivalently formulated in terms of an infinite-dimensional linear program. 
\begin{lemma} \label{lem:lp_exactness}
Let $f:X \to \bR$ be lsc with $\emptyset \neq X \subset \bR^m$ compact. Then we have 
\begin{align}
\min_{x \in X} f(x) = \min_{\mu \in \cP(X)} ~\langle \mu, f \rangle,
\end{align}
and $x^* \in \argmin_{x \in X} f(x)$ is a solution to~$\min_{x \in X} f(x)$ if and only if $\delta_{x^*}$ is a minimizer of $\min_{\mu \in \cP(X)} \langle \mu, f \rangle$.
\end{lemma}
\begin{proof}
The result follows immediately via compactness of $X$ and the properties of probability measures: By compactness of $X$ and since $f$ is lsc relative to $X$ we know $x^*$ exists.
Let $f_{\min} = f(x^*)$. By the properties of the Lebesgue integral we have: 
$$
\int_X f(x) \dd \mu(x) \geq \int_X f_{\min} \dd \mu(x) = f_{\min}\int_X 1 \dd \mu(x) =  \min_{x \in X} f(x),
$$
for all $\mu \in \cP(X)$ and $\min_{x \in \bR^m} f(x) = \langle \delta_{x^*}, f \rangle = f(x^*)$.
\end{proof}

Let $\pi_u : \Omega^\cV \to \Omega$, $\pi_{uv} : \Omega^\cE \to \Omega^2$ denote the canonical projections onto the $u\textsuperscript{th}$ resp. $u\textsuperscript{th}$ and $v\textsuperscript{th}$ components.
Then, in our case, applying the reformulation from \cref{lem:lp_exactness} directly to the cost function $F$ with $X=\Omega^\cV$ we obtain by linearity:
\begin{align*}
\min_{\mu \in\cP(\Omega^\cV)} \;\langle \mu, F \rangle &= \min_{\mu \in\cP(\Omega^\cV)} \;\left\langle \mu, \sum_{u \in \cV} f_u \circ \pi_u + \sum_{uv \in \cE}  f_{uv} \circ \pi_{uv} \right\rangle \\
&= \min_{\mu \in\cP(\Omega^\cV)} \;\sum_{u \in \cV} \langle\pi_u \sharp \mu, f_u \rangle + \sum_{uv \in \cE} \langle \pi_{uv} \sharp \mu, f_{uv} \rangle.
\end{align*}
This relaxation is known as the \emph{full marginal polytope relaxation} which is, however, intractable if $\cV$ is large as one minimizes over probability measures on the product space $\cP(\Omega^\cV)$.
Instead, we consider the following linear programming relaxation of \cref{eq:mrf} which is also referred to as the \emph{local marginal polytope relaxation} \cite{peng2011convex,fix2014duality,wald2014tightness,ruozzi2015exactness} which is more tractable as the optimization variable lies in the product space of probability measures $\cP(\Omega)^\cV$:
\begin{align}
\inf_{\mu \in \ccP(\Omega)^\cV} \left\{\mathcal{F}(\mu) := \sum_{u \in \cV} \langle \mu_u, f_u \rangle + \sum_{uv \in \cE} \ot_{f_{uv}}(\mu_u, \mu_v) \right\}.
\tag{R-P}
\label{eq:opt_relax}
\end{align}
Here, $\ot_{f_{uv}}$ denotes the optimal transportation \cite{kantorovich1960} with marginals $(\mu_u, \mu_v)$ and cost $f_{uv}$ defined by
\begin{align}	
\label{eq:kant_optimal_transport}
\ot_{f_{uv}}(\mu_u, \mu_v) = \inf_{\mu_{uv} \in \Pi(\mu_u, \mu_v)} \langle \mu_{uv}, f_{uv}\rangle.
\end{align}
The constraint set $\Pi(\mu_u, \mu_v)$ consists of all Borel probability measures on $\Omega^2$ with specified marginals $\mu_u$ and $\mu_v$: 
\begin{align} \label{eq:marginalization_constraints}
\Pi(\mu_{u}, \mu_v)=\left\{ \mu_{uv} \in \cP(\Omega^2) : \pi_u\sharp \mu_{uv} = \mu_u, \; \pi_v\sharp \mu_{uv} = \mu_v \right\},
\end{align}
where $\pi_u : \Omega \times \Omega \to \Omega$ corresponds to the canonical projection onto the $u\textsuperscript{th}$ component.
Note that for finite state-spaces (i.e., $|\Omega| < \infty$), the set $\cP(\Omega)$ can be identified with the standard $(|\Omega| - 1)$-dimensional probability simplex and $\Pi(\mu_u, \mu_v)$ with the set of nonnegative $|\Omega| \times |\Omega|$ matrices whose rows and columns sum up to $\mu_v$ and $\mu_u$.
In that case, the linear program given in \cref{eq:opt_relax} is equivalent to the well-known finite-dimensional local marginal polytope relaxation for MRFs, which is for example studied in \cite{werner2007linear}. We further remark that the case of finite $\Omega$ has been extensively studied in the literature, see, e.g., \cite{kappes2013comparative,wainwright2008graphical} for recent overviews.
	
For the more challenging setting of continuous state-spaces, a major difficulty stems from the fact that the linear programming relaxation \cref{eq:opt_relax} is an infinite-dimensional optimization problem posed in the space of Borel probability measures. Perhaps due to this difficulty, discrete MRF approaches are still routinely applied despite the continuous nature of $\Omega$. This is typically done by considering a finite sample approximation of $\Omega$, so that the infinite-dimensional linear program reduces to a finite-dimensional one. This, however, may lead to discretization errors and comes with an exponential complexity in the
dimension of $\Omega$. 

Due to the fact that $\Pi(\delta_x, \delta_{x'}) = \{ \delta_{(x, x')} \}$ one sees that restricting $\mu_u$ (and therefore $\mu_{uv}$) to be Dirac probability measures, the formulation \cref{eq:opt_relax} reduces to the original problem \cref{eq:mrf}. As one instead considers the larger convex set of all probability measures, it is a \emph{relaxation} which lower bounds \cref{eq:mrf}, i.e., we have the following important relation:
\begin{equation}
\cref{eq:opt_relax} \leq \cref{eq:mrf}.
\end{equation}
When considering a relaxation, the immediate question arises whether this lower bound is attained, i.e., the relaxation is tight and therefore the above inequality holds with equality.
For finite and ordered $\Omega$ the situation is well-understood, see \cite{werner2007linear}. 
For continuous $\Omega\subset \bR^m$ a total order is possible if $m=1$, i.e., $\Omega$ is an interval. Indeed, if in addition $F$ is submodular, see, e.g., \cite[Sec.~2.1]{bach2019submodular}, tightness of the local marginal polytope relaxation can be derived from \cite[Thm.~2]{bach2019submodular} regardless of nonconvexity of $f_u$:
\begin{proposition} \label{prop:tightness_local_marginal}
Let $\emptyset \neq \Omega \subset \bR$ be compact and $f_u:\Omega \to \bR$, $f_{uv}:\Omega \times \Omega \to \bR$ be continuous with $f_u$ possibly nonconvex. If $f_{uv}$ is submodular for all $uv \in \cE$, $F$ is submodular as well and the relaxation is tight, i.e.,
\begin{align} \label{eq:tightness}
\cref{eq:opt_relax} = \cref{eq:mrf}.
\end{align}
\end{proposition}
\begin{proof}
First, we show that $F$ is submodular. The unaries $f_u$ are submodular as functions of a single variable are submodular and hence as a nonnegative sum of submodular functions, $F$ is submodular itself.

For $\mu \in \ccP(\Omega)$ define the cumulative distribution function $\mathcal F_\mu \colon \Omega \to [0, 1]$ as
\begin{align}
\mathcal F_\mu (x) = \mu(\{y \in \Omega \colon y \geq x\})
\end{align}
and the ``inverse'' cumulative distribution function $\mathcal F_\mu^{-1} \colon [0, 1] \to \Omega$ as
\begin{align}
\mathcal F_\mu^{-1} (t) = \sup \{x \in \Omega \colon \mathcal F_\mu (x) \geq t\}.
\end{align}
Submodularity of $F$ now implies by \cite[Thm.~2]{bach2019submodular}
\begin{align}
\cref{eq:mrf} &= \inf_{\mu \in \ccP(\Omega)^\cV} \int_0^1 \sum_{u \in \cV} f_u(\mathcal F_{\mu_u}^{-1}(t)) + \sum_{uv \in \cE} f_{uv}(\mathcal F_{\mu_u}^{-1}(t), \mathcal F_{\mu_v}^{-1}(t)) \dd t \\
&= \inf_{\mu \in \ccP(\Omega)^\cV} \sum_{u \in \cV} \int_0^1 f_u(\mathcal F_{\mu_u}^{-1}(t)) \dd t + \sum_{uv \in \cE} \int_0^1 f_{uv}(\mathcal F_{\mu_u}^{-1}(t), \mathcal F_{\mu_v}^{-1}(t)) \dd t.
\end{align}
Using submodularity of $f_u$ and $f_{uv}$, and applying \cite[Prop.~2]{bach2019submodular} and \cite[Prop.~4]{bach2019submodular} we get
\begin{align}
\int_0^1 f_u(\mathcal F_{\mu_u}^{-1}(t)) \dd t &= \langle \mu_u, f_u \rangle, \\
\int_0^1 f_{uv}(\mathcal F_{\mu_u}^{-1}(t), \mathcal F_{\mu_v}^{-1}(t)) \dd t &= \ot_{f_{uv}}(\mu_u, \mu_v)
\end{align}
by the definition of the optimal transportation in \cref{eq:kant_optimal_transport}. Hence it holds $\cref{eq:opt_relax} = \cref{eq:mrf}$ and the relaxation $\cref{eq:opt_relax}$ is tight.
\end{proof}
\begin{remark}
A solution to the local marginal polytope relaxation $\cref{eq:opt_relax}$ even allows for the reconstruction of a globally optimal solution to the unrelaxed original problem $\cref{eq:mrf}$: More precisely, $\mu \in \ccP(\Omega)^\cV$ is a minimizer of $\cref{eq:opt_relax}$ if and only if $u \mapsto \mathcal F_{\mu_u}^{-1} (t)$ is a minimizer of $\cref{eq:mrf}$ for almost all $t \in [0, 1]$, see \cite[Thm.~2]{bach2019submodular}.
\end{remark}

For $f_{uv} (x, y) = g(x - y)$ and $g$ convex, $f_{uv}$ is submodular, see \cite[Sec.~2.2]{bach2019submodular}. Especially the total variation-like couplings considered in the experimental sections are therefore submodular and thus the local marginal polytope relaxation is tight.

\section{Dual Discretization for the Continuous MRF} \label{sec:dual_discretization}
\subsection{A Reduced Dual Formulation}
A Lagrangian relaxation to the infinite-dimensional problem~\cref{eq:opt_relax} is obtained by dualizing the marginalization constraints~\cref{eq:marginalization_constraints} for each $e \in \cE$ with Lagrange multipliers $\lambda_e \in \cC(\Omega)^2$. This is equivalent to a substitution of the optimal transportation with its dual formulation.
Adopting the approach of \cite{fix2014duality} a finite-dimensional problem is obtained by approximating the Lagrange multipliers in terms of finite linear combinations of certain basis functions. These approximations are chosen in such a way that the classical Lagrangian relaxation and the discrete approach, described above, are special cases of the considered framework.

In contrast to previous approaches \cite{fix2014duality}, we restrict ourselves to metric pairwise terms $f_{uv}(x,y)=d(x,y)$ which eventually leads to a different dual formulation and turns out more tractable: Then, the optimal transportation $\ot_{f_{uv}}(\mu_u, \mu_v)$ in Problem~\cref{eq:opt_relax} is the Wasserstein-$1$ distance $W_1^d(\mu_u, \mu_v)$ induced by the metric $d$ between $\mu_u$ and $\mu_v$. Furthermore we assume that $f_u:\Omega \to \bR$ is lsc and $\Omega\subset \bR^m$ is a compact nonempty set. In particular this implies that $f_u$ is proper and bounded from below.

Thanks to optimal transport duality theory \cite{Vil08} we are therefore able to obtain a more compact dual formulation, which is instrumental to derive a tractable implementation for a piecewise polynomial discretization later on: More precisely, we substitute $\ot_{f_{uv}}(\mu_u, \mu_v)=W_1^d(\mu_u, \mu_v)$ in Problem~\cref{eq:opt_relax} with its dual formulation 
\begin{align} \label{eq:wasserstein1_dual}
W_1^d(\mu_u, \mu_v) = \sup_{\lambda \in \Lip_d(\Omega)} \int \lambda(x) \dd(\mu_u- \mu_v)(x),
\end{align}
where $\lambda$ is $1$-Lipschitz with respect to the metric $d$, i.e., 
\begin{align}
\lambda \in \Lip_{d}(\Omega) = \{ \lambda : \Omega \to \bR : |\lambda(x) - \lambda(y)| \leq d(x, y) \}.
\end{align}
In contrast to the general dual formulation of optimal transport, which involves two Lagrange multipliers per edge, interacting via the constraint set, the Wasserstein-$1$ dual involves only a single dual variable for each edge, that satisfies a Lipschitz constraint. We show in \cref{sec:cone_program} that even though this formulation still involves infinitely many constraints, thanks to convex algebraic geometry~\cite{blekherman2012semidefinite}, there exists a tractable finite representation of this constraint set in terms of SDP.

For notational convenience, we assign an arbitrary orientation to the edges in $\mathcal{G}$. After introducing a graph divergence operator $\Div : \cC(\Omega)^\cE \to \cC(\Omega)^\cV$ defined by:
\begin{equation}
-(\Div  \lambda)_u =  \sum_{v : (u,v) \in \cE}\lambda_{(u,v)} - \sum_{v : (v,u) \in \cE} \lambda_{(v,u)},
\end{equation}
an interchange of $\min$ and $\sup$ yields the following reduced dual problem, which is the starting point for further discussion and a tractable implementation:
\begin{equation*} 
\begin{aligned}
&\sup_{\lambda\in \cC(\Omega)^\cE} \left\{D(\lambda):=-\sum_{u \in \cV}  \sigma_{\cP(\Omega)}(-f_u +(\Div \lambda)_u) - \sum_{e \in \cE} \iota_{\mathcal{K}}(\lambda_e) \right\},
\end{aligned}
\tag{R-D}
\label{eq:dualMPTsmall}
\end{equation*}
where $\mathcal{K}=\Lip_{d}(\Omega)$ and $\sigma_{\ccP(\Omega)}(-f_u + (\Div \lambda)_u) = \sup_{\mu \in \ccP(\Omega)} \langle \mu, -f_u + (\Div \lambda)_u \rangle$ is the support function of $\ccP(\Omega)$ at $-f_u + (\Div \lambda)_u$, which, thanks to \cref{lem:lp_exactness}, can be rewritten:
\begin{align} \label{eq:inner_minimization}
-\sigma_{\ccP(\Omega)}(-f_u + (\Div \lambda)_u) = \min_{x \in \Omega} (f_u - (\Div \lambda)_u)(x).
\end{align}
Under mild assumptions we have strong duality:
\begin{proposition}
\label{prop:map_inference_mrf}
Let $\emptyset\neq \Omega \subset \bR^m$ be compact, let $f_u : \Omega \to \bR$ be lsc, and $f_{uv}=d$, where $d : \Omega^2 \to \bR$ is lsc and a metric. Then, the following strong duality holds:
\begin{align} \label{eq:strong_duality_reduced}
\cref{eq:opt_relax} = \cref{eq:dualMPTsmall},
\end{align}
and a maximizer of \cref{eq:dualMPTsmall} exists.
\end{proposition}
\begin{proof}
Define $f:\cM(\Omega)^\cV \to \exR$
$$
f(\mu):=\sum_{u \in \cV} \langle \mu_u, f_u \rangle + \iota_{\cP(\Omega)}(\mu_u),
$$
which is convex, proper and lsc due to \cite[Lem.~1.1.3]{San15}.
Likewise, define the functional $g:\cM(\Omega)^\cE \to \exR$
$$
g(\nu):=\sum_{e \in \cE} \sup_{\lambda_e \in \cK} \langle \nu_e, \lambda_e \rangle = \sum_{e \in \cE} \sigma_{\mathcal{K}}(\nu_e),
$$
which is proper convex lsc, as it is a pointwise supremum over linear functionals.
Define $\nabla:\cM(\Omega)^\cV \to \cM(\Omega)^\cE$
$$
(\nabla \mu)_{(u,v)} = \mu_u -\mu_v,
$$
which is bounded and linear.
Then we compute the convex conjugates $f^*:\cC(\Omega)^\cV \to \exR$ to
$$
f^*(\theta)= \sum_{u \in \cV} \sigma_{\cP(\Omega)}(\theta_u - f_u),
$$
and the functional $g^*:\cC(\Omega)^\cE \to \exR$
$$
g^*(\lambda)= \sum_{e \in \cE} \iota_{\mathcal{K}}(\lambda_e),
$$
and $\nabla^*:\cC(\Omega)^\cE \to \cC(\Omega)^\cV = -\Div$.

Choose $x \in \Omega$ and define $\mu:=(\delta_x)^\cV$. Then $f(\mu) = \sum_{u \in \cV} f_u(x) < \infty$. In addition we have $\langle \delta_x -\delta_x, \lambda_{e} \rangle =0$ for all $\lambda_e \in \cC(\Omega)$ and therefore $g(\nabla \mu) = 0$. Now consider a weakly$^*$-convergent sequence $\cM(\Omega)^\cE\ni\nu^t \overset{\ast}{\rightharpoonup} \nabla \mu$. This means for all $e\in \cE$ and $\lambda_e \in \cC(\Omega)$ we have $\langle \nu_e^t, \lambda_e \rangle \to \langle (\nabla \mu)_e, \lambda_e \rangle =0$.
In particular this implies that $g(\nu^t) \to 0$, and therefore $g$ is continuous at $\nabla \mu$.
Then we can invoke the Fenchel--Rockafellar duality~Theorem~\cite[Thm.~1]{Ro67} and obtain that
$\cref{eq:opt_relax} = \cref{eq:dualMPTsmall}$ and a maximizer of \cref{eq:dualMPTsmall} exists.
\end{proof}

\subsection{Discretization of the Reduced Dual Formulation for the Metric MRF}
The next step in our strategy to obtain a tractable formulation is to restrict $\lambda_{(u,v)} \in \cC(\Omega)$ in Problem~\cref{eq:dualMPTsmall} to a subspace $\Lambda = \langle \varphi_0, \ldots, \varphi_n \rangle$ spanned by basis functions $\varphi_k \in \cC(\Omega)$ and instead consider $\sup_{\lambda \in \Lambda^\cE} D(\lambda)$.
In the situation of \cref{prop:tightness_local_marginal}, for any hierarchy of increasingly expressive dual subspaces $\Lambda_1 \subset \Lambda_2 \subset \cdots \subset \cC(\Omega)$, where $\Lambda_1$ is the space of affine functions, the induced hierarchy of dual problems
\begin{align} \label{eq:hierarchy}
\cref{eq:classical_lagr_relaxation} &= \max_{\lambda \in (\Lambda_1)^\cE}\; D(\lambda) \leq \max_{\lambda \in (\Lambda_2)^\cE}\; D(\lambda) \leq \cdots \leq \max_{\lambda \in \cC(\Omega)^\cE} \;D(\lambda) \overset{\cref{eq:strong_duality_reduced}}{=}\cref{eq:opt_relax} \overset{\cref{eq:tightness}}{=} \cref{eq:mrf},
\end{align}
leads to a reduction of the duality gap where the first equality holds due to the explanation that will be given in \cref{sec:lifting_primal_problem}. In general, properness of the inclusions in the hierarchy of the dual subspaces need not imply strict inequalities in \cref{eq:hierarchy}.
However, for a piecewise polynomial hierarchy with increasing degrees and/or number of pieces we will show in \cref{prop:gapconvergence} that the duality gap $\cref{eq:mrf} - \max_{\lambda \in \Lambda^\cE}\; D(\lambda)$ eventually vanishes with rate $\mathcal{O}(1 / (K \cdot \sqrt{\deg}))$ as the number of pieces $K$ and/or the degree $\deg$ goes to $\infty$.

In the context of dual discretization it is crucial to discuss the duality between finite-dimensional subspaces of $\cC(\Omega)$ and certain equivalence classes of measures and in particular moment spaces. This will be particularly important in \cref{sec:lifting} where we derive and study the primal problem corresponding to the discretized dual problem:

Due to Riesz' theorem the dual space of continuous functions on a compact set $\cC(\Omega)$ is given by the space of Radon measures $\ccM(\Omega)$. Note that conversely for any $\lambda \in \Lambda$ we obtain a linear functional on $\ccM(\Omega)$ by mapping $\mu \in \ccM(\Omega)$ to $\mu \mapsto \langle \mu, \lambda \rangle = \int \lambda(x) \dd \mu(x)$.
Although any $\mu \in \ccM(\Omega)$ maps to a linear functional on $\Lambda$, this mapping is not injective. In particular there exist distinct measures $\mu \neq \nu$ such that $\langle \mu, \lambda \rangle = \langle \nu, \lambda \rangle$ for every $\lambda \in \Lambda$, i.e.,\ $\mu$ and $\nu$ induce the same linear functional on $\Lambda$. Instead, invoking \cite[Thm.~4.9]{rudin1991functional} the dual space $\Lambda^*$ can be related to a quotient space of $\ccM(\Omega)$ as follows:
$\Lambda^*$ is isometrically isomorphic to a quotient space on $\ccM(\Omega)$ via
\begin{align}
\Lambda^* \cong \ccM(\Omega) / \Lambda^0 := \left\{ \mu + \Lambda^0 \colon \mu \in \ccM(\Omega) \right\},
\end{align}
where $\Lambda^0 := \left\{ \mu \in \ccM \colon \langle \mu, \lambda \rangle = 0 \text{ for all $\lambda \in \Lambda$} \right\}$ is the annihilator of $\Lambda$. Any $\mu \in \ccM(\Omega)$ generates an equivalence class in $\ccM / \Lambda^0$ denoted by $[\mu]$ and the corresponding element in $\Lambda^*$ is denoted by $[\mu]_{\Lambda}$.

Further restricting to the cone of functions in $\Lambda$ that are nonnegative on $\Omega$, the corresponding dual cone has a nice interpretation: It can be identified with a certain space of moments: 
Given a measure $\mu \in \ccM_+(\Omega)$, for a particular choice of basis functions $\{\varphi_0, \varphi_1, \ldots, \varphi_n\}$, we refer to $\langle \mu, \varphi_k \rangle=\int \varphi_k(x) \dd \mu(x)$ as the $k$\textsuperscript{th} \emph{moment} of $\mu$ and likewise the moment space is
\begin{align} \label{eq:nonnegative_moments}
(\bbM)_+ = \left\{ y \in \bR^{n+1} : \exists \mu \in \cM_+(\Omega) : y_k = \langle \mu, \varphi_k \rangle \right\} \subset \bR^{n+1}.
\end{align}
This is discussed in \cref{sec:cone_program} in the context of implementation. The above result then shows that the moment space $(\bbM)_+$ identifies the space of equivalence classes of nonnegative measures, where two nonnegative measures are considered equivalent if they have the same moments.

For the remainder of this paper, for simplicity, we fix the basis $\{\varphi_0, \varphi_1, \ldots, \varphi_n\}$ and identify the dual variable $\lambda_{(u,v)}$ with its vector of coefficients $p_{(u,v)} \in \bR^{n+1}$.

In view of the Wasserstein-$1$ dual in \cref{eq:wasserstein1_dual} constant components in $\lambda_{(u,v)} \in \Lambda$ cancel out and therefore, assuming $\varphi_0$ is the constant function, we can choose $p_0=0$. As a consequence we can represent $(\Div \lambda)_u$ in terms of its coefficients $(\Div_\Lambda p)_u$ w.r.t. the basis functions $\{\varphi_1, \ldots, \varphi_n\}$ that span the subspace $\Lambda \subset \cC(\Omega)$, for an appropriately chosen linear mapping $\Div_\Lambda :(\bR^n)^\cE \to (\bR^n)^\cV$.

We identify the basis functions $\varphi_k$ as the component functions of the mapping $\varphi : \Omega \to \bR^n$. Then, substituting the representation $\lambda_{(u,v)} = \langle p_{(u,v)}, \varphi(\cdot) \rangle$ in the support function we obtain thanks to \cref{eq:inner_minimization}
the following discretized dual problem:
\begin{equation*}
\sup_{p \in (\bR^n)^\cE} \left\{D_\Lambda(p) := \sum_{u \in \cV} ~ \min_{x \in \Omega} \; -\big\langle (\Div_\Lambda p)_u, \varphi(x) \big \rangle + f_u(x) - \sum_{e \in \cE}\iota_{\mathcal{K}_\Lambda}(p_e)\right\} = \sup_{\lambda \in \Lambda^\cE} D(\lambda),
\tag{dR-D}
\label{eq:discrete_dual}
\end{equation*}
where the constraint set $\mathcal{K}_{\Lambda}$ is the set of coefficients of Lipschitz functions in $\Lambda$:
\begin{align}
\mathcal{K}_{\Lambda} := \{ p \in \bR^n : \langle p, \varphi(\cdot) \rangle \in \Lip_{d}(\Omega) \}.
\end{align}

\section{Lifted Convex Conjugates: A Primal View} \label{sec:lifting}
\subsection{The Discretized Primal Problem} \label{sec:lifting_primal_problem}
A central goal of this section is to study the primal problem induced by the dual discretization. This allows us to show that discretizations which obey a certain extremality property conserve the original cost when restricting to discretized Diracs, see \cref{thm:general_binconj}.

In order to derive the discretized primal problem note that for any $f_u:\Omega \to \bR$ the expression $\sup_{x \in \Omega} \langle (\Div_\Lambda p)_u, \varphi(x) \rangle - f_u(x)$ in \cref{eq:discrete_dual} (up to the presence of $\varphi(x)$) resembles the form of a convex conjugate.
Indeed, exploiting the notion of an extended real-valued function this is the convex conjugate of a lifted version $f_\Lambda :\bR^n \to \exR$ of $f:\Omega \to \bR$ which is defined by
\begin{align} \label{eq:lifted_cost}
\lifted{f}{\Lambda} (y) = \begin{cases}
f(x), & \text{if $y = \varphi(x)$ for some $x \in \Omega$,}\\
\infty, & \text{otherwise},
\end{cases}
\end{align}
and gives rise to our notion of a \emph{lifted convex conjugate} which is studied in depth in the course of this section.

Since the ``lifting'' to a higher dimensional space happens through the application of the mapping $\varphi$, whose component functions are the elements of the basis, it is intuitive to refer to $\varphi$ as the lifting aka feature map in this context. Likewise we refer to the curve described by $x \mapsto \varphi(x)$ as the moment curve.

With this construction at hand and via an application of Fenchel--Rockafellar duality we are ready to state the primal problem of the discretized dual problem in terms of the lifted biconjugates $(f_u)_\Lambda^{**}$:
\begin{align}
\min_{y \in (\bR^n)^\cV} ~\left\{\mathcal{F}_\Lambda(y):=\sum_{u \in \cV} (f_u)_\Lambda^{**}(y_u) + \sum_{e \in \cE} \sigma_{\mathcal{K}_{\Lambda}}\big((\nabla_\Lambda\; y)_e \big) \right\},
\tag{dR-P}\label{eq:discrete_primal} 
\end{align}
where $\nabla_\Lambda:=-\Div_\Lambda^*$ is the negative adjoint of the graph divergence, the graph gradient operator. A comparison of \cref{eq:discrete_primal} with the convex relaxation~\cref{eq:classical_lagr_relaxation} shows the effect of the dual discretization, where the classical convex biconjugates are replaced with biconjugates of the ``lifted'' functions $(f_u)_\Lambda$. Indeed, without lifting, i.e., $\varphi(x)=x$, we recover the classical biconjugates and therefore the convex relaxation~\cref{eq:classical_lagr_relaxation}. In particular this explains the first equality in the hierarchy~\cref{eq:hierarchy}.

\subsection{Lifting and Moments}
\begin{figure}[t!]
\centering
\begin{subfigure}[b]{0.27\linewidth}
\centering
\includegraphics[width=\linewidth]{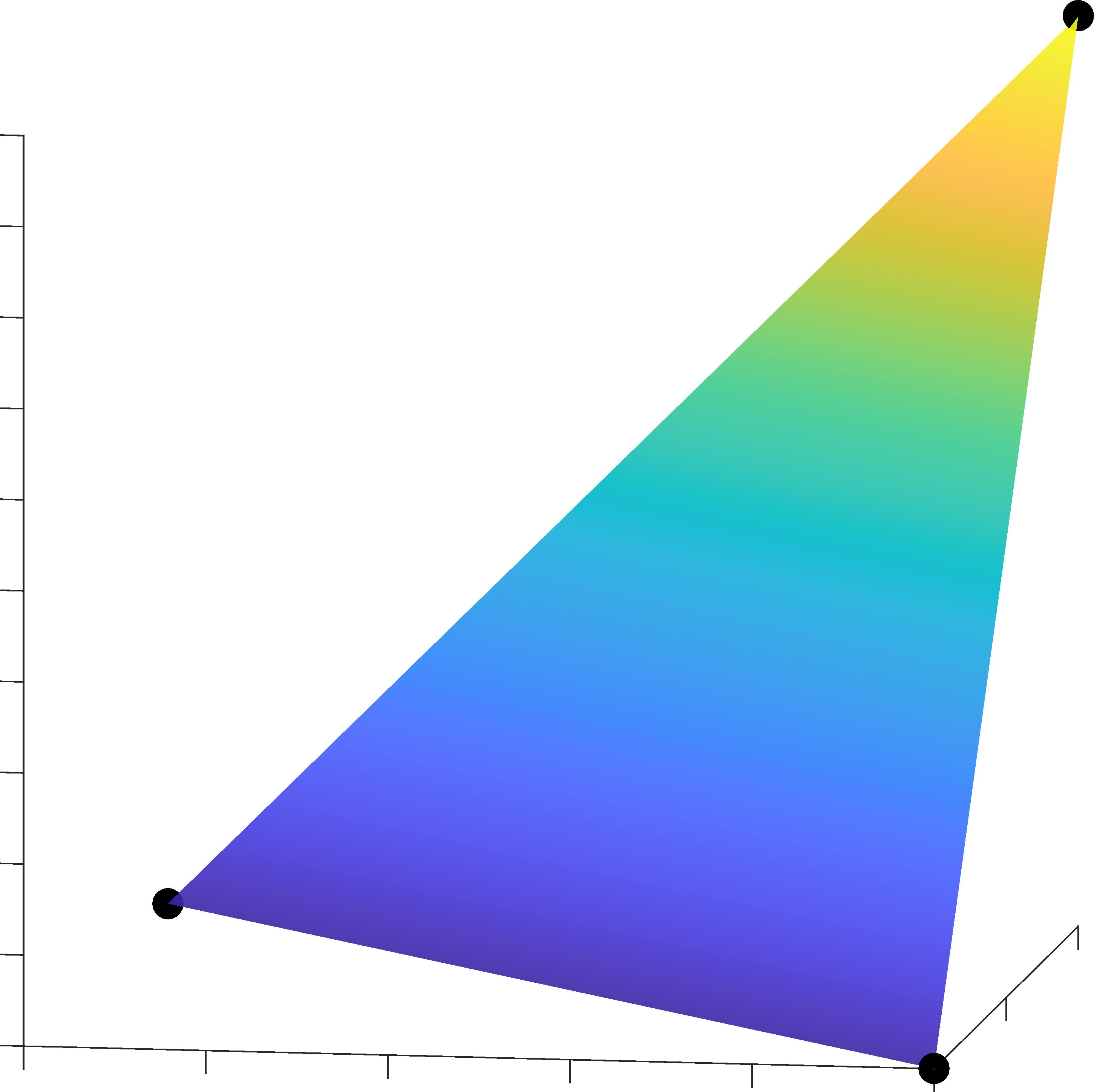}
\caption{Standard Lifting}
\end{subfigure}
\hfill
\begin{subfigure}[b]{0.27\linewidth}
\centering
\includegraphics[width=\linewidth]{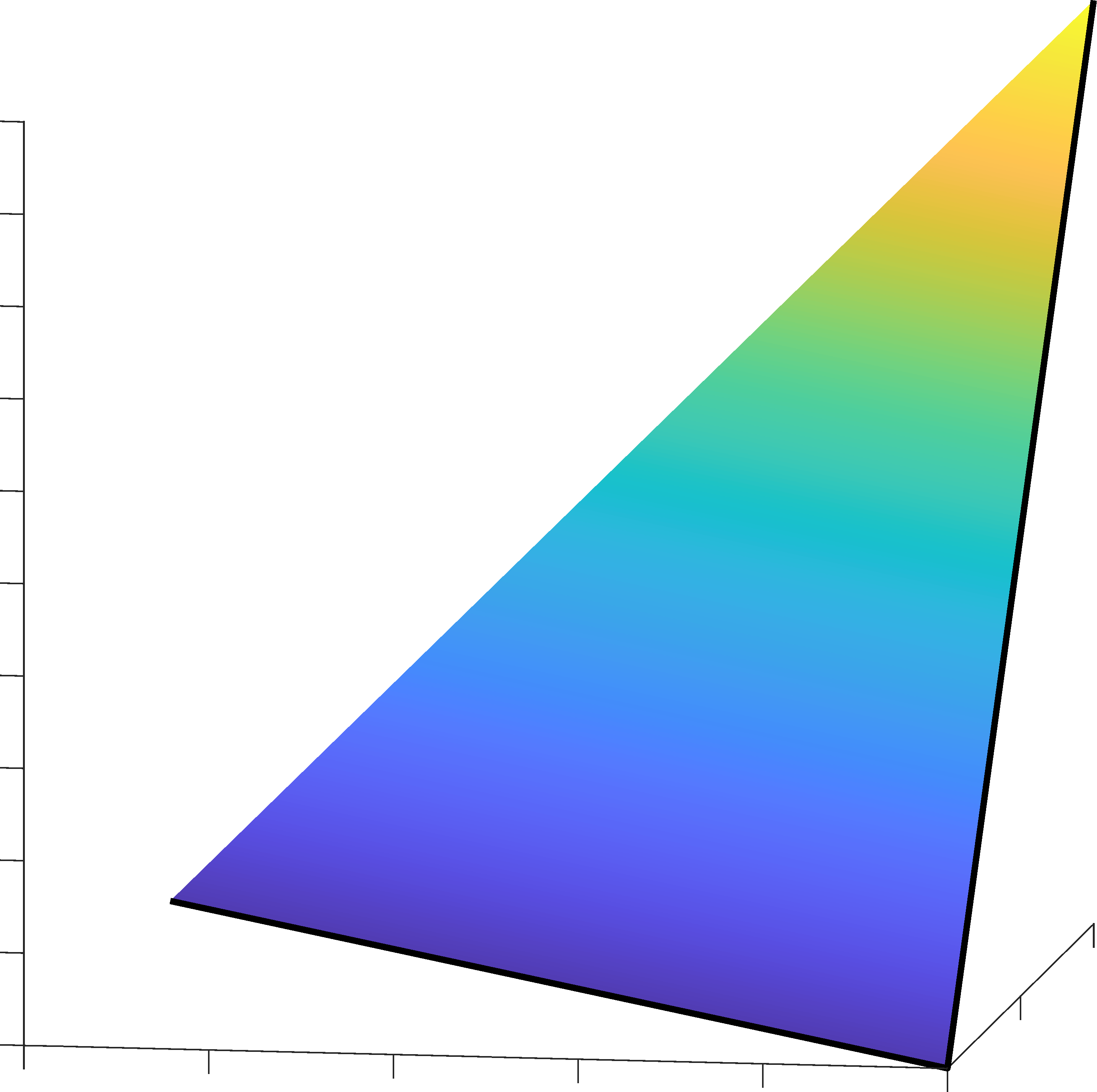}
\caption{Sublabel Lifting \cite{laude16eccv}}			
\end{subfigure}
\hfill
\begin{subfigure}[b]{0.27\linewidth}
\centering
\includegraphics[width=\linewidth]{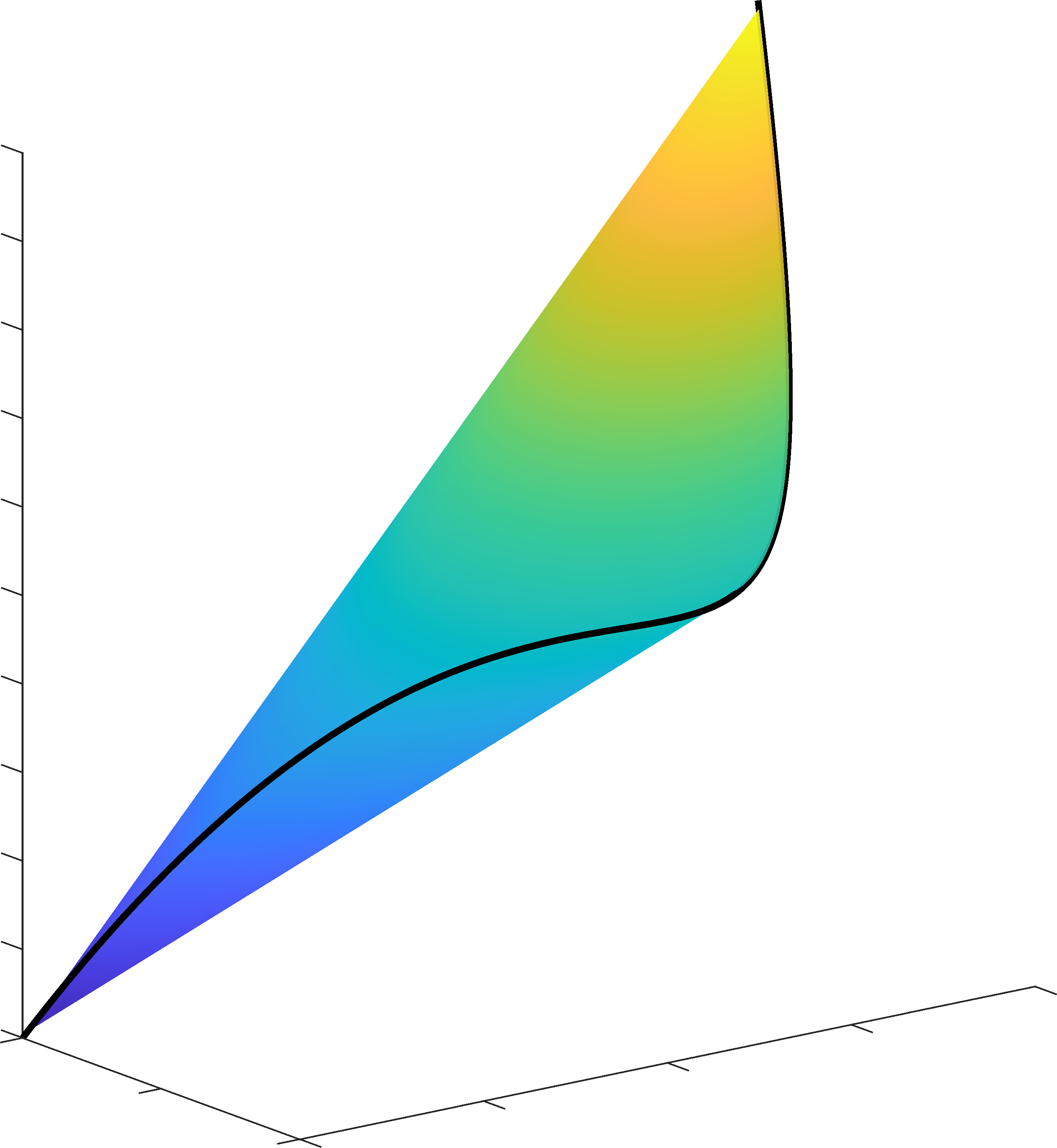}
\caption{Polynomial Lifting (Ours)}
\end{subfigure}
\caption{Different finite-dimensional approximations $\bbP$ of the infinite-dimensional space of probability measures $\ccP([-1, 1])$ with $\Omega=[-1,1]$. Left and middle: $2$-dimensional probability simplex and right: Monomial moment space $\con\{ (x, x^2, x^3) : x \in [-1,1]\}$ of degree $3$. The approximations are obtained as the convex hulls of the black curves $x \mapsto \varphi(x)$ for 3 different choices of $\varphi$. From left to right cf. \cref{ex:sampling}, \cref{ex:piecewise_linear} and \cref{ex:polynomial}. 
In all cases, the black curves themselves correspond to Dirac measures and the convex hulls of the curves correspond to the space of probability measures. In contrast to the simplex that only has a finite number of extreme points, the monomial moment curve comprises a continuum of extreme points so that no Dirac measure on the monomial moment curve can be expressed as a convex combination of other Diracs.
}
\label{fig:moment_curve}
\end{figure}
A fundamental question in the study of the lifted biconjugate $(f_u)_\Lambda^{**}$ is the characterization of its domain $\dom (f_u)_\Lambda^{**}=\con \varphi(\Omega)$, which, assuming $f_u$ is bounded from below, equals the convex hull of the image of $\varphi$. For that purpose we define the probability moment space, as the set of all vectors $y \in \bR^n$ for which there exists a probability measure $\mu\in \ccP(\Omega)$ such that the $k^{\text{th}}$ component of $y$ is the $k^{\text{th}}$ moment of $\mu$:
\begin{align} \label{eq:moments}
\bbP = \left\{ y \in \bR^n : \exists \mu\in \ccP(\Omega), y_k = \langle \mu, \varphi_k \rangle \right\},
\end{align}

In other words $\bbP$ is the set of all ``infinite'' convex combinations of points $\varphi(x) \in \bR^n$ with $x \in \Omega$, i.e., of points that belong to the image of~$\varphi$. 
The following relation to $(\bbM)_+$ holds: If $\varphi_0\equiv 1$, $\cP_\Lambda = \{y \in (\cM_\Lambda)_+ : y_0 = 1\}$. Also see \cref{lem:nonnegativity_cone}.

Furthermore, note that for some $x \in \Omega$ the moment vector $\varphi(x)\! = \int \varphi(x') \dd \delta_{x}(x')$ of a Dirac measure $\delta_x$ sets up a correspondence between lifted points $\varphi(x)$ and the Dirac measure $\delta_x$ itself.
Thus the moment space can be written equivalently as the set of all finite convex combinations of moment vectors of Dirac measures, in the same way the unit simplex is the convex hull of the unit vectors:
\begin{proposition} \label{prop:convex_hull_moments}
Let $\emptyset \neq \Omega$ be compact and $\Lambda=\langle \varphi_0,\ldots ,\varphi_n\rangle \subset \cC(\Omega)$ with $\varphi_0 \equiv 1$. Then every moment vector of a probability measure is a finite convex combination of moment vectors of Dirac measures, i.e.\
\begin{align}
\con({\varphi(\Omega)}) = \cP_\Lambda,
\end{align}
and $\cP_\Lambda$ is nonempty and compact.
\end{proposition}
\begin{proof}
Note that $\cP_\Lambda$ is convex and bounded. It is also closed: To this end consider a sequence $y^t \to y$ with $y^t \in\cP_\Lambda$. This means for any $t$ there exists $\mu^t \in \cP(\Omega)$ with $y_k^t = \int_\Omega \varphi_k(x) \dd \mu^t(x)$ for $k = 1, \ldots, n$. Since $\Omega$ is compact, due to Prokhorov's theorem, see \cite[Sec.~1.1]{San15}, there exists a weakly$^*$ convergent subsequence $\mu^{t_j} \overset{\ast}{\rightharpoonup} \mu$ and, hence, $\int_\Omega \varphi_k(x) \dd \mu^{t_j}(x) = y_k^{t_j} \to  \int_\Omega \varphi_k(x) \dd \mu(x) = y_k$. Therefore $y \in \cP_\Lambda$. Next we show identity of the support functions of the convex sets $\con({\varphi(\Omega)})$ and $\cP_\Lambda$ as this implies the equality of the sets. To this end let $f \in \Lambda=\langle \varphi_0,\ldots ,\varphi_n\rangle$ with $\varphi_0 =1$. We write $f(x) = \langle \varphi(x),f \rangle$. Assume that $f_0=0$. We have the identities:
\begin{align*}
\sigma_{\varphi(\Omega)}(-f) =-\min_{x \in \Omega} ~\langle \varphi(x), f \rangle =-\min_{\mu\in \cP(\Omega)} ~\int_\Omega f (x) \dd \mu(x) = -\min_{y\in \cP_\Lambda} ~\langle y, f \rangle =\sigma_{\cP_\Lambda}(-f),
\end{align*}
where the first equality follows by the definition of the support function, the second equality by \cref{lem:lp_exactness} and the third equality by the definition of $\cP_\Lambda$, the identity 
$$
\min_{\mu\in \cP(\Omega)} ~\int_\Omega f (x) \dd \mu(x) = \min_{\mu\in \cP(\Omega)} \sum_{k=1}^n f_k \int_\Omega \varphi_k(x) \dd\mu(x),
$$
and the substitution $y_k =  \int_\Omega \varphi_k(x) \dd\mu(x)$.
Since for each such $f$ the support functions are equal we have equality of the support functions of $\varphi(\Omega)$ and $\cP_\Lambda$. Since $\Omega$ is compact and $\varphi$ continuous and the convex hull of a compact set stays compact, cf. \cite[Cor. 2.30]{RoWe98}, we can replace $\varphi(\Omega)$ with its convex hull $\con \varphi(\Omega)$ and the conclusion follows.	
\end{proof}
	
The proof of the above proposition also reveals that for $f \in \Lambda$, the lifted biconjugate is a linear function over the moment space,

\begin{align} \label{eq:biconj_linear_lifted}
\lifted{f}{\Lambda}^{**}(y) = \langle y, f \rangle + \ind_{\bbP}(y),
\end{align}
whose minimization is actually equivalent to minimizing the original function.

Specializing $\varphi_k$ to the monomial basis and $\Omega$ to a set defined via polynomial inequalities, this yields the formulation proposed by \cite{lasserre2001global,lasserre2002semidefinite} for constrained polynomial optimization.

More generally, a meaningful notion of moments $y_k$ is induced by a lifting map $\varphi$ which is a homeomorphism between $\Omega$ and $\varphi(\Omega)$. 
\begin{definition}[lifting map] \label{def:lifting_map}
Let $\Omega$ be compact and be nonempty. Then we say the mapping $\varphi :\Omega \to \bR^n$ is a lifting map if $\varphi$ is continuous on $\Omega$ and injective with continuous inverse $\varphi^{-1}:\varphi(\Omega) \to \bR^m$.
\end{definition}
	
It is instructive to discuss possible choices for $\varphi$ including the ones that correspond to existing discretizations for the continuous MRF such as the discrete approach and the piecewise linear approach. In the latter two cases, the moment space is merely the unit simplex. In that sense, the monomial moment space can be interpreted as a ``nonlinear probability simplex'', which, in contrast to the unit simplex, has infinitely many extreme points, see \cref{fig:moment_curve}. For the same reason, as we will see in the course of this section, it better suits the continuous nature of our optimization problem.

The discrete sampling-based approach is recovered by the following choice of $\varphi$:
\begin{example} \label{ex:sampling}
We discretize the interval $\Omega=[a,b]$, $a<b$ and re-define $\Omega := \{t_1, \ldots, t_n \}$ with $t_k \in [a,b]$, $t_k < t_{k+1}$. For any $t_k \in \Omega$ let
$\varphi(t_k) = e_k$ with $e_k \in \bR^{n}$ being the $k^{\text{th}}$ unit vector. As a result $\varphi$ is the canonical basis and spans the space of discrete functions $f : \{t_1, \ldots, t_n \} \to \bR$ and
the moment space $\con({\varphi(\Omega)}) = \bbP$ is given by the unit simplex.
\end{example}	
The above example can be extended by assigning any points $x \in (t_k, t_{k+1})$ to points on the connecting line between two corresponding Diracs which results in a more continuous formulation:
\begin{example} \label{ex:piecewise_linear}
Let $\Omega=[a,b]$, $a < b$: 
Let $t_k < t_{k+1}$ and $t_1 = a$, $t_{n} = b$ be a sequence of knots that subdivide the interval $\Omega$ into $n-1$ subintervals $[t_k, t_{k+1}] =:\Omega_k$. We define $\varphi(x) = \alpha e_k + (1-\alpha) e_{k+1}$,
with $e_k \in \bR^{n}$ being the $k^{\text{th}}$ unit vector, $\alpha\in [0,1]$ such that $\alpha t_k + (1-\alpha) t_{k+1} = x$. This yields a  $2$-sparse lifting map that has been used in
the related work of \cite{laude16eccv}. Its component functions $\varphi_k$ are the finite element hat basis functions that span the space $\Lambda$ of univariate piecewise linear functions on $\Omega$ and the moment space $\con({\varphi(\Omega)}) = \bbP$ is given
by the unit simplex.
\end{example}
For $\varphi_k$ being chosen as the monomials we obtain the classical notion of moments:
\begin{example} \label{ex:polynomial}
For $\varphi_0 = 1$, the space of univariate polynomials $\Lambda=\bR[x]$ with maximum degree $n$ is spanned by the monomials:
\begin{align} \label{eq:moment_curve}
\varphi(x) = (x, x^2, \ldots, x^{n}),
\end{align}
and $\con({\varphi(\Omega)}) = \bbP$ is the monomial moment space.
\end{example}

\begin{example} \label{ex:caratheodory_curve}
Identifying $\Omega = \{x \in \bR^2 \colon \lVert x \rVert_2 = 1 \}$ with the complex unit circle $\Omega \cong \{ z \in \mathbb C \colon \lvert z \rvert = 1\}$ and again assuming $\varphi_0 = 1$, the mapping
\begin{align}
\varphi (z) = (\Real(z), \Img(z), \ldots, \Real(z^n), \Img(z^n)) \in \bR^{2n}
\end{align}
spans the space $\Lambda$ of real trigonometric polynomials of maximum degree $n$. Parametrizing elements $z \in \Omega$ via the bijection between $z=e^{i\omega}$ and its angle $\omega \in [0, 2\pi )$ the components of $\
\varphi$ are the Fourier basis functions, which define the Carath\'eodory curve \cite{caratheodory1911variabilitatsbereich}. The convex hull $\con({\varphi(\Omega)}) = \bbP$ is the trigonometric moment space.
\end{example}

\subsection{Extremal Subspaces}
In contrast to the piecewise linear lifting, the (trigonometric) polynomial lifting is extremal in the sense that no Dirac measure on the  moment curve can be expressed as a convex combination of other Diracs, which sets up a certain one-to-one correspondence between Diracs $\delta_x$ and lifted points $\varphi(x)$. This is illustrated in \cref{fig:moment_curve} monomial case. 

More formally, we call a lifting map $\varphi$ an extremal curve if each $y \in \varphi(\Omega)$ is an extreme point of $\con(\varphi(\Omega))$ defined according to \cite[Sec. 18]{Roc70}.
\begin{definition}[extreme points]
Let $C$ be a convex set and $y \in C$. Then $y$ is called an extreme point of $C$ if there is no way to express $y$ as a convex combination $y=(1-\alpha) x + \alpha z$ of $x,z \in C$ and $0 < \alpha < 1$, except by taking $y = x =z$.
\end{definition}
\begin{definition}[extremal moment curve] \label{def:extremal_curve}
Let $\Omega$ be compact and be nonempty. Then we say the mapping $\varphi :\Omega \to \bR^n$ is an extremal moment curve if $\varphi$ is a lifting map and any point $y \in \varphi(\Omega) \subset \bR^n$ is an extreme point of $\con \varphi(\Omega)$.
\end{definition}
Via a change of basis it becomes clear that the definition of extremality is independent of a specific choice of a basis for $\Lambda$. Therefore extremality is rather a property of the subspace $\Lambda$. This also motivates the following lemma which shows that extremality is inherited along a hierarchy $\Theta \subset \Lambda \subset \cC(\Omega)$.

\begin{lemma}[extremal subspaces] \label{lem:extremal_subspace}
Let $\emptyset\neq \Omega \subset \bR^m$ be compact. Let $\Theta \subset \Lambda\subset \cC(\Omega)$ be a hierarchy of finite-dimensional subspaces of the space of continuous functions $\cC(\Omega)$. Let $\Theta=\langle \theta_1, \ldots, \theta_n \rangle$ such that $\theta:\Omega \to \bR^{n}$ is an extremal curve. Then $\Lambda$ is spanned by an extremal curve as well.
\end{lemma}
\begin{proof}
$\{ \theta_1, \ldots, \theta_n \}$ is a basis of $\Theta \subset \Lambda$ and therefore linearly independent. Since $\Lambda$ is finite-dimensional in view of the basis extension theorem $\{\theta_1, \ldots, \theta_n \}$ can be extended to a basis $\varphi=(\theta_1, \ldots, \theta_n, \psi_{1},\ldots,\psi_{k})$ of $\Lambda$ with vectors $\psi_i \in G$, where $\spn G= \Lambda$ and $|G| <\infty$ such that $\spn \varphi = \Lambda$.

Now choose $y \in \Omega$ and consider $\varphi(y)$. Let $\alpha \in (0,1)$ and $\varphi(y) = \alpha x + (1-\alpha)z$ for $x,z\in \con \varphi(\Omega) \subset \bR^{n+k}$. Due to Carath\'eodory \cite[Thm. 2.29]{RoWe98} there exist coefficients $\alpha_{i}$, $\beta_l > 0$ such that $x = \sum_{i=1}^{n+k+1} \alpha_i \varphi(x^{(i)})$ and $z = \sum_{l=1}^{n+k+1} \beta_l \varphi(z^{(l)})$, $z^{(l)},x^{(i)}\in \Omega$ with $\sum_{i=1}^{n+k+1} \alpha_i =1$, $\sum_{l=1}^{n+k+1} \beta_l = 1$.

This implies that $\theta(y)=\alpha \sum_{i=1}^{n+k+1} \alpha_i \theta(x^{(i)})+ (1-\alpha) \sum_{l=1}^{n+k+1} \beta_l \theta(z^{(l)})$. Extremality of $\theta$ implies that $\theta(y)=\sum_{i=1}^{n+k+1} \alpha_i \theta(x^{(i)})=\sum_{l=1}^{n+k+1} \beta_l \theta(z^{(l)})$ and therefore $\theta(y)=\theta(x^{(i)})=\theta(z^{(l)})$. Since $\theta$ is an extremal curve it is injective and therefore $y=z^{(l)}=x^{(i)}$. This implies that $\varphi(y)=x = z$.
\end{proof}
	
\begin{lemma}[extremality of quadratic subspace]
Let $\emptyset\neq \Omega \subset \bR^m$ be compact. Let $\varphi(x)=(x_1, x_2, \ldots, x_m, \|x\|^2)$. Assume that $\langle \varphi_1, \ldots, \varphi_{m+1} \rangle \subset \Lambda \subset \cC(\Omega)$. Then $\Lambda$ is spanned by an extremal curve.
\end{lemma}
\begin{proof}
Choose $y \in \Omega$. Let $\alpha \in (0,1)$ and $\varphi(y) = \alpha x + (1-\alpha)z$ for $x,z\in \con \varphi(\Omega) \subset \bR^{m+1}$. Due to Carath\'eodory \cite[Thm. 2.29]{RoWe98} there exist coefficients $\alpha_{i}$, $\beta_k > 0$ such that $x = \sum_{i=1}^{m+2} \alpha_i \varphi(x^{(i)})$ and $z = \sum_{k=1}^{m+2} \beta_k \varphi(z^{(k)})$, $z^{(k)},x^{(i)}\in \Omega$ with $\sum_{i=1}^{m+2} \alpha_i =1$, $\sum_{k=1}^{m+2} \beta_k = 1$. 

Now choose $f(x)=\|x-y\|^2=\|x\|^2 - 2\langle y, x\rangle + \|y\|^2$. Then we have for $a = (-2 y_1, \ldots, -2y_m, 1)$ and $a_0 = \|y\|^2$:
\begin{align*}
0 = f(y) = \langle \varphi(y), a \rangle +a_0 &= \left\langle \alpha \sum_{i=1}^{m+2} \alpha_i \varphi(x^{(i)}) + (1-\alpha) \sum_{k=1}^{m+2} \beta_k \varphi(z^{(k)}), a \right\rangle +a_0 \\
&= \sum_{i=1}^{m+2} \alpha \cdot \alpha_i \cdot \langle \varphi( x^{(i)}), a \rangle + \sum_{k=1}^{m+2} ( 1 - \alpha) \cdot \beta_k \cdot \langle \varphi(z^{(k)}), a \rangle +a_0\\
&= \sum_{i=1}^{m+2} \alpha \cdot \alpha_i \cdot f(x^{(i)}) + \sum_{k=1}^{m+2} ( 1 - \alpha ) \cdot \beta_k \cdot f( z^{(k)})
\end{align*}
As $f(x^{(i)})>0$ for $x^{(i)} \neq y$ and $\alpha >0$ it holds that $x^{(i)} \neq y$ implies $\alpha_{i} = 0$. The same is true for $z^{(k)}$ and $\beta_{k}$. Hence $x = \varphi(y) = z$, and therefore $\varphi : \Omega \to \bR^{m+1}$ is an extremal curve. In view of \cref{lem:extremal_subspace} $\Lambda$ is spanned by an extremal curve.
\end{proof}
This shows that in a piecewise polynomial discretization with degree at least $2$ the corresponding basis inherits the extreme point property from the extremality of the subspace of quadratic functions.
\begin{figure}[t!]
\begin{subfigure}[b]{0.325\linewidth}
\centering
\hspace{1cm}
\includegraphics[width=\linewidth]{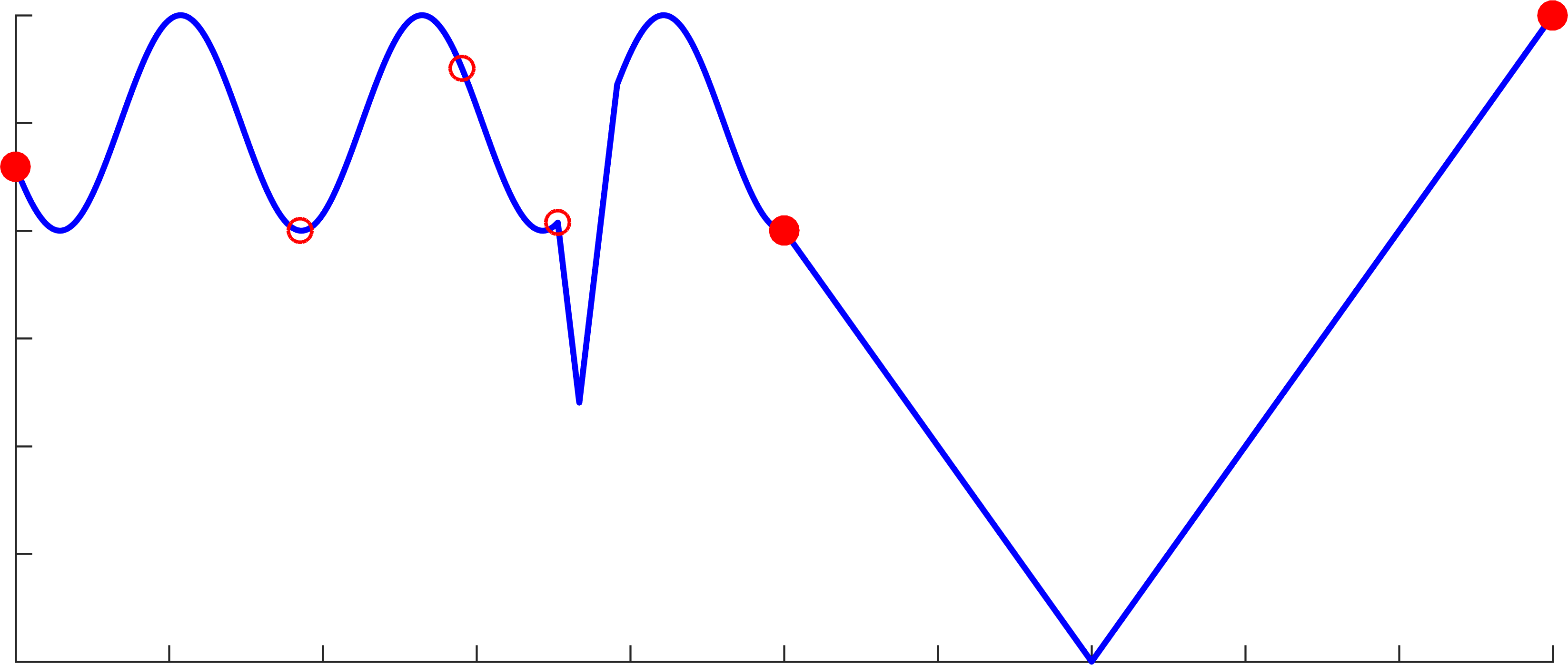}
\end{subfigure}		
\hfill
\begin{subfigure}[b]{0.325\linewidth}
\centering
\hspace{1cm}
\includegraphics[width=\linewidth]{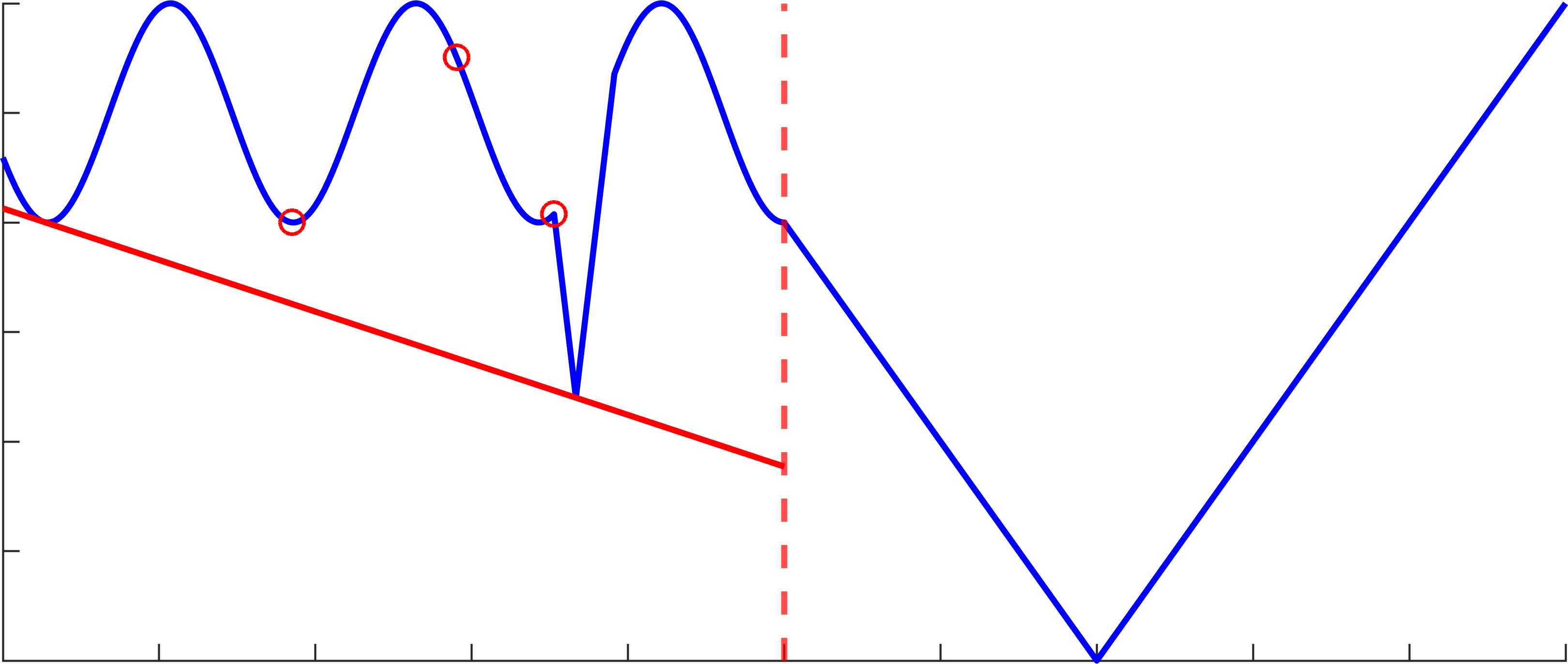}
\end{subfigure}
\hfill
\begin{subfigure}[b]{0.325\linewidth}
\centering
\includegraphics[width=\linewidth]{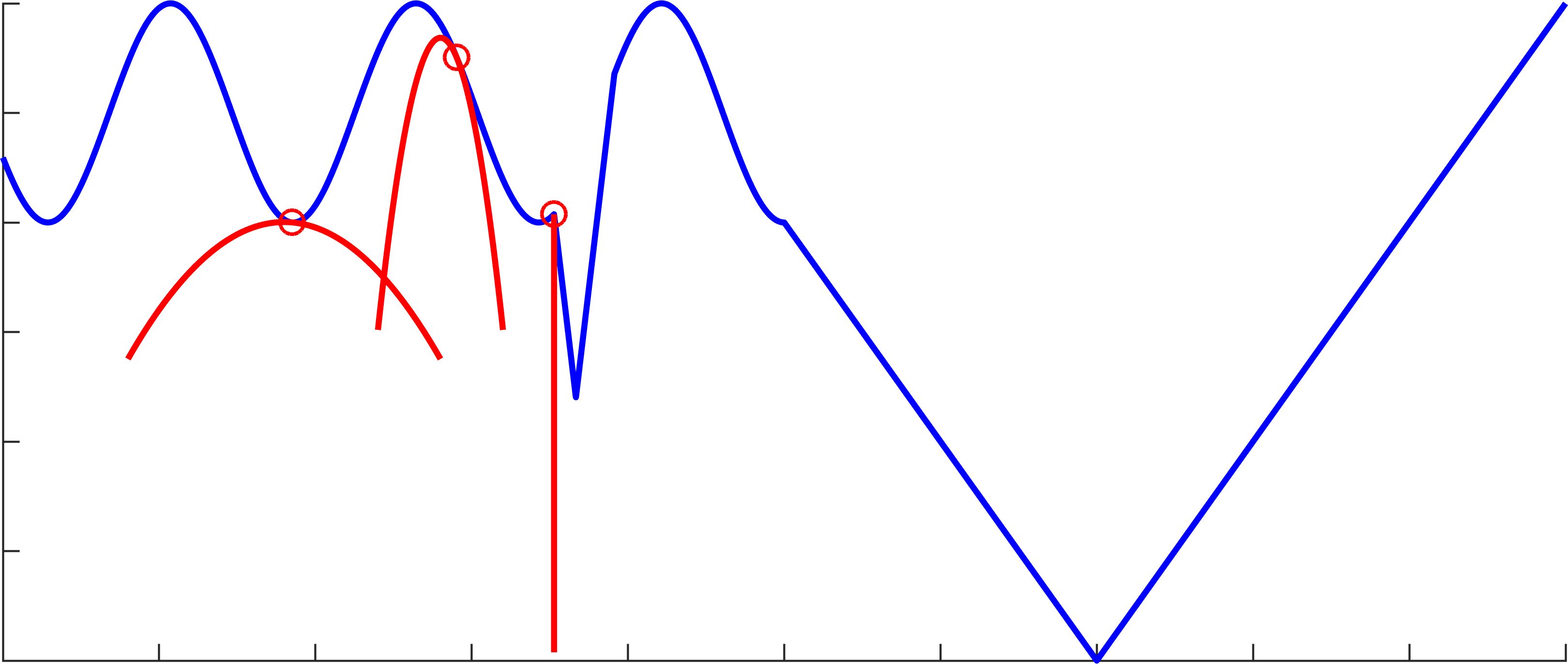}
\end{subfigure}\\
\begin{subfigure}[b]{0.325\linewidth}
\centering
\includegraphics[width=\linewidth]{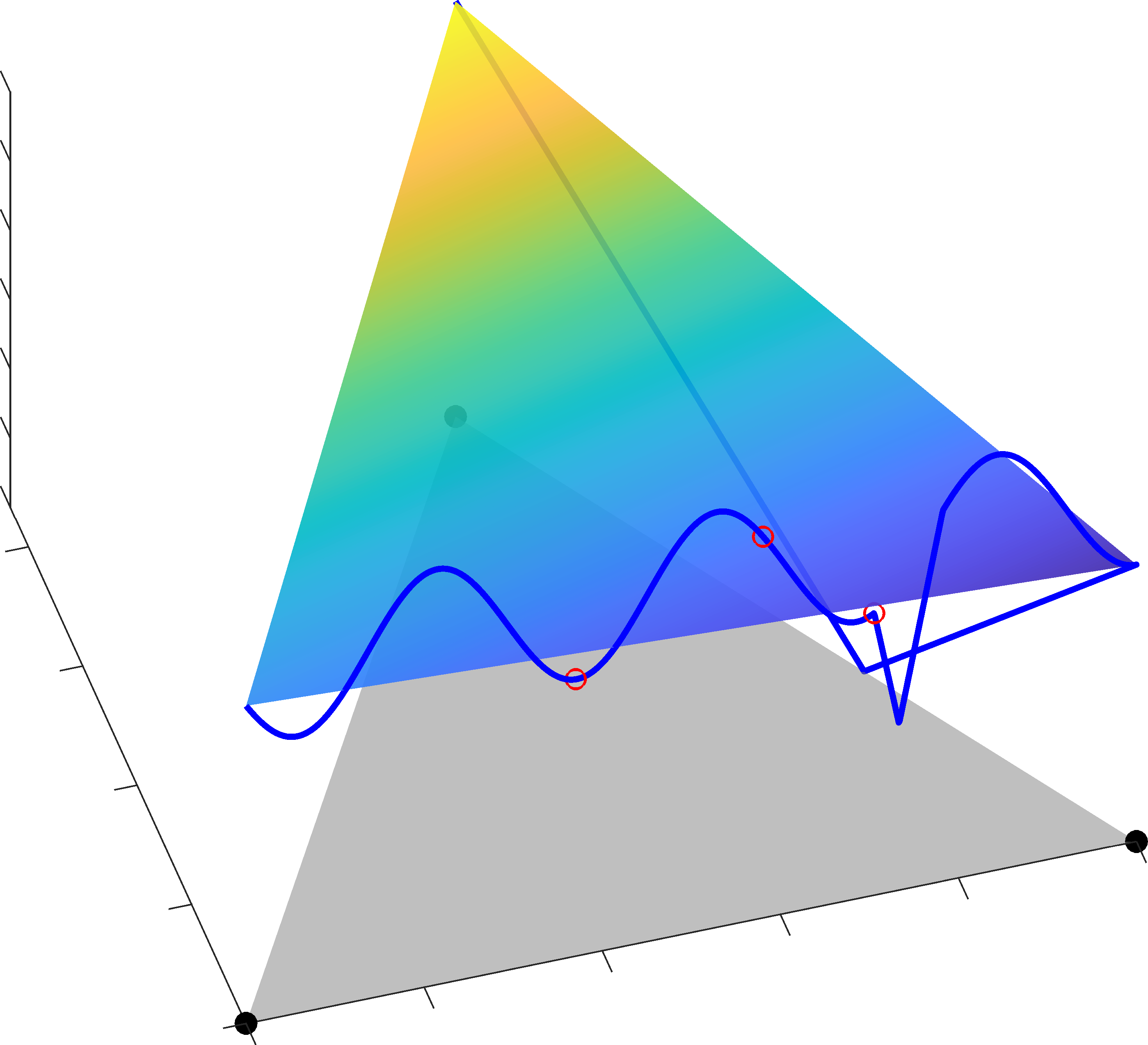}
\caption{Standard Lifting}
\end{subfigure}
\hfill
\begin{subfigure}[b]{0.325\linewidth}
\centering
\includegraphics[width=\linewidth]{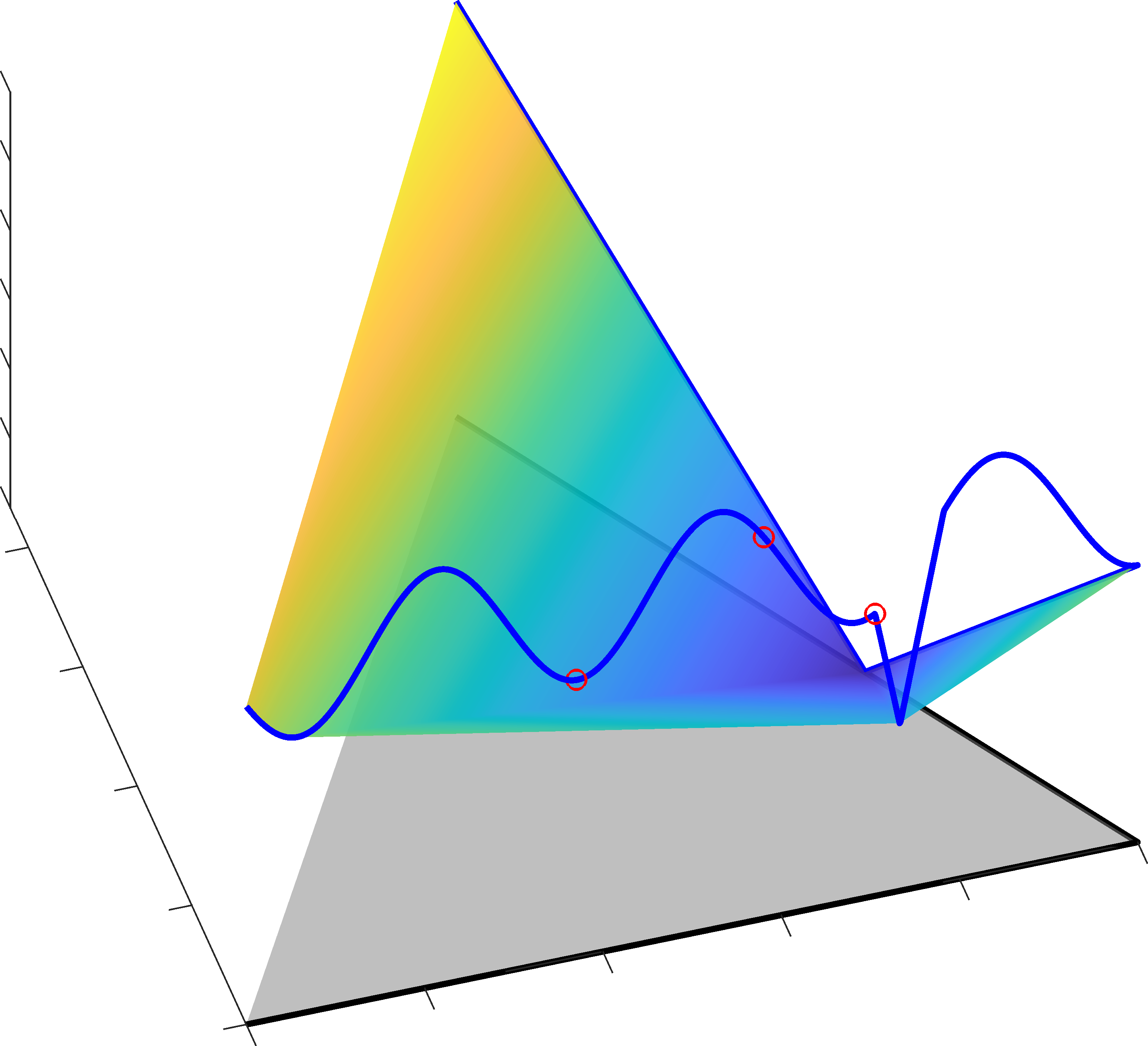}
\caption{Sublabel Lifting \cite{laude16eccv}}
\end{subfigure}
\hfill
\begin{subfigure}[b]{0.325\linewidth}
\centering
\includegraphics[width=\linewidth]{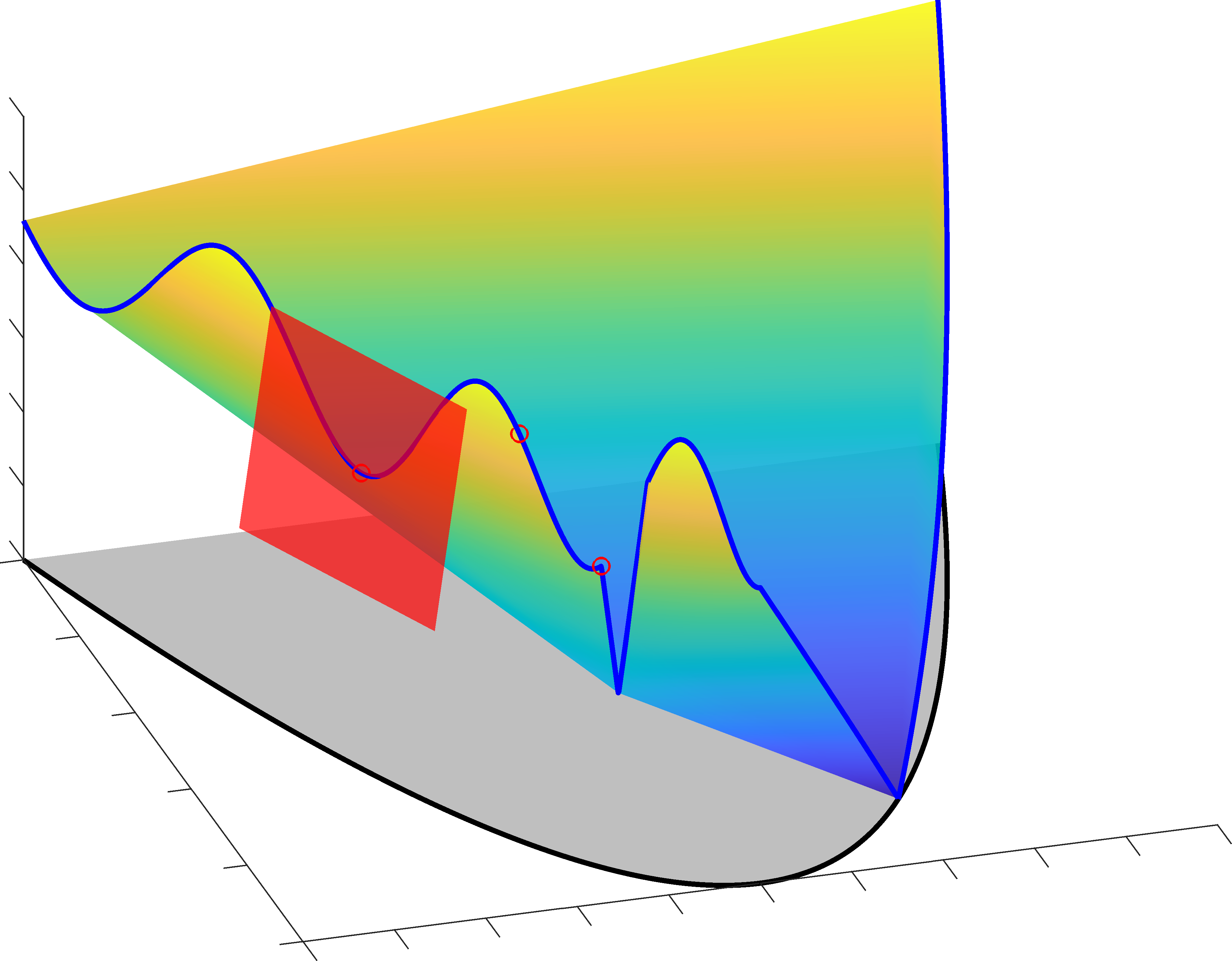}
\caption{Polynomial Lifting (Ours)}
\end{subfigure}
\caption{Lifted version $\lifted{f}{\Lambda}$ of the function $f$ for 3 different choices of $\varphi$. In the top row the same nonconvex function $f$ (blue curves) is depicted. The colored surfaces in the bottom row correspond to the lifted biconjugates $\lifted{f}{\Lambda}^{**}$, where the gray shadow areas correspond to their domains $\dom \lifted{f}{\Lambda}^{**} = \bbP$ and the blue curves correspond to the lifted cost $f_\Lambda$, see \cref{eq:lifted_cost}. The black curves in the bottom row correspond to the moment curve described by Diracs $\varphi(\Omega)$, i.e., the domain of $\lifted{f}{\Lambda}$. From left to right, cf. \cref{ex:sampling}, \cref{ex:piecewise_linear} and \cref{ex:polynomial}. The nonlinear supporting dual functions $q_{\lambda,\beta}(x) = \langle \varphi (x), \lambda\rangle-\beta$ (red curves) to $f$ in the top row (middle and right), are transformed into linear supporting hyperplanes $l_{\lambda,\beta}(y)=\langle y, \lambda \rangle - \beta$ (red surfaces) to $\lifted{f}{\Lambda}^{**}$ in the
bottom row through the feature map $\varphi$. In the language of kernel methods, these functions can be interpreted as the nonlinear decision boundaries that separate individual points on the graph of $f$ from the epigraph of $f$. Only in the most right case such a separation is possible. As a result the polynomial lifting preserves the nonconvex cost function $\lifted{f}{\Lambda}^{**} \circ \varphi = f$ on $\varphi(\Omega)$ whereas the $2$-sparse lifting (middle) only leads to a piecewise convex under-approximation $\lifted{f}{\Lambda}^{**} \circ \varphi \leq f$.
}
\label{fig:envelope}
\end{figure}
Extremal curves are key to preserve the cost function when restricted to the set of discretized Diracs $\varphi(x)$:
\begin{theorem}\label{thm:general_binconj}
Let $\Omega \subset \bR^m$ be nonempty and compact and let $f:\Omega \to \bR$ be lsc. Furthermore, let $\varphi:\Omega \to \bR^n$ be an extremal curve.
Then we have
\begin{align}
f_\Lambda^{**} \circ \varphi = f
\end{align}
on $\Omega$.
In addition, we have that $\con f_\Lambda = f_\Lambda^{**}$.
\end{theorem}
\begin{proof}
Since $f$ is lsc and $\Omega$ compact $f$ is bounded from below, i.e., there is $\gamma > -\infty$ so that $f(x)\geq \gamma$ for all $x \in \Omega$. We have $\dom f_\Lambda = \varphi(\Omega) \subset \bR^n$ and 
in view of \cite[Prop. 2.31]{RoWe98} it holds for any $y \in \bR^{n}$,
\begin{align*}
(\con f_\Lambda)(y) &= \inf\left\{ \sum_{i=1}^{n+1} \lambda_if_\Lambda(y^{(i)}) : \lambda_i \geq 0, \, \sum_{i=1}^{n+1} \lambda_i = 1, \, y = \sum_{i=1}^{n+1} \lambda_i y^{(i)},\, y^{(i)} \in \bR^n  \right\} \\
&= \inf\left\{ \sum_{i=1}^{n+1} \lambda_i f(x^{(i)}) : \lambda_i \geq 0, \, \sum_{i=1}^{n+1} \lambda_i = 1, \, y = \sum_{i=1}^{n+1} \lambda_i\varphi(x^{(i)}),\, x^{(i)} \in \Omega  \right\} \geq \gamma.
\end{align*}
Let $x \in \Omega$. Since the only possible convex combination of the extreme point $\varphi(x)$ from points $y^{(i)} \in \varphi(\Omega)$ is $\varphi(x)$ itself, we have
\begin{align*}
(\con f_\Lambda)(\varphi(x)) &= \inf\left\{ \sum_{i=1}^{n+1} \lambda_if_\Lambda(y^{(i)}) : \lambda_i \geq 0, \, \sum_{i=1}^{n+1} \lambda_i = 1, \,\varphi(x) = \sum_{i=1}^{n+1} \lambda_i y^{(i)},\, y^{(i)} \in \varphi(\Omega)  \right\} \\
&=f_\Lambda(\varphi(x)) =f(x).
\end{align*}
This shows that $\con f_\Lambda \circ \varphi = f$.
Since $\con f_\Lambda$ is bounded from below $\con f_\Lambda$ is proper.

Let $f$ be lsc. 

Then $f_\Lambda$ inherits its lower semicontinuity from $f$:

Assume that $y \in \varphi(\Omega)$. Then there exists $x \in \Omega$ with $y = \varphi(x)$ and we have due to the continuity of $\varphi$ and $\varphi^{-1}$:
\begin{align*}
\liminf_{y' \to \varphi(x)}~f_\Lambda(y')&:= \lim_{\delta \to 0^+} \inf ~\{f_\Lambda(y') : y' \in B_\delta(\varphi(x))\} \\
&=\lim_{\delta \to 0^+} \inf~ \{f_\Lambda(\varphi(x')) : \varphi(x') \in B_\delta(\varphi(x)) \}
=\lim_{\epsilon \to 0^+} \inf~ \{ f(x') : x' \in B_\epsilon(x) \} \\
&= \liminf_{x' \to x}~f(x')\geq f(x) =f_\Lambda(y).
\end{align*}	
Since $\Omega$ is compact and $\varphi$ continuous the image $\varphi(\Omega) = \dom f_\Lambda$ is compact as well. Since $f_\Lambda$ is also bounded from below it is coercive in the sense of \cite[Def. 3.25]{RoWe98} (also called super-coercive in other literature).
Then we can invoke \cite[Cor. 3.47]{RoWe98} and deduce that $\con f_\Lambda$ is proper, lsc and convex. In view of \cite[Thm. 11.1]{RoWe98} we have $\con f_\Lambda = \cl \con f_\Lambda = f_\Lambda^{**}$.
\end{proof}
For a geometric intuition of this theorem we refer to \cref{fig:envelope}. 

For a piecewise polynomial discretization with degree at least $2$, due to extremality, the primal discretized energy $\mathcal{F}_\Lambda$ restricted to $\varphi(\Omega)^\cV$ agrees with the original energy $F$. In particular, this implies that an obtained Dirac solution $(\varphi(x_u^*))_{u \in \cV}$ of the discretization corresponds to a solution of the original problem in the same way integer solutions of LP relaxations are certificates of optimality for the corresponding ILP. 
\begin{proposition}\label{prop:tightness}
Let $\emptyset \neq \Omega \subset \bR^m$.
Let the metric $d$ be induced by a norm and assume the space of linear functions on $\Omega$ is contained in $\Lambda$. Furthermore, assume $\Lambda$ is spanned by an extremal
curve $\varphi:\Omega \to \bR^n$ and $f_u$ lsc. Then for any $y \in (\bR^n)^\cV$, with $y_u =\varphi(x_u)$, $x_u \in \Omega$ the following identity holds true:
$$
\mathcal{F}_\Lambda(y) = F(x).
$$
In particular, whenever $y^*$ is a solution of problem \cref{eq:discrete_primal} such that $y^*\in (\bR^n)^\cV$, with $y_u^* =\varphi(x_u^*)$, for some $x_u^* \in \Omega$, $x^*$ is a solution of the original problem \cref{eq:mrf}.
\end{proposition}
\begin{proof}
Let $d$ be induced by some norm $\lVert \cdot \rVert$ and denote its dual norm by $\lVert \cdot \rVert_*$. As shown in \cref{thm:general_binconj}, the unaries preserve the original cost functions at $\varphi(\Omega)$. Hence it remains to show that the for the pairwise costs it holds $\sigma_{\cK_\Lambda}((\nabla_\Lambda y)_{(u,v)}) = \|x_u- x_v\|$. By assumption $(\nabla_\Lambda y)_{(u,v)} = \varphi(x_u) - \varphi(x_v)$.
We rewrite
$$
\sigma_{\mathcal{K}_\Lambda}(\varphi(x_u) - \varphi(x_v)) = \sup_{\lambda \in \Lambda \cap \Lip_d(\Omega)} \lambda(x_u) - \lambda(x_v) \leq \|x_u - x_v \|. 
$$
For any $x_u, x_v \in \Omega$ we have
\begin{align*}
\lVert x_u - x_v \rVert = \sup_{\substack{p \in \bR^m:\|p\|_* \leq 1}} |\langle x_u - x_v, p \rangle| = \langle x_u - x_v, p^* \rangle,
\end{align*}
where $p^*$ denotes the maximizer in the supremum, which exists due to the compactness of the unit ball in a finite-dimensional space. Define the linear function $\lambda^* := \langle \cdot, p^* \rangle$ and note that by assumption $\lambda^* \in \Lambda$. In addition we have shown that $\lambda^* \in \Lip_d(\Omega)$ is $1$-Lipschitz.
This implies $\lVert x_u - x_v \rVert = \lambda^*(x_u) - \lambda^*(x_v) \leq \sup_{\lambda \in \Lambda \cap \Lip_d(\Omega)} \lambda(x_u) - \lambda(x_v) \leq \lVert x_u - x_v \rVert$ and hence equality holds.
\end{proof}
	
\subsection{A Generalized Conjugacy Perspective}
The previous results can be obtained from a generalized conjugacy point of view. In particular, the convex conjugate of the lifted function $f_\Lambda^*$ is comprised by the notion of $\Phi$-conjugacy, see \cite[Ch.~11L*]{RoWe98}:

\begin{definition}[generalized conjugate functions] \label{def:phi_conjugates}
Let $X$ and $Y$ be nonempty sets. Let $\Phi: X \times Y \to \exR$ be any function. Let $f: X \to \exR$. Then the $\Phi$-conjugate of $f$ on $Y$ at $y \in Y$ is defined by
\begin{align}
f^\Phi(y) := \sup_{x \in X} \Phi(x, y) - f(x),
\end{align}
and the $\Phi$-biconjugate of $f$ back on $X$ at $x \in X$ is given by
\begin{align}
f^{\Phi\Phi}(x) := \sup_{y \in Y} \Phi(x, y) - f^\Phi(y).
\end{align}
We say that $f$ is a $\Phi$-envelope on $X$ if $f$ can be written in terms of a pointwise supremum of a collection of elementary functions $x \mapsto \Phi(x, y)  - \alpha$, where $(\alpha, y) \in \exR \times Y$ is the parameter element.

Let $g: Y \to \exR$. Then, the $\Phi$-conjugate of $g$ on $X$ at $x \in X$ is defined by
\begin{align}
g^\Phi(x) := \sup_{y \in Y} \Phi(x, y) - g(y),
\end{align}
and the $\Phi$-biconjugate of $g$ back on $Y$ at $y \in Y$ is given by
\begin{align}
g^{\Phi\Phi}(y) := \sup_{x \in X} \Phi(x, y) - g^\Phi(x).
\end{align}
We say that $g$ is a $\Phi$-envelope on $Y$ if $g$ can be written in terms of a pointwise supremum of a collection of elementary functions $y \mapsto \Phi(x, y) - \beta$, where $(\beta, x) \in \exR \times X$ is the parameter element.

\end{definition}

For $X=\Omega$, $Y = \bR^n$, and $\Phi(x, y) = \langle y, \varphi(x) \rangle$, the convex conjugate $f_\Lambda^*$ is identical to the $\Phi$-conjugate $f^\Phi$ of $f$, while its biconjugate $f_\Lambda^{**}$ is the tightest lsc convex extension of $f_\Lambda$ to $\con \varphi(\Omega)$. The $\Phi$-biconjugate $f^{\Phi\Phi}$ of $f$ at a point $x \in \Omega$ is equal to the classical biconjugate of $f_\Lambda$, evaluated at $\varphi(x)$, i.e., $f^{\Phi\Phi} = f_\Lambda^{**} \circ \varphi$ on $\Omega$, showing that the $\Phi$-biconjugate is a convexly composite function and, therefore, it is nonconvex in general. Actually, $\Phi$-conjugacy also comprises lifting to measures via $\varphi(x) = \delta_x$, $\Phi$ the corresponding dual pairing
and $Y=\cC(\Omega)$.

As a consequence of \cite[Ex.~11.63]{RoWe98}, the considered $\Phi$-conjugacy can be interpreted in terms of under-approximation by functions in $\Lambda$. In analogy to the biconjugate $f^{**}$, which is the pointwise supremum of affine-linear functions majorized by $f$, the $\Phi$-biconjugate $f^{\Phi\Phi}$ is the pointwise supremum of functions in $\Lambda$ up to constant translation majorized by $f$. This point of view also relates $(f^\Phi)^*$ and $f^{\Phi\Phi}$ by each other.
\begin{remark}\label{rem:hyperplane}
The function $(f^\Phi)^* = f_\Lambda^{**}$ is the pointwise supremum of all affine-linear functions $l_{\lambda,\beta}:=\langle \cdot, \lambda \rangle - \beta$ for which $q_{\lambda,\beta}:=\Phi(\cdot, \lambda) - \beta$ is majorized by $f$.
This can be seen as follows: The Legendre--Fenchel conjugate can be characterized via the identity
$(f^\Phi)^*(y) = \sup_{(\lambda, \beta) \in \epi (f^\Phi)} \langle y, \lambda \rangle - \beta$.
The observation now follows from the fact that $(\lambda, \beta) \in \epi (f^\Phi)$ if and only if $q_{\lambda,\beta}$ is majorized by $f$.
\end{remark}
The correspondence between $f_\Lambda^{**}$ and $(f^\Phi)^*$ and between the minorizers $l_{\lambda,\beta}$ and $q_{\lambda,\beta}$ is illustrated in \cref{fig:envelope}. Note that this is closely related to the idea of feature maps $\varphi$ in linear classifiers. 

\Cref{thm:general_binconj} identifies all lsc functions $f:X \to \bR$ as $\Phi$-envelopes whenever $\varphi$ is extremal:
\begin{corollary} 
Let $\Omega \subset \bR^m$ be nonempty and compact and let $f:\Omega \to \bR$ be bounded from below. Furthermore, let $\varphi:\Omega \to \bR^n$ be an extremal curve.
Then we have
\begin{align}
f^{\Phi\Phi} = f_\Lambda^{**} \circ \varphi = \cl f,
\end{align}
on $\Omega$.
\end{corollary}
\begin{proof}
Since $f$ is finite-valued and bounded from below on $\Omega$ we have $(\cl f)(x) >-\infty$ for $x \in \Omega$ and therefore $\cl f$ is finite-valued. By \cite[Ex.~11.63]{RoWe98} $f^{\Phi\Phi}$ is the largest $\Phi$-envelope below $f$. Since $\varphi$ is continuous relative to $\Omega$, $f^{\Phi\Phi} = f_\Lambda^{**} \circ \varphi$ is lsc relative to $\Omega$. Since $\cl f$ is the largest lsc function below $f$ we have $f^{\Phi\Phi} \leq \cl f$. Since $\cl f \leq f$ we also have $(\cl f)^{\Phi\Phi} \leq f^{\Phi\Phi}$. Invoking \cref{thm:general_binconj} we have
$$
\cl f = (\cl f)^{\Phi\Phi} \leq f^{\Phi\Phi} \leq \cl f,
$$
on $\Omega$. Therefore $f^{\Phi\Phi} = \cl f$ on $\Omega$.
\end{proof}
Up to the presence of the compact set $\Omega$, this result generalizes the basic quadratic transform \cite[Ex.~11.66]{RoWe98} originally due to \cite[Prop.~3.4]{poliquin1990subgradient} (for lsc functions only), which is obtained by choosing $\varphi(x)=(x_1,x_2, \ldots, x_m, \|x\|^2)$.
In \cite[Thm.~1]{balder1977extension} a similar duality formula is shown for $\Phi$-couplings of a certain ``needle-type''. Our result is instead based on the extremality condition,
which, from a primal point of view, captures an intuitive and sharp (sufficient) condition for the above result for the one-sided linear couplings we consider.
For the component functions of $\varphi$ being the hat basis, see \cref{ex:piecewise_linear}, the class of $\Phi$-envelopes are the piecewise convex functions, see, \cref{fig:envelope} middle.

\section{A Tractable Conic Program for MRFs} \label{sec:cone_program}

\subsection{Nonnegativity and Moments}
After discretization a next step to obtain a practical implementation is to derive finite characterizations of the lifted biconjugates $f_\Lambda^{**}$ and the constraint set $\mathcal{K}_\Lambda$. We will show that the formulations can be rewritten in terms of a semi-infinite conic program which can be implemented using semidefinite programming in the piecewise polynomial case.

For now let $f_u \in \Lambda$. \Cref{eq:biconj_linear_lifted} then shows, that the challenging part is to characterize the moment space $\cP_\Lambda$:
The following result shows that up to normalization, $\cP_\Lambda$ can be written in terms of the dual cone of the cone of functions in $\Lambda$ that are nonnegative on $\Omega$. 
\begin{lemma} \label{lem:nonnegativity_cone}
Let $(\cM_\Lambda)_+$ be the cone of moments of nonnegative measures as defined in \cref{eq:nonnegative_moments}
and let $\mathcal{N}_\Lambda$ be the cone of the coefficients of the functions in $\Lambda=\langle \varphi_0, \ldots, \varphi_{n}\rangle$ that are nonnegative on $\Omega$ defined as: 
\begin{align}
\mathcal{N}_\Lambda:= \{ p \in \bR^{n+1} : \langle p, \varphi(x)\rangle \geq 0, \;\forall\,x \in \Omega \}.
\end{align}
Then $(\cM_\Lambda)_+$ is equal to $\mathcal{N}_\Lambda^*$, where $\mathcal{N}_\Lambda^*$ denotes the dual cone of $\mathcal{N}_\Lambda$.

If, in addition, $\varphi_0\equiv 1$ we also have $\cP_\Lambda = \{y \in (\cM_\Lambda)_+ \colon y_0 = 1\}$.
\end{lemma}
\begin{proof}
Let $y \in (\cM_\Lambda)_+$. This means there exists $\mu \in \cM_+(\Omega)$ such that $y_k := \int_\Omega \varphi_k(x) \dd \mu(x)$. Let $p \in \mathcal{N}_\Lambda$. Because of $\langle p, \varphi(\cdot) \rangle \in \Lambda \subset \cC(\Omega)$ and $\langle p, \varphi(x)\rangle \geq 0$ for all $x \in \Omega$ and $\mu \in \cM_+(\Omega)$ is a nonnegative measure it holds $\langle p, y \rangle = \sum_{k=0}^n p_k \int_\Omega \varphi_k(x) \dd \mu(x)= \int_\Omega \langle p, \varphi(x) \rangle \dd \mu(x) \geq 0$ for all $x \in \Omega$. Since $p\in \mathcal{N}_\Lambda$ was an arbitrary choice from $\mathcal{N}_\Lambda$ we have $y \in \mathcal{N}_\Lambda^*$.

Next we show $(\cM_\Lambda)_+^{*} \subseteq \mathcal{N}_\Lambda$ as this implies $\mathcal{N}_\Lambda^* \subseteq (\cM_\Lambda)_+^{**} = (\cM_\Lambda)_+$, where the last equality holds since $ (\cM_\Lambda)_+$ is convex by definition of convexity and closed by the same argument used in the proof of \cref{prop:convex_hull_moments}.
Take $p \in (\cM_\Lambda)_+^{*}$. Let $x \in \Omega$. Now we choose $y \in (\cM_\Lambda)_+$ such that $y_k = \langle \delta_x, \varphi_k \rangle =\varphi_k(x)$. Then $p \in (\cM_\Lambda)_+^{*}$ implies that $\langle p, y \rangle \geq 0$. Since the choice $x \in \Omega$ was arbitrary we have $\langle p, \varphi(x) \rangle \geq 0$ for all $x \in \Omega$ and therefore $p \in \mathcal{N}_\Lambda$. 
 
Finally, $\cP_\Lambda = \{y \in (\cM_\Lambda)_+ \colon y_0 = 1\}$ follows from the fact that $\mu \in \cM(\Omega)$ is an element of $\cP(\Omega)$ if and only if $\mu \in \cM_+(\Omega)$ and $\langle \mu, \varphi_0 \rangle = 1$.
\end{proof}

Before we specialize $\Lambda$ to the space of polynomials, we derive a cone programming formulation of the Lipschitz constraints $\lambda_{(u,v)} \in \Lip_{d}(\Omega)$. Here we restrict $\Omega=[a,b]$ to be a compact interval. In the following, we provide an implementation for two specific metrics. As before, this boils down to nonnegativity of functions: Firstly, we consider total variation regularization, i.e. $d(x, y) = |x - y|$: Assume that $\Lambda$ is closed under differentiation, i.e., $\varphi$ is differentiable and $\varphi_k' \in \Lambda$. Then, the condition $\lambda_{(u,v)} \in \Lip_{d}([a,b])$ can be phrased in terms of the constraints $-1 \leq \lambda_{(u,v)}'(x) \leq 1$ for all $x\in [a,b]$, where $\lambda_{(u,v)}'$ is the derivative of $\lambda_{(u,v)}$. Equivalently, this means that the coefficients of the functions $1+\lambda_{(u,v)}'$ and $1-\lambda_{(u,v)}'$ are in $\bbN$. Secondly, we consider Potts regularization, i.e., $d(x, y) = \llbracket x = y \rrbracket$.
Since $\lambda_{(u,v)}$ is a univariate function, and constant terms in the dual variable do not matter, the condition $\lambda_{(u,v)} \in \Lip_{d}([a,b])$ can be equivalently phrased as $0 \leq \lambda_{(u,v)} \leq 1$, see, \cite[Ex.~1.17]{villani2003topics}. Equivalently, this means that the coefficients of the functions $\lambda_{(u,v)}$ and $1-\lambda_{(u,v)}$ are in $\bbN$.

\subsection{Semidefinite Programming and Nonnegative Polynomials}
As we have seen in the previous section, an important ingredient for a tractable formulation is the efficient characterization of nonnegativity of functions in a finite-dimensional subspace $\Lambda \subset \cC(\Omega)$. A promising choice of $\Lambda$ in that regards is the space of polynomials. Indeed, the characterization of nonnegativity of polynomials is a fundamental problem in \emph{convex algebraic geometry} surveyed in \cite{blekherman2012semidefinite}:
Let $\bR[x_1, \ldots, x_m]$ denote the ring of possibly multivariate polynomials with $p\in \bR[x_1, \ldots, x_m]$. Then $p=\sum_{\alpha \in I}p_\alpha x^\alpha$ for monomials $x^\alpha$. Let $\deg p$ denote its degree. A key result from real algebraic geometry is the Positivstellensatz due to \cite{krivine1964anneaux} and \cite{stengle1974nullstellensatz} refined in \cite{schmudgen1991thek} and \cite{putinar1993positive}. It characterizes polynomials $p \in \bR[x_1, \ldots, x_m]$ that are positive on semi-algebraic sets $\Omega$, i.e. $p(x) > 0$ for all $x \in \Omega$, where $\Omega$ is defined in terms of polynomial inequalities.
	
Key to such results is a \emph{certificate} of nonnegativity of the polynomial $p$ that involves \emph{sum-of-squares} (SOS) multipliers $q$, where $q$ is SOS if $q(x) = \sum_{i=1}^N q_i^2(x)$ for polynomials $q_i \in \bR[x_1, \ldots, x_m]$.
For intervals $X=[a,b]$, thanks to \cite[Thm. 3.72]{blekherman2012semidefinite} originally due to \cite[Cor. 2.3]{powers2000polynomials} we have following result:
\begin{lemma} \label{lem:pos-uni-interval}
Let $a < b$. Then the univariate polynomial $p \in \bR[x]$ is nonnegative on $[a,b]$ if and only if it can be written as
\begin{align}
p(x) =\begin{cases}
s(x) + (x-a)\cdot(b-x)\cdot t(x) & \text{if $\deg p$ is even,} \\
(x - a)\cdot s(x) + (b-x)\cdot t(x) & \text{if $\deg p$ is odd,}
\end{cases}
\end{align}
where $s, t \in \bR[x]$ are sum of squares. If $\deg p = 2n$, then we have $\deg s \leq 2n$, $\deg t\leq 2n -2$, while if $\deg p=2n+1$, then $\deg s \leq 2n, \deg t\leq 2n$.
\end{lemma}
Remarkably, the above result provides us with explicit upper bounds of the degrees of the SOS multipliers $s$ and $t$ that are important to derive a practical implementation: Then the SOS constraints can be formulated in terms of semidefinite and affine inequalities:
We adopt \cite[Lem. 3.33]{blekherman2012semidefinite} and \cite[Lem. 3.34]{blekherman2012semidefinite}:
\begin{lemma} \label{lem:sos_sdp}
A univariate polynomial $p \in \bR[x]$ with $\deg p = 2n$, $n \geq 0$ is SOS if and only if there exists a positive semidefinite matrix $Q \in \bR^{n+1 \times n+1}$ such that 
\begin{align} \label{eq:coefficients}
p_{k} =\sum_{\substack{0\leq i,j \leq n, \\i+j = k}} Q_{ij}, \quad \forall\, 0\leq k \leq 2n.
\end{align}
\end{lemma}
Invoking the results above SDP-duality yields the following compact representation of $(\bbM)_+$:
\begin{lemma}
Let $n\geq0$. For odd degree $2n+1$, $y \in (\bbM)_+$ if and only if 
\begin{align}
b M_{0, n}(y) \succeq M_{1,n}(y) \succeq a M_{0, n}(y),
\end{align}
for Hankel matrices
\begin{align}
M_{i,n}(y) := \begin{bmatrix} y_i & y_{i+1} & \hdots & y_{i+n} \\ y_{i+1} & y_{i+2} & \hdots & y_{i+n+1} \\ \vdots & & \hdots & \vdots \\y_{i+n} & &\hdots & y_{i+2n} \end{bmatrix}.
\end{align}
For even degree $2n$, $y \in (\bbM)_+$ if and only if 
\begin{gather}
M_{0, n}(y) \succeq 0,\\
\left( a + b \right) M_{1, n-1}(y) - ab M_{0,n-1}(y) \succeq M_{2,n-1}(y).
\end{gather}
\end{lemma}
\begin{proof}
Follows by \cref{lem:pos-uni-interval} and \cref{lem:sos_sdp} invoking elementary SDP-duality.
\end{proof}

\subsection{Convergence of a Piecewise Polynomial Hierarchy}
In experiments, we will discretize the dual problem with a piecewise polynomial family of functions. For intervals, the following proposition shows that either by increasing the number of pieces or the degree of the polynomial the primal-dual gap can be reduced.
In our case we approximate the Lipschitz dual variable in terms of a Lipschitz spline. As a consequence existing results such as \cite[Thm.~2]{fix2014duality} do not apply. Instead, we use a construction based on Bernstein-polynomials. Then the result follows from \cite[Thm.~1]{brown1987lipschitz}.
\begin{proposition}\label{prop:gapconvergence}
Assume that $\Omega = [a, b] \subset \bR$, $a < b$, and let the metric $d$ be given by $d(x, y) = |x - y|$.
Furthermore, let $\Lambda\subset \cC(\Omega)$ be the space spanned by continuous piecewise polynomials on intervals $[t_i, t_{i+1}]$
defined by a regularly spaced grid with nodes given by $t_i = a + (b - a) \cdot (i -1) / K$, $i = 1, \hdots, K + 1$.
Then the optimality gap satisfies:
$$
\cref{eq:mrf} - \cref{eq:discrete_dual} = \mathcal{O}(1 / (K \cdot \sqrt{\deg})),
$$
where $\deg$ is the degree of the polynomial on each piece.
\end{proposition}
\begin{proof}
We consider the discretized dual problem~\cref{eq:discrete_dual} where $\cK_\Lambda$ is the set of coefficients corresponding to $1$-Lipschitz piecewise polynomials on $[a,b]$ of degree $\deg$ with $K$ pieces.
Also recall that we have the following relations between the dual and primal problems: $\cref{eq:discrete_dual} \leq \cref{eq:dualMPTsmall}=\cref{eq:opt_relax}$ where the last equality follows from strong duality \cref{prop:map_inference_mrf}.

Now, let us denote a maximizer of \cref{eq:dualMPTsmall} as $\lambda^* \in \Lip_d(\Omega)^\cE$. Existence of such a dual maximizer follows by \cref{prop:map_inference_mrf}.
Then, one has for any $\lambda \in \Lambda^\cE$:
\begin{align}
\min_{x \in \Omega} ~ f_u(x) - (\Div \lambda^*)_u(x) &= \min_{x \in \Omega} ~ f_u(x) - (\Div \lambda)_u(x) - (\Div \lambda^*)_u(x) + (\Div \lambda)_u(x) \notag \\
&\leq \min_{x \in \Omega} ~ f_u(x) - (\Div \lambda)_u(x) + \| -(\Div (\lambda^* - \lambda))_u \|_\infty.
\end{align}
This allows us to bound the optimality gap by:
\begin{align}
\cref{eq:dualMPTsmall} - \cref{eq:discrete_dual}  &\leq \sum_{u \in \cV} \| -(\Div (\lambda - \lambda^*))_u \|_{\infty} \notag \\
& \leq \sum_{u \in \cV} |d(u)| \cdot \sup_{e \in \cE} ~ \| \lambda_e - \lambda_e^* \|_{\infty} \notag \\
& \leq 2 |\cE| \cdot \sup_{e \in \cE} \| \lambda_e - \lambda_e^* \|_{\infty}, \label{eq:bound_laststep}
\end{align}
where $d(u)$ denotes the degree of the vertex $u$.

For a $L$-Lipschitz function $f : [0, 1] \to \bR$ there exists a Bernstein polynomial $p : [0, 1] \to \bR$ with $p(0) = f(0)$ and $p(1) = f(1)$
such that $\| p - f \|_\infty \leq \frac{3L}{2} \deg^{-1/2}$ \cite[Thm.~2.6]{carothers1998short}. By \cite[Thm.~1]{brown1987lipschitz},
this polynomial is $L$-Lipschitz as well. For each $e\in \cE$ we pick the coefficients of the function $\lambda_e$ such that it
approximates the optimal dual variable $\lambda_e^*$ with such a polynomial individually on each interval $[t_i, t_{i+1}]$. Then one obtains an overall $1$-Lipschitz polynomial with
the following bound:
\begin{equation}
\| \lambda_e - \lambda_e^* \|_{\infty} \leq \frac{3(b - a)}{2K\sqrt{\deg}}.
\end{equation}
Inserting this into \cref{eq:bound_laststep} yields via $\cref{eq:opt_relax} = \cref{eq:dualMPTsmall}$:
\begin{equation}
\cref{eq:opt_relax} - \cref{eq:discrete_dual} \leq 3 |\cE| \frac{(b - a)}{K \sqrt{\deg}},
\end{equation}
which due to tightness $\cref{eq:opt_relax} =\cref{eq:mrf}$ from \cref{prop:tightness_local_marginal} gives the stated $\mathcal{O}(1 / (K\cdot \sqrt{\deg}))$ rate.
\end{proof}

\subsection{A First-Order Primal-Dual Algorithm} \label{sec:impl_pdhg}
We are now ready to describe the algorithm for solving the resulting semidefinite program. We first consider the case $f_u \in \Lambda$ and $\Lambda$ is the space of univariate polynomials.
We propose to use the PDHG~\cite{chambolle2011first} algorithm, as it can exploit the partially separable structure of our SDP. The primal-dual algorithm optimizes the problem \cref{eq:discrete_primal} via alternating projected gradient descent/ascent steps applied to the saddle-point formulation of \cref{eq:discrete_primal}:
\begin{align} \label{eq:saddle_point_pdhg}
\min_{y \in (\cP_{\Lambda})^\cV} ~\max_{p \in (\mathcal{K}_\Lambda)^\cE} ~~\langle y, f - \Div_\Lambda p \rangle,
\end{align}
which is obtained by expanding the support function in Problem~\cref{eq:discrete_primal} and substituting the expression~\cref{eq:biconj_linear_lifted} for the lifted biconjugates.
In each iteration the algorithm performs a projected gradient ascent step in the dual $p$ followed by a projected gradient descent step in the primal variable $y$.
Subsequently it performs an extrapolation step in the primal. 
The projections onto the sets $(\cP_{\Lambda})^\cV$ and $(\mathcal{K}_\Lambda)^\cE$ are separable and can therefore be carried out in parallel on a GPU using the SDP characterizations derived above. For practicality, we introduce additional auxiliary variables and linear constraints to decouple the affine constraints~\cref{eq:coefficients} and the SDP constraints. The projection operator of the semidefinite cone can then be solved using an eigenvalue decomposition. 

\subsection{Piecewise Polynomial Duals and Nonlinear Lifted Biconjugates}
The polynomial discretization can be extended by means of a continuous piecewise polynomial representation of the dual variables resulting in a possibly more accurate approximation of the dual subspace $\Lambda$. Then, both, nonnegativity and Lipschitz continuity can be enforced on each piece $\Omega_k$ individually. Continuity of the piecewise polynomial dual variables can be enforced via linear constraints.
The corresponding primal variable $y$ belongs to $y \in (\cM_\Lambda)_+ \times (\cM_\Lambda)_+ \times \cdots \times (\cM_\Lambda)_+$. Then the restriction that $y$ is a moment vector of a probability measure supported on the whole space $\Omega$ yields an additional sum-to-one constraint on the $0^{\text{th}}$ moments $1= \sum_{k=1}^K y_{k,0}$.

Another issue to address is when $f_u \not\in \Lambda$ which results in a nonlinear lifted biconjugate over the moment-space as in \cref{fig:envelope}. 

The formulation which is derived next addresses both: In particular it allows one to choose $\Lambda$ independently from $f_u$ which can even be discontinuous, as long as $f_u$ has a piecewise polynomial structure.
Key to the formulation is to rewrite the inner minimum in the dual formulation~\cref{eq:discrete_dual} exploiting a duality between nonnegativity and minimization of functions:
\captionsetup[subfigure]{labelformat=parens,justification=centering,singlelinecheck=false}
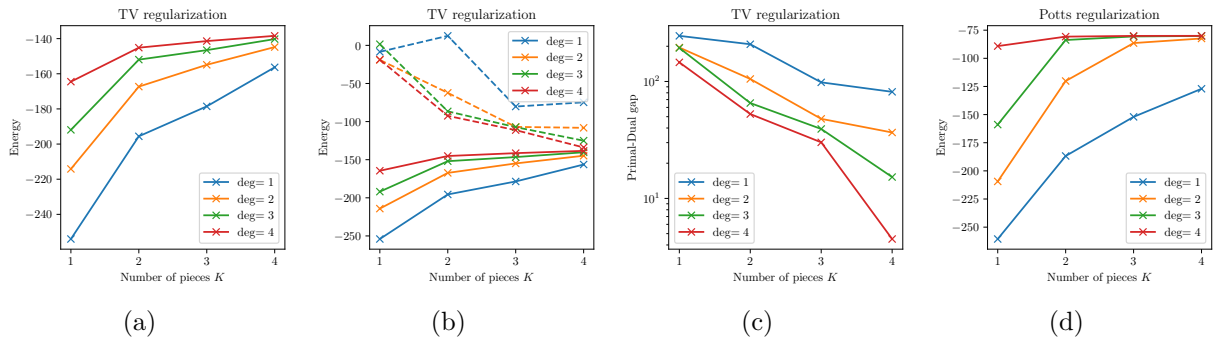
\begin{figure*}[!t]
\centering
\begin{subfigure}[b]{0.23\linewidth}
\centering
\scalebox{0.45}{\input{figures/energiesl1.pgf}}
\caption{}
\end{subfigure}
\hfill
\begin{subfigure}[b]{0.23\linewidth}
\centering
\scalebox{0.45}{\input{figures/energies_pd.pgf}}
\caption{}
\end{subfigure}
\hfill
\begin{subfigure}[b]{0.23\linewidth}
\centering
\scalebox{0.45}{\input{figures/energies_pd_gap.pgf}}
\caption{}
\end{subfigure}
\hfill
\begin{subfigure}[b]{0.23\linewidth}
\centering
\scalebox{0.45}{\input{figures/energiesl0.pgf}}
\caption{}
\end{subfigure}	
\caption{Primal and dual energies for MAP-inference in a continuous MRF with TV regularization using a piecewise polynomial hierarchy of dual variables.
(a) shows the dual energy for TV. In (b) the dashed lines correspond to the primal energy at the rounded solution and the solid lines correspond to the dual energy.
(c) shows the gap between the nonconvex primal energy at the rounded solution and the dual energy for TV regularization. (d) shows the dual energies for Potts regularization.
}
\label{fig:rel_tightness}
\end{figure*} 
Let $\Omega=[a,b]$, $a<b$ be a compact interval. Let $a = t_1 < t_2 < t_3< \cdots <t_{K+1} = b$ be a sequence of knots, where $\Omega_k:= [t_k, t_{k+1}]$. Let $\Theta$ be the space of univariate polynomials with some maximum degree $n$. 
Let $f:\Omega \to \bR$ be a possibly discontinuous lsc piecewise polynomial function defined by $f(x) = \min_{1\leq k \leq K} f_k(x) + \iota_{\Omega_k}(x)$ with $f_k \in \Theta$, i.e., $f_k(x) = \langle \varphi(x), w_k \rangle$ for coefficients $w_k \in \bR^{n+1}$, where $\varphi_0\equiv 1$.
First observe the following duality between nonnegativity and minimization of a lsc function:
$$
\min_{x \in \Omega} ~f(x) = \max_{q \in \bR} ~q - \iota_{\cN(\Omega)}(f-q),
$$
where we denote by $\cN(\Omega)=\{\lambda :\Omega \to \bR : \lambda(x) \geq 0, \forall\; x \in \Omega \}$ the cone of nonnegative lsc functions on $\Omega$.
Then we obtain for $A q = (e_0 q, \ldots, e_0 q)$, where $e_0=(1, 0, \ldots, 0) \in \bR^{n+1}$ is the $0^{\text{th}}$ unit vector:
\begin{align} \label{eq:nonnegativity_piecewise}
\min_{x \in \Omega} ~f(x) &= \max_{q \in \bR} ~q - \iota_{\cN(\Omega)}(f-q) \\
&=\max_{q \in \bR}~q - \sum_{k=1}^K \iota_{\cN_{\Theta}}(w_k - A_k q).
\end{align}
Fenchel--Rockafellar duality then yields:
\begin{align*} 
\min_{x \in \Omega} ~f(x) &= \min_{y \in (\bR^{n+1})^K} \iota_{\{1\}}(A^* y) + \sum_{k=1}^K \langle y_k, w_k \rangle + \iota_{(\cM_{\Theta})_+}(y_k) \notag \\
&= \min_{\substack{y \in ((\cM_{\Theta})_+)^K \\ \sum_{k=1}^K y_{k,0} = 1}} \sum_{k=1}^K\langle y_k, w_k \rangle.
\end{align*}
This formulation can be substituted in the dual problem~\cref{eq:discrete_dual} and we obtain
\begin{align}\label{eq:decomposition_piecwise_moments}
\sup_{p \in ((\bR^n)^K)^\cE} \sum_{u  \in \cV} ~\min_{\substack{y \in ((\cM_{\Theta})_+)^K \\ \sum_{k=1}^K y_{k,0} = 1}} \sum_{k=1}^K \langle y_k, (w - \Div_\Lambda p)_{u,k} \rangle - \sum_{e \in \cE} \iota_{\mathcal{K}_\Lambda}(p_e).
\end{align}
Here the dual variables $\lambda$ are chosen such that $(\Div_\Lambda p)_u$ represents a piecewise polynomial with knots $a = t_1 < t_2 < t_3< \cdots <t_{K+1} = b$ such that for each piece we have $\langle (\Div_\Lambda p)_{u,k}, \varphi(\cdot)\rangle \in \Theta$. Note that this does not require $\Lambda$ to be equal to the whole space of continuous piecewise polynomials of degree $n$. Indeed, $\Lambda$ can be a subspace thereof which covers the case where $f_u \not\in \Lambda$.

\section{Numerical Experiments}
\label{sec:num}
\subsection{Empirical Convergence Study}
\begin{figure*}[t!]
\centering
\renewcommand{\tabcolsep}{0.7mm}
\begin{tabular}{cccc}
Left image stereo pair & Standard $k=30$ & $k=5,\deg=1$ & $k=5,\deg=7$ \\ [2mm]
\includegraphics[width=0.24\linewidth]{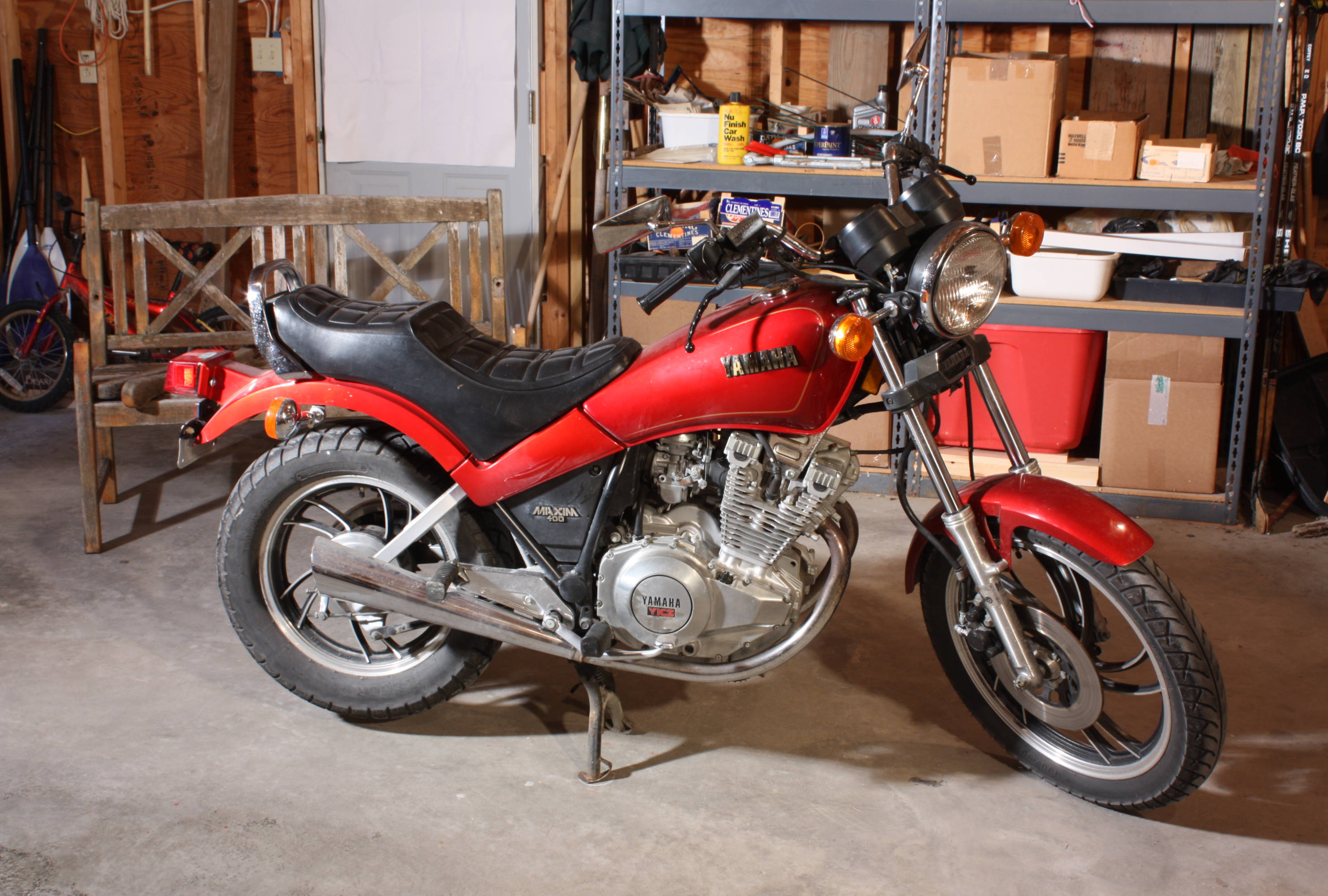}&\includegraphics[width=0.24\linewidth]{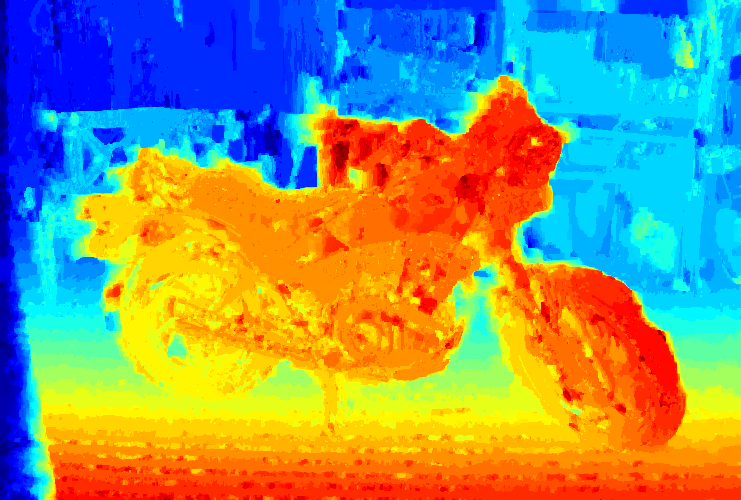}&\includegraphics[width=0.24\linewidth]{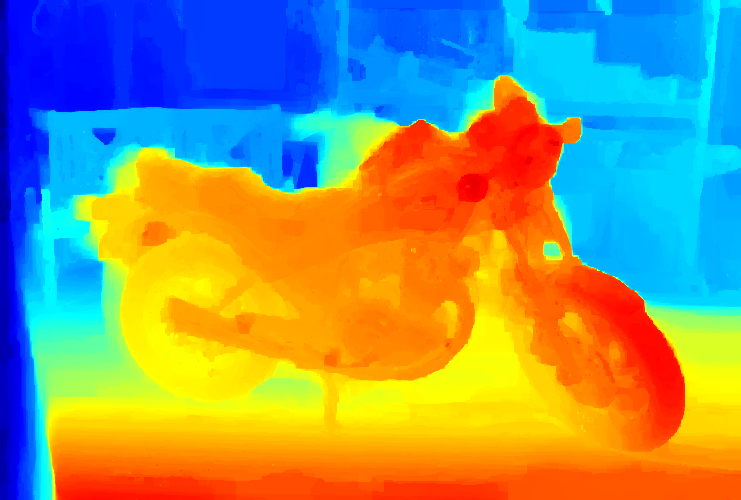} & \includegraphics[width=0.24\linewidth]{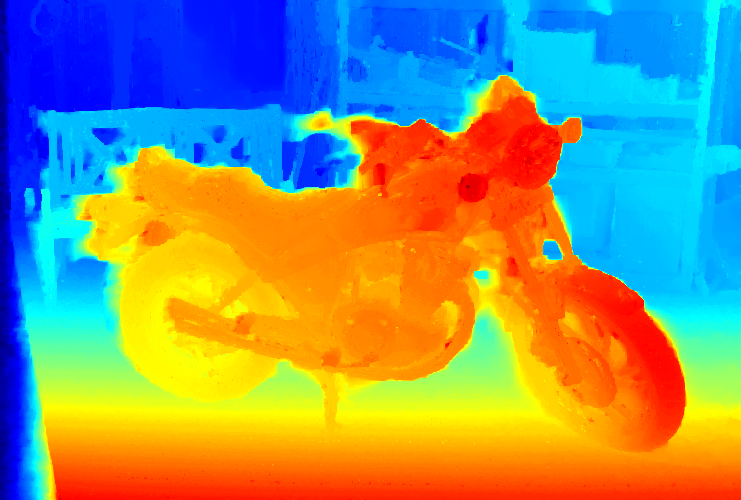} \\
& rounded $24611.51$ &  rounded $22283.25$ & rounded $19428.49$ \\
&  &  dual $16227.80$ &  dual $17472.13$\\[2mm]
\includegraphics[width=0.24\linewidth]{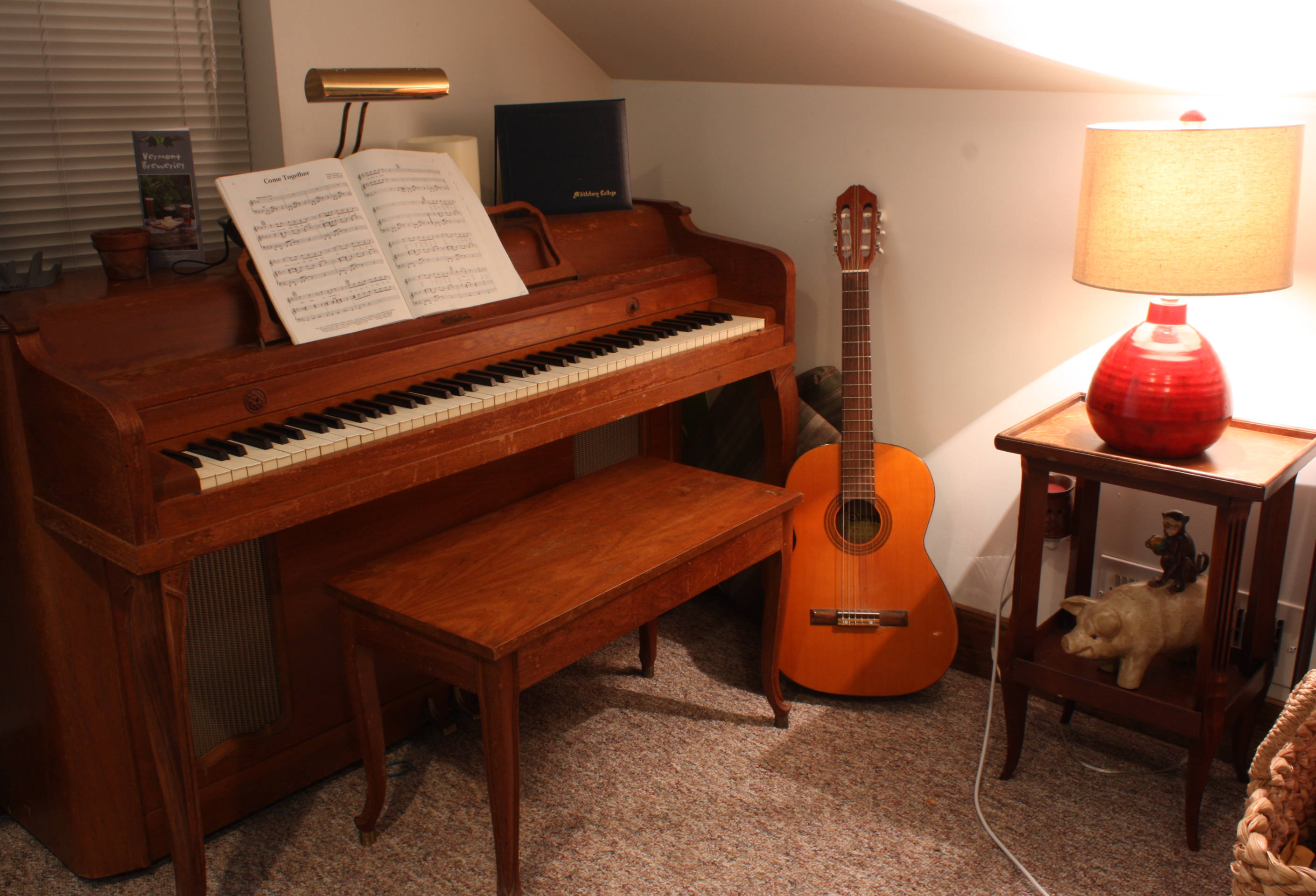}&\includegraphics[width=0.24\linewidth]{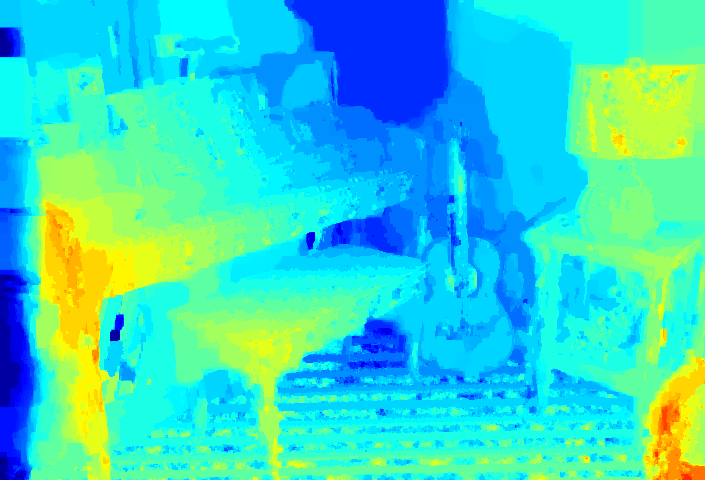}&\includegraphics[width=0.24\linewidth]{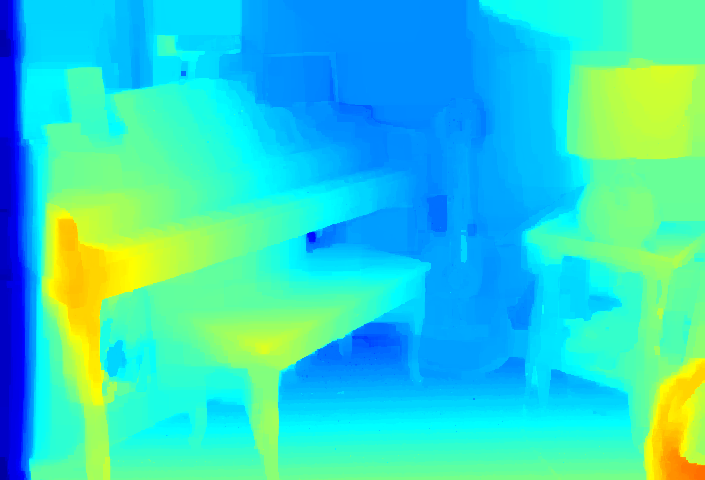} & \includegraphics[width=0.24\linewidth]{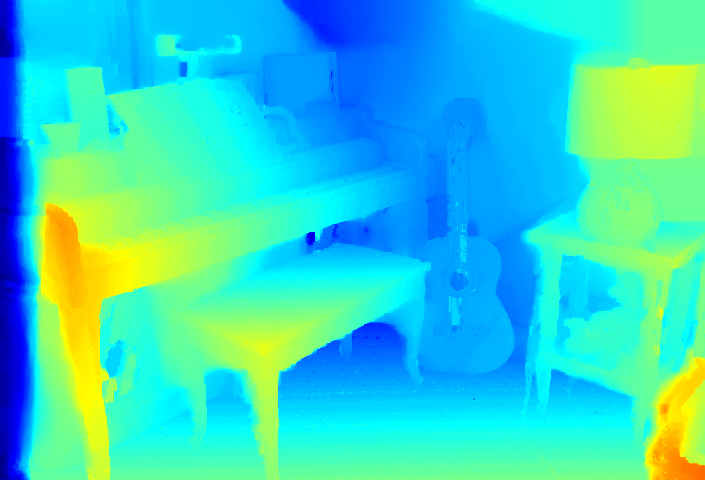} \\ 
& rounded $17510.55$ &  rounded $15027.55$ & rounded $13380.71$ \\
&  &  dual $10962.53$ &  dual $12008.45$
\end{tabular}

\caption{Stereo disparity estimation from a stereo image pair: Left: Standard MRF/OT discretization implemented using a continuous piecewise linear under-approximation for the unaries and piecewise linear duals. middle: piecewise linear duals. right: piecewise polynomial duals. The visual appearance of the solution to the standard MRF/OT discretization shows a strong grid bias. The dual energy gap increases for increasing the degree and/or the number of pieces. Likewise the energy at the rounded solution decreases.}
\label{fig:stereo_motorcycle}
\end{figure*}
In this first experiment we evaluate the local marginal polytope relaxation of the MRF formulation~\cref{eq:mrf} using a piecewise polynomial hierarchy of dual variables. We choose the graph $(\cV, \cE)$ to be a square grid of size $16 \times 16$. I.e., the vertices $\cV$ correspond to the points in the plane with its $x$- and $y$-coordinates being integers in the range $1,2,\ldots, 16$, and two vertices are connected by an edge whenever the corresponding points are at distance 1. We fix a random polynomial data term $f_u: [-1, 1] \to \bR$ of degree $4$ at each vertex $u$ by fitting a random sample of data points. To obtain a high-accuracy solution we solve the primal SDP formulation corresponding to the saddle-point formulation~\cref{eq:decomposition_piecwise_moments} with MOSEK\footnote{\url{https://www.mosek.com/products/academic-licenses}}. For recovering a primal solution at each vertex $u$ we compute the mode w.r.t. the $0^{\text{th}}$ moments to select the best interval denoted by $k^* = \argmax_{1\leq k \leq K} (y_u)_{k, 0}$. Then we compute the mean of the discretized measure corresponding to the $(k^*)^{\text{th}}$ interval as $x_u = (y_u)_{k^*, 1}$.

\Cref{fig:rel_tightness} visualizes the primal and dual energies for varying degrees and/or number of pieces of the dual variable. While the dual energy strictly increases with higher degrees and/or number of pieces the primal energy is evaluated at the rounded solution and therefore does not strictly decrease in general. While for TV increasing the degree vs. increasing the number of pieces (for $K \cdot \deg$ constant) leads to similar performance, for Potts, in many situations, increasing the degree leads to larger dual energies, e.g., consider $\deg=4,K=1$ vs. $\deg=1,K=4$, red curve vs. blue curve in \cref{fig:rel_tightness}(d). In further experiments, we observed, that this holds in particular when the structure of the dual variables and the unaries match, i.e., $f_u, \lambda_{e} \in \Lambda$. Note that for Potts, since the dual variables are uniformly bounded on $\Omega$ and the derivative can be unbounded we drop the continuity constraint which leads to a more compact formulation and larger dual energies.
\begin{table}
\begin{center}
\begin{tabular}{c || c | c | c} 
\multicolumn{4}{c}{Dual energies} \\
\hline
$\deg$ & $K=1$ & $K=3$ & $K=5$ \\ [0.5ex] 
\hline\hline
$1$ & 14180.08 & 15733.71 & 16227.80 \\
$2$ & 15052.18 & 16430.97 & 16773.75 \\
$3$ & 15601.89 & 16778.87 & 17055.25 \\
$4$ & 15938.92 & 16998.63 & 17235.21 \\
$5$ & 16191.40 & 17147.92 & 17346.49 \\
$6$ & 16369.95 & 17243.47 & 17422.26 \\
$7$ & 16480.22 & 17308.91 & 17472.13
\end{tabular}
\qquad
 \begin{tabular}{c || c | c | c} 
  \multicolumn{4}{c}{Energies rounded} \\
  \hline
  $\deg$ & $K=1$ & $K=3$ & $K=5$ \\ [0.5ex] 
 \hline\hline
$1$ & 28982.45 & 25005.05 & 22283.26 \\
$2$ & 31038.14 & 22525.12 & 21049.32 \\
$3$ & 28505.41 & 21500.65 & 20428.34 \\
$4$ & 27255.48 & 20841.62 & 20049.04 \\
$5$ & 25795.56 & 20344.17 & 19764.23 \\
$6$ & 24081.51 & 20032.77 & 19559.00 \\
$7$ & 23142.67 & 19869.71 & 19428.50
\end{tabular}
\end{center}
\caption{Energies for stereo matching Motorcycle. Left: Dual energies. Right: Primal energies at the rounded solution.}
\label{tab:stereo_motorcycle}
\end{table}

\subsection{Stereo Matching} \label{sec:stereo_matching}
In this experiment we consider stereo matching using the anisotropic relaxation~\cref{eq:opt_relax}. We consider the Motorcycle and the Piano image pairs from the Middlebury stereo benchmark~\cite{scharstein2014high}. We downsample the images by factor $4$. The disparity cost term is first calculated using 135 discrete disparities obtained by shifting the images by the corresponding amount of pixels and comparing the image gradients. More specifically, given a RGB image $I$ mapping from $[1, \ldots, n_y] \times [1, \ldots, n_x]$ to a RGB value in $\bR^3$, the $x$ and $y$ derivatives are calculated as $I_x (i, j) = I(i, \min\{ j + 1, n_x\}) - I(i, j)$ and $I_y (i, j) = I(\min\{i + 1, n_y\}, j) - I(i, j)$. For a stereo image pair $(L, R)$ and a disparity $d \in \bN$ the cost at pixel $(i, j)$ is then calculated as
\begin{align*}
	D(i, j, d) = &\min\{\lVert L_x(i, \min\{j + d, n_x\}) - R_x(i, j) \rVert_1, 0.1\} + \\
	&\min\{\lVert L_y(i, \min\{j + d, n_x\}) - R_y(i, j) \rVert_1, 0.1\}.
\end{align*}
Then, the cost dataterm is approximated from below in terms of a continuous piecewise cubic polynomial $f_u : [1, 135] \to \bR$ using 30 pieces at each $u \in \cV$. This dataterm is precomputed once during a preprocessing and subsequently used as a benchmark for the different methods that we compare.
For the coupling, we use a total variation-like regularization, i.e., $f_{uv}(x_u, x_v)=\alpha|x_u - x_v|$ with weight $\alpha=0.2$.
In \cref{fig:stereo_motorcycle} we compare the standard MRF/OT discretization as described in \cref{ex:sampling} with our framework using piecewise linear and piecewise polynomial dual variables with degree 7 both with 5 pieces.
The standard MRF/OT discretization is equivalent to a piecewise linear approximation of the data term with piecewise linear duals in our framework. The piecewise linear approximation is obtained by sampling the piecewise cubic polynomial $f_u$ at the interval boundaries of the pieces. The reported energy for the solution of the standard approach in \cref{fig:stereo_motorcycle} is also evaluated using the piecewise linear cost. As the resulting optimization problem is large-scale we solve the saddle-point formulation~\cref{eq:saddle_point_pdhg} with PDHG~\cite{chambolle2011first} as described in \cref{sec:impl_pdhg} using the GPU-based PDHG framework prost\footnote{\url{https://github.com/tum-vision/prost}}. In contrast to the previous experiment which uses a combined mode and mean rounding procedure we found the plain mean of the discretized measure to produce better results on real data: More explicitly we recover a solution according to $x_u = \sum_{k=1}^K t_k (y_u)_{k,0}$ at each vertex $u \in \cV$.
In \cref{tab:stereo_motorcycle} we compare both, dual and nonconvex primal energies, for a larger hierarchy of dual subspaces.

\section{Discussion}
We presented a method to reduce duality gaps in the Lagrangian relaxation of the MAP-inference problem in a continuous MRF taking a nonlinear optimization-driven approach. Our theoretical contribution identifies \emph{extremality} of the lifting as a key component, as it leads to convexifications which do not discard any information for a wide range of nonconvex cost functions. Using results from convex algebraic geometry we provide, to our knowledge, the first tractable formulation of the polynomial discretization in terms of semidefinite programming. We have provided a parallel implementation of a first-order primal dual algorithm on a GPU which can handle large problems. Indeed, the approach of \cite{lasserre2001global,lasserre2002semidefinite} applied directly to the original problem~\cref{eq:mrf} (with $f_u, f_{uv}$ polynomial) attempts to solve the \emph{full} marginal polytope relaxation which is tight but intractable for large $\cV$ as the number of coupling moments explodes (grows like a polynomial $|\cV|^{\deg}$). In contrast, our framework applies to the \emph{local} marginal polytope relaxation which is not tight in general but leads to a tractable formulation as it exploits the sparse structure of the optimization problem.

As revealed by our experiments, both increasing the number of pieces and the degree of the dual variables successfully reduces the nonconvex duality gap.
In particular, the algorithm is applied to the stereo matching problem between two images, showing significant improvements over piecewise constant or piecewise linear discretizations in the dual which demand a higher number of samples.

For total variation regularized problems, our theory suggests that the duality gap vanishes like $\mathcal{O}(1 / (K \cdot \sqrt{\deg}))$. The faster convergence in the number of pieces $K$ is also confirmed by our experiments. However, setting the degree $\deg \geq 2$ has the attractive theoretical property that all information of the original nonconvex cost is preserved. For Potts regularization, our experiments indicate that increasing the degree leads to better results than increasing the pieces.

Going beyond the piecewise-linear setting might be particularly promising for vector- or manifold-valued label-spaces with $\dim(\Omega) > 1$.
Labeling problems in such higher-dimensional settings were recently considered in \cite{laude16eccv,mollenhoff2019lifting,vogt2020lifting}. There, due to the piecewise-linear dual variables, the complexity grows exponentially with $\dim(\Omega)$. Moreover, as discussed in \cite{vogt2020lifting}, the piecewise-approach inherently requires an approximation of $\Omega$ by a triangulated manifold. Building upon the techniques developed in this paper, an interesting direction for future work is to investigate whether these two drawbacks can be overcome by considering suitably defined approximation spaces for the dual variables in vector- or manifold-valued settings.

\section*{Acknowledgments}
We would like to thank Johannes Milz, Peter Ochs and Jan-Hendrik Lange for their valuable feedback on an early version of this manuscript.

%% file: figures/energiesl1.pgf
\begingroup%
\makeatletter%
\begin{pgfpicture}%
\pgfpathrectangle{\pgfpointorigin}{\pgfqpoint{3.400000in}{3.400000in}}%
\pgfusepath{use as bounding box, clip}%
\begin{pgfscope}%
\pgfsetbuttcap%
\pgfsetmiterjoin%
\definecolor{currentfill}{rgb}{1.000000,1.000000,1.000000}%
\pgfsetfillcolor{currentfill}%
\pgfsetlinewidth{0.000000pt}%
\definecolor{currentstroke}{rgb}{1.000000,1.000000,1.000000}%
\pgfsetstrokecolor{currentstroke}%
\pgfsetdash{}{0pt}%
\pgfpathmoveto{\pgfqpoint{0.000000in}{0.000000in}}%
\pgfpathlineto{\pgfqpoint{3.400000in}{0.000000in}}%
\pgfpathlineto{\pgfqpoint{3.400000in}{3.400000in}}%
\pgfpathlineto{\pgfqpoint{0.000000in}{3.400000in}}%
\pgfpathclose%
\pgfusepath{fill}%
\end{pgfscope}%
\begin{pgfscope}%
\pgfsetbuttcap%
\pgfsetmiterjoin%
\definecolor{currentfill}{rgb}{1.000000,1.000000,1.000000}%
\pgfsetfillcolor{currentfill}%
\pgfsetlinewidth{0.000000pt}%
\definecolor{currentstroke}{rgb}{0.000000,0.000000,0.000000}%
\pgfsetstrokecolor{currentstroke}%
\pgfsetstrokeopacity{0.000000}%
\pgfsetdash{}{0pt}%
\pgfpathmoveto{\pgfqpoint{0.692593in}{0.499691in}}%
\pgfpathlineto{\pgfqpoint{3.300000in}{0.499691in}}%
\pgfpathlineto{\pgfqpoint{3.300000in}{3.100926in}}%
\pgfpathlineto{\pgfqpoint{0.692593in}{3.100926in}}%
\pgfpathclose%
\pgfusepath{fill}%
\end{pgfscope}%
\begin{pgfscope}%
\pgfsetbuttcap%
\pgfsetroundjoin%
\definecolor{currentfill}{rgb}{0.000000,0.000000,0.000000}%
\pgfsetfillcolor{currentfill}%
\pgfsetlinewidth{0.803000pt}%
\definecolor{currentstroke}{rgb}{0.000000,0.000000,0.000000}%
\pgfsetstrokecolor{currentstroke}%
\pgfsetdash{}{0pt}%
\pgfsys@defobject{currentmarker}{\pgfqpoint{0.000000in}{-0.048611in}}{\pgfqpoint{0.000000in}{0.000000in}}{%
\pgfpathmoveto{\pgfqpoint{0.000000in}{0.000000in}}%
\pgfpathlineto{\pgfqpoint{0.000000in}{-0.048611in}}%
\pgfusepath{stroke,fill}%
}%
\begin{pgfscope}%
\pgfsys@transformshift{0.811112in}{0.499691in}%
\pgfsys@useobject{currentmarker}{}%
\end{pgfscope}%
\end{pgfscope}%
\begin{pgfscope}%
\definecolor{textcolor}{rgb}{0.000000,0.000000,0.000000}%
\pgfsetstrokecolor{textcolor}%
\pgfsetfillcolor{textcolor}%
\pgftext[x=0.811112in,y=0.402469in,,top]{\color{textcolor}\rmfamily\fontsize{10.000000}{12.000000}\selectfont \(\displaystyle 1\)}%
\end{pgfscope}%
\begin{pgfscope}%
\pgfsetbuttcap%
\pgfsetroundjoin%
\definecolor{currentfill}{rgb}{0.000000,0.000000,0.000000}%
\pgfsetfillcolor{currentfill}%
\pgfsetlinewidth{0.803000pt}%
\definecolor{currentstroke}{rgb}{0.000000,0.000000,0.000000}%
\pgfsetstrokecolor{currentstroke}%
\pgfsetdash{}{0pt}%
\pgfsys@defobject{currentmarker}{\pgfqpoint{0.000000in}{-0.048611in}}{\pgfqpoint{0.000000in}{0.000000in}}{%
\pgfpathmoveto{\pgfqpoint{0.000000in}{0.000000in}}%
\pgfpathlineto{\pgfqpoint{0.000000in}{-0.048611in}}%
\pgfusepath{stroke,fill}%
}%
\begin{pgfscope}%
\pgfsys@transformshift{1.601235in}{0.499691in}%
\pgfsys@useobject{currentmarker}{}%
\end{pgfscope}%
\end{pgfscope}%
\begin{pgfscope}%
\definecolor{textcolor}{rgb}{0.000000,0.000000,0.000000}%
\pgfsetstrokecolor{textcolor}%
\pgfsetfillcolor{textcolor}%
\pgftext[x=1.601235in,y=0.402469in,,top]{\color{textcolor}\rmfamily\fontsize{10.000000}{12.000000}\selectfont \(\displaystyle 2\)}%
\end{pgfscope}%
\begin{pgfscope}%
\pgfsetbuttcap%
\pgfsetroundjoin%
\definecolor{currentfill}{rgb}{0.000000,0.000000,0.000000}%
\pgfsetfillcolor{currentfill}%
\pgfsetlinewidth{0.803000pt}%
\definecolor{currentstroke}{rgb}{0.000000,0.000000,0.000000}%
\pgfsetstrokecolor{currentstroke}%
\pgfsetdash{}{0pt}%
\pgfsys@defobject{currentmarker}{\pgfqpoint{0.000000in}{-0.048611in}}{\pgfqpoint{0.000000in}{0.000000in}}{%
\pgfpathmoveto{\pgfqpoint{0.000000in}{0.000000in}}%
\pgfpathlineto{\pgfqpoint{0.000000in}{-0.048611in}}%
\pgfusepath{stroke,fill}%
}%
\begin{pgfscope}%
\pgfsys@transformshift{2.391358in}{0.499691in}%
\pgfsys@useobject{currentmarker}{}%
\end{pgfscope}%
\end{pgfscope}%
\begin{pgfscope}%
\definecolor{textcolor}{rgb}{0.000000,0.000000,0.000000}%
\pgfsetstrokecolor{textcolor}%
\pgfsetfillcolor{textcolor}%
\pgftext[x=2.391358in,y=0.402469in,,top]{\color{textcolor}\rmfamily\fontsize{10.000000}{12.000000}\selectfont \(\displaystyle 3\)}%
\end{pgfscope}%
\begin{pgfscope}%
\pgfsetbuttcap%
\pgfsetroundjoin%
\definecolor{currentfill}{rgb}{0.000000,0.000000,0.000000}%
\pgfsetfillcolor{currentfill}%
\pgfsetlinewidth{0.803000pt}%
\definecolor{currentstroke}{rgb}{0.000000,0.000000,0.000000}%
\pgfsetstrokecolor{currentstroke}%
\pgfsetdash{}{0pt}%
\pgfsys@defobject{currentmarker}{\pgfqpoint{0.000000in}{-0.048611in}}{\pgfqpoint{0.000000in}{0.000000in}}{%
\pgfpathmoveto{\pgfqpoint{0.000000in}{0.000000in}}%
\pgfpathlineto{\pgfqpoint{0.000000in}{-0.048611in}}%
\pgfusepath{stroke,fill}%
}%
\begin{pgfscope}%
\pgfsys@transformshift{3.181482in}{0.499691in}%
\pgfsys@useobject{currentmarker}{}%
\end{pgfscope}%
\end{pgfscope}%
\begin{pgfscope}%
\definecolor{textcolor}{rgb}{0.000000,0.000000,0.000000}%
\pgfsetstrokecolor{textcolor}%
\pgfsetfillcolor{textcolor}%
\pgftext[x=3.181482in,y=0.402469in,,top]{\color{textcolor}\rmfamily\fontsize{10.000000}{12.000000}\selectfont \(\displaystyle 4\)}%
\end{pgfscope}%
\begin{pgfscope}%
\definecolor{textcolor}{rgb}{0.000000,0.000000,0.000000}%
\pgfsetstrokecolor{textcolor}%
\pgfsetfillcolor{textcolor}%
\pgftext[x=1.996297in,y=0.223457in,,top]{\color{textcolor}\rmfamily\fontsize{10.000000}{12.000000}\selectfont Number of pieces \(\displaystyle K\)}%
\end{pgfscope}%
\begin{pgfscope}%
\pgfsetbuttcap%
\pgfsetroundjoin%
\definecolor{currentfill}{rgb}{0.000000,0.000000,0.000000}%
\pgfsetfillcolor{currentfill}%
\pgfsetlinewidth{0.803000pt}%
\definecolor{currentstroke}{rgb}{0.000000,0.000000,0.000000}%
\pgfsetstrokecolor{currentstroke}%
\pgfsetdash{}{0pt}%
\pgfsys@defobject{currentmarker}{\pgfqpoint{-0.048611in}{0.000000in}}{\pgfqpoint{0.000000in}{0.000000in}}{%
\pgfpathmoveto{\pgfqpoint{0.000000in}{0.000000in}}%
\pgfpathlineto{\pgfqpoint{-0.048611in}{0.000000in}}%
\pgfusepath{stroke,fill}%
}%
\begin{pgfscope}%
\pgfsys@transformshift{0.692593in}{0.904870in}%
\pgfsys@useobject{currentmarker}{}%
\end{pgfscope}%
\end{pgfscope}%
\begin{pgfscope}%
\definecolor{textcolor}{rgb}{0.000000,0.000000,0.000000}%
\pgfsetstrokecolor{textcolor}%
\pgfsetfillcolor{textcolor}%
\pgftext[x=0.279012in,y=0.856645in,left,base]{\color{textcolor}\rmfamily\fontsize{10.000000}{12.000000}\selectfont \(\displaystyle -240\)}%
\end{pgfscope}%
\begin{pgfscope}%
\pgfsetbuttcap%
\pgfsetroundjoin%
\definecolor{currentfill}{rgb}{0.000000,0.000000,0.000000}%
\pgfsetfillcolor{currentfill}%
\pgfsetlinewidth{0.803000pt}%
\definecolor{currentstroke}{rgb}{0.000000,0.000000,0.000000}%
\pgfsetstrokecolor{currentstroke}%
\pgfsetdash{}{0pt}%
\pgfsys@defobject{currentmarker}{\pgfqpoint{-0.048611in}{0.000000in}}{\pgfqpoint{0.000000in}{0.000000in}}{%
\pgfpathmoveto{\pgfqpoint{0.000000in}{0.000000in}}%
\pgfpathlineto{\pgfqpoint{-0.048611in}{0.000000in}}%
\pgfusepath{stroke,fill}%
}%
\begin{pgfscope}%
\pgfsys@transformshift{0.692593in}{1.313949in}%
\pgfsys@useobject{currentmarker}{}%
\end{pgfscope}%
\end{pgfscope}%
\begin{pgfscope}%
\definecolor{textcolor}{rgb}{0.000000,0.000000,0.000000}%
\pgfsetstrokecolor{textcolor}%
\pgfsetfillcolor{textcolor}%
\pgftext[x=0.279012in,y=1.265724in,left,base]{\color{textcolor}\rmfamily\fontsize{10.000000}{12.000000}\selectfont \(\displaystyle -220\)}%
\end{pgfscope}%
\begin{pgfscope}%
\pgfsetbuttcap%
\pgfsetroundjoin%
\definecolor{currentfill}{rgb}{0.000000,0.000000,0.000000}%
\pgfsetfillcolor{currentfill}%
\pgfsetlinewidth{0.803000pt}%
\definecolor{currentstroke}{rgb}{0.000000,0.000000,0.000000}%
\pgfsetstrokecolor{currentstroke}%
\pgfsetdash{}{0pt}%
\pgfsys@defobject{currentmarker}{\pgfqpoint{-0.048611in}{0.000000in}}{\pgfqpoint{0.000000in}{0.000000in}}{%
\pgfpathmoveto{\pgfqpoint{0.000000in}{0.000000in}}%
\pgfpathlineto{\pgfqpoint{-0.048611in}{0.000000in}}%
\pgfusepath{stroke,fill}%
}%
\begin{pgfscope}%
\pgfsys@transformshift{0.692593in}{1.723028in}%
\pgfsys@useobject{currentmarker}{}%
\end{pgfscope}%
\end{pgfscope}%
\begin{pgfscope}%
\definecolor{textcolor}{rgb}{0.000000,0.000000,0.000000}%
\pgfsetstrokecolor{textcolor}%
\pgfsetfillcolor{textcolor}%
\pgftext[x=0.279012in,y=1.674803in,left,base]{\color{textcolor}\rmfamily\fontsize{10.000000}{12.000000}\selectfont \(\displaystyle -200\)}%
\end{pgfscope}%
\begin{pgfscope}%
\pgfsetbuttcap%
\pgfsetroundjoin%
\definecolor{currentfill}{rgb}{0.000000,0.000000,0.000000}%
\pgfsetfillcolor{currentfill}%
\pgfsetlinewidth{0.803000pt}%
\definecolor{currentstroke}{rgb}{0.000000,0.000000,0.000000}%
\pgfsetstrokecolor{currentstroke}%
\pgfsetdash{}{0pt}%
\pgfsys@defobject{currentmarker}{\pgfqpoint{-0.048611in}{0.000000in}}{\pgfqpoint{0.000000in}{0.000000in}}{%
\pgfpathmoveto{\pgfqpoint{0.000000in}{0.000000in}}%
\pgfpathlineto{\pgfqpoint{-0.048611in}{0.000000in}}%
\pgfusepath{stroke,fill}%
}%
\begin{pgfscope}%
\pgfsys@transformshift{0.692593in}{2.132107in}%
\pgfsys@useobject{currentmarker}{}%
\end{pgfscope}%
\end{pgfscope}%
\begin{pgfscope}%
\definecolor{textcolor}{rgb}{0.000000,0.000000,0.000000}%
\pgfsetstrokecolor{textcolor}%
\pgfsetfillcolor{textcolor}%
\pgftext[x=0.279012in,y=2.083881in,left,base]{\color{textcolor}\rmfamily\fontsize{10.000000}{12.000000}\selectfont \(\displaystyle -180\)}%
\end{pgfscope}%
\begin{pgfscope}%
\pgfsetbuttcap%
\pgfsetroundjoin%
\definecolor{currentfill}{rgb}{0.000000,0.000000,0.000000}%
\pgfsetfillcolor{currentfill}%
\pgfsetlinewidth{0.803000pt}%
\definecolor{currentstroke}{rgb}{0.000000,0.000000,0.000000}%
\pgfsetstrokecolor{currentstroke}%
\pgfsetdash{}{0pt}%
\pgfsys@defobject{currentmarker}{\pgfqpoint{-0.048611in}{0.000000in}}{\pgfqpoint{0.000000in}{0.000000in}}{%
\pgfpathmoveto{\pgfqpoint{0.000000in}{0.000000in}}%
\pgfpathlineto{\pgfqpoint{-0.048611in}{0.000000in}}%
\pgfusepath{stroke,fill}%
}%
\begin{pgfscope}%
\pgfsys@transformshift{0.692593in}{2.541186in}%
\pgfsys@useobject{currentmarker}{}%
\end{pgfscope}%
\end{pgfscope}%
\begin{pgfscope}%
\definecolor{textcolor}{rgb}{0.000000,0.000000,0.000000}%
\pgfsetstrokecolor{textcolor}%
\pgfsetfillcolor{textcolor}%
\pgftext[x=0.279012in,y=2.492960in,left,base]{\color{textcolor}\rmfamily\fontsize{10.000000}{12.000000}\selectfont \(\displaystyle -160\)}%
\end{pgfscope}%
\begin{pgfscope}%
\pgfsetbuttcap%
\pgfsetroundjoin%
\definecolor{currentfill}{rgb}{0.000000,0.000000,0.000000}%
\pgfsetfillcolor{currentfill}%
\pgfsetlinewidth{0.803000pt}%
\definecolor{currentstroke}{rgb}{0.000000,0.000000,0.000000}%
\pgfsetstrokecolor{currentstroke}%
\pgfsetdash{}{0pt}%
\pgfsys@defobject{currentmarker}{\pgfqpoint{-0.048611in}{0.000000in}}{\pgfqpoint{0.000000in}{0.000000in}}{%
\pgfpathmoveto{\pgfqpoint{0.000000in}{0.000000in}}%
\pgfpathlineto{\pgfqpoint{-0.048611in}{0.000000in}}%
\pgfusepath{stroke,fill}%
}%
\begin{pgfscope}%
\pgfsys@transformshift{0.692593in}{2.950265in}%
\pgfsys@useobject{currentmarker}{}%
\end{pgfscope}%
\end{pgfscope}%
\begin{pgfscope}%
\definecolor{textcolor}{rgb}{0.000000,0.000000,0.000000}%
\pgfsetstrokecolor{textcolor}%
\pgfsetfillcolor{textcolor}%
\pgftext[x=0.279012in,y=2.902039in,left,base]{\color{textcolor}\rmfamily\fontsize{10.000000}{12.000000}\selectfont \(\displaystyle -140\)}%
\end{pgfscope}%
\begin{pgfscope}%
\definecolor{textcolor}{rgb}{0.000000,0.000000,0.000000}%
\pgfsetstrokecolor{textcolor}%
\pgfsetfillcolor{textcolor}%
\pgftext[x=0.223457in,y=1.800309in,,bottom,rotate=90.000000]{\color{textcolor}\rmfamily\fontsize{10.000000}{12.000000}\selectfont Energy}%
\end{pgfscope}%
\begin{pgfscope}%
\pgfpathrectangle{\pgfqpoint{0.692593in}{0.499691in}}{\pgfqpoint{2.607407in}{2.601235in}}%
\pgfusepath{clip}%
\pgfsetrectcap%
\pgfsetroundjoin%
\pgfsetlinewidth{1.505625pt}%
\definecolor{currentstroke}{rgb}{0.121569,0.466667,0.705882}%
\pgfsetstrokecolor{currentstroke}%
\pgfsetdash{}{0pt}%
\pgfpathmoveto{\pgfqpoint{0.811112in}{0.617929in}}%
\pgfpathlineto{\pgfqpoint{1.601235in}{1.813405in}}%
\pgfpathlineto{\pgfqpoint{2.391358in}{2.161938in}}%
\pgfpathlineto{\pgfqpoint{3.181482in}{2.616544in}}%
\pgfusepath{stroke}%
\end{pgfscope}%
\begin{pgfscope}%
\pgfpathrectangle{\pgfqpoint{0.692593in}{0.499691in}}{\pgfqpoint{2.607407in}{2.601235in}}%
\pgfusepath{clip}%
\pgfsetbuttcap%
\pgfsetroundjoin%
\definecolor{currentfill}{rgb}{0.121569,0.466667,0.705882}%
\pgfsetfillcolor{currentfill}%
\pgfsetlinewidth{1.003750pt}%
\definecolor{currentstroke}{rgb}{0.121569,0.466667,0.705882}%
\pgfsetstrokecolor{currentstroke}%
\pgfsetdash{}{0pt}%
\pgfsys@defobject{currentmarker}{\pgfqpoint{-0.041667in}{-0.041667in}}{\pgfqpoint{0.041667in}{0.041667in}}{%
\pgfpathmoveto{\pgfqpoint{-0.041667in}{-0.041667in}}%
\pgfpathlineto{\pgfqpoint{0.041667in}{0.041667in}}%
\pgfpathmoveto{\pgfqpoint{-0.041667in}{0.041667in}}%
\pgfpathlineto{\pgfqpoint{0.041667in}{-0.041667in}}%
\pgfusepath{stroke,fill}%
}%
\begin{pgfscope}%
\pgfsys@transformshift{0.811112in}{0.617929in}%
\pgfsys@useobject{currentmarker}{}%
\end{pgfscope}%
\begin{pgfscope}%
\pgfsys@transformshift{1.601235in}{1.813405in}%
\pgfsys@useobject{currentmarker}{}%
\end{pgfscope}%
\begin{pgfscope}%
\pgfsys@transformshift{2.391358in}{2.161938in}%
\pgfsys@useobject{currentmarker}{}%
\end{pgfscope}%
\begin{pgfscope}%
\pgfsys@transformshift{3.181482in}{2.616544in}%
\pgfsys@useobject{currentmarker}{}%
\end{pgfscope}%
\end{pgfscope}%
\begin{pgfscope}%
\pgfpathrectangle{\pgfqpoint{0.692593in}{0.499691in}}{\pgfqpoint{2.607407in}{2.601235in}}%
\pgfusepath{clip}%
\pgfsetrectcap%
\pgfsetroundjoin%
\pgfsetlinewidth{1.505625pt}%
\definecolor{currentstroke}{rgb}{1.000000,0.498039,0.054902}%
\pgfsetstrokecolor{currentstroke}%
\pgfsetdash{}{0pt}%
\pgfpathmoveto{\pgfqpoint{0.811112in}{1.433340in}}%
\pgfpathlineto{\pgfqpoint{1.601235in}{2.391685in}}%
\pgfpathlineto{\pgfqpoint{2.391358in}{2.645651in}}%
\pgfpathlineto{\pgfqpoint{3.181482in}{2.852038in}}%
\pgfusepath{stroke}%
\end{pgfscope}%
\begin{pgfscope}%
\pgfpathrectangle{\pgfqpoint{0.692593in}{0.499691in}}{\pgfqpoint{2.607407in}{2.601235in}}%
\pgfusepath{clip}%
\pgfsetbuttcap%
\pgfsetroundjoin%
\definecolor{currentfill}{rgb}{1.000000,0.498039,0.054902}%
\pgfsetfillcolor{currentfill}%
\pgfsetlinewidth{1.003750pt}%
\definecolor{currentstroke}{rgb}{1.000000,0.498039,0.054902}%
\pgfsetstrokecolor{currentstroke}%
\pgfsetdash{}{0pt}%
\pgfsys@defobject{currentmarker}{\pgfqpoint{-0.041667in}{-0.041667in}}{\pgfqpoint{0.041667in}{0.041667in}}{%
\pgfpathmoveto{\pgfqpoint{-0.041667in}{-0.041667in}}%
\pgfpathlineto{\pgfqpoint{0.041667in}{0.041667in}}%
\pgfpathmoveto{\pgfqpoint{-0.041667in}{0.041667in}}%
\pgfpathlineto{\pgfqpoint{0.041667in}{-0.041667in}}%
\pgfusepath{stroke,fill}%
}%
\begin{pgfscope}%
\pgfsys@transformshift{0.811112in}{1.433340in}%
\pgfsys@useobject{currentmarker}{}%
\end{pgfscope}%
\begin{pgfscope}%
\pgfsys@transformshift{1.601235in}{2.391685in}%
\pgfsys@useobject{currentmarker}{}%
\end{pgfscope}%
\begin{pgfscope}%
\pgfsys@transformshift{2.391358in}{2.645651in}%
\pgfsys@useobject{currentmarker}{}%
\end{pgfscope}%
\begin{pgfscope}%
\pgfsys@transformshift{3.181482in}{2.852038in}%
\pgfsys@useobject{currentmarker}{}%
\end{pgfscope}%
\end{pgfscope}%
\begin{pgfscope}%
\pgfpathrectangle{\pgfqpoint{0.692593in}{0.499691in}}{\pgfqpoint{2.607407in}{2.601235in}}%
\pgfusepath{clip}%
\pgfsetrectcap%
\pgfsetroundjoin%
\pgfsetlinewidth{1.505625pt}%
\definecolor{currentstroke}{rgb}{0.172549,0.627451,0.172549}%
\pgfsetstrokecolor{currentstroke}%
\pgfsetdash{}{0pt}%
\pgfpathmoveto{\pgfqpoint{0.811112in}{1.888239in}}%
\pgfpathlineto{\pgfqpoint{1.601235in}{2.705833in}}%
\pgfpathlineto{\pgfqpoint{2.391358in}{2.816422in}}%
\pgfpathlineto{\pgfqpoint{3.181482in}{2.944858in}}%
\pgfusepath{stroke}%
\end{pgfscope}%
\begin{pgfscope}%
\pgfpathrectangle{\pgfqpoint{0.692593in}{0.499691in}}{\pgfqpoint{2.607407in}{2.601235in}}%
\pgfusepath{clip}%
\pgfsetbuttcap%
\pgfsetroundjoin%
\definecolor{currentfill}{rgb}{0.172549,0.627451,0.172549}%
\pgfsetfillcolor{currentfill}%
\pgfsetlinewidth{1.003750pt}%
\definecolor{currentstroke}{rgb}{0.172549,0.627451,0.172549}%
\pgfsetstrokecolor{currentstroke}%
\pgfsetdash{}{0pt}%
\pgfsys@defobject{currentmarker}{\pgfqpoint{-0.041667in}{-0.041667in}}{\pgfqpoint{0.041667in}{0.041667in}}{%
\pgfpathmoveto{\pgfqpoint{-0.041667in}{-0.041667in}}%
\pgfpathlineto{\pgfqpoint{0.041667in}{0.041667in}}%
\pgfpathmoveto{\pgfqpoint{-0.041667in}{0.041667in}}%
\pgfpathlineto{\pgfqpoint{0.041667in}{-0.041667in}}%
\pgfusepath{stroke,fill}%
}%
\begin{pgfscope}%
\pgfsys@transformshift{0.811112in}{1.888239in}%
\pgfsys@useobject{currentmarker}{}%
\end{pgfscope}%
\begin{pgfscope}%
\pgfsys@transformshift{1.601235in}{2.705833in}%
\pgfsys@useobject{currentmarker}{}%
\end{pgfscope}%
\begin{pgfscope}%
\pgfsys@transformshift{2.391358in}{2.816422in}%
\pgfsys@useobject{currentmarker}{}%
\end{pgfscope}%
\begin{pgfscope}%
\pgfsys@transformshift{3.181482in}{2.944858in}%
\pgfsys@useobject{currentmarker}{}%
\end{pgfscope}%
\end{pgfscope}%
\begin{pgfscope}%
\pgfpathrectangle{\pgfqpoint{0.692593in}{0.499691in}}{\pgfqpoint{2.607407in}{2.601235in}}%
\pgfusepath{clip}%
\pgfsetrectcap%
\pgfsetroundjoin%
\pgfsetlinewidth{1.505625pt}%
\definecolor{currentstroke}{rgb}{0.839216,0.152941,0.156863}%
\pgfsetstrokecolor{currentstroke}%
\pgfsetdash{}{0pt}%
\pgfpathmoveto{\pgfqpoint{0.811112in}{2.449252in}}%
\pgfpathlineto{\pgfqpoint{1.601235in}{2.844855in}}%
\pgfpathlineto{\pgfqpoint{2.391358in}{2.921755in}}%
\pgfpathlineto{\pgfqpoint{3.181482in}{2.982688in}}%
\pgfusepath{stroke}%
\end{pgfscope}%
\begin{pgfscope}%
\pgfpathrectangle{\pgfqpoint{0.692593in}{0.499691in}}{\pgfqpoint{2.607407in}{2.601235in}}%
\pgfusepath{clip}%
\pgfsetbuttcap%
\pgfsetroundjoin%
\definecolor{currentfill}{rgb}{0.839216,0.152941,0.156863}%
\pgfsetfillcolor{currentfill}%
\pgfsetlinewidth{1.003750pt}%
\definecolor{currentstroke}{rgb}{0.839216,0.152941,0.156863}%
\pgfsetstrokecolor{currentstroke}%
\pgfsetdash{}{0pt}%
\pgfsys@defobject{currentmarker}{\pgfqpoint{-0.041667in}{-0.041667in}}{\pgfqpoint{0.041667in}{0.041667in}}{%
\pgfpathmoveto{\pgfqpoint{-0.041667in}{-0.041667in}}%
\pgfpathlineto{\pgfqpoint{0.041667in}{0.041667in}}%
\pgfpathmoveto{\pgfqpoint{-0.041667in}{0.041667in}}%
\pgfpathlineto{\pgfqpoint{0.041667in}{-0.041667in}}%
\pgfusepath{stroke,fill}%
}%
\begin{pgfscope}%
\pgfsys@transformshift{0.811112in}{2.449252in}%
\pgfsys@useobject{currentmarker}{}%
\end{pgfscope}%
\begin{pgfscope}%
\pgfsys@transformshift{1.601235in}{2.844855in}%
\pgfsys@useobject{currentmarker}{}%
\end{pgfscope}%
\begin{pgfscope}%
\pgfsys@transformshift{2.391358in}{2.921755in}%
\pgfsys@useobject{currentmarker}{}%
\end{pgfscope}%
\begin{pgfscope}%
\pgfsys@transformshift{3.181482in}{2.982688in}%
\pgfsys@useobject{currentmarker}{}%
\end{pgfscope}%
\end{pgfscope}%
\begin{pgfscope}%
\pgfsetrectcap%
\pgfsetmiterjoin%
\pgfsetlinewidth{0.803000pt}%
\definecolor{currentstroke}{rgb}{0.000000,0.000000,0.000000}%
\pgfsetstrokecolor{currentstroke}%
\pgfsetdash{}{0pt}%
\pgfpathmoveto{\pgfqpoint{0.692593in}{0.499691in}}%
\pgfpathlineto{\pgfqpoint{0.692593in}{3.100926in}}%
\pgfusepath{stroke}%
\end{pgfscope}%
\begin{pgfscope}%
\pgfsetrectcap%
\pgfsetmiterjoin%
\pgfsetlinewidth{0.803000pt}%
\definecolor{currentstroke}{rgb}{0.000000,0.000000,0.000000}%
\pgfsetstrokecolor{currentstroke}%
\pgfsetdash{}{0pt}%
\pgfpathmoveto{\pgfqpoint{3.300000in}{0.499691in}}%
\pgfpathlineto{\pgfqpoint{3.300000in}{3.100926in}}%
\pgfusepath{stroke}%
\end{pgfscope}%
\begin{pgfscope}%
\pgfsetrectcap%
\pgfsetmiterjoin%
\pgfsetlinewidth{0.803000pt}%
\definecolor{currentstroke}{rgb}{0.000000,0.000000,0.000000}%
\pgfsetstrokecolor{currentstroke}%
\pgfsetdash{}{0pt}%
\pgfpathmoveto{\pgfqpoint{0.692593in}{0.499691in}}%
\pgfpathlineto{\pgfqpoint{3.300000in}{0.499691in}}%
\pgfusepath{stroke}%
\end{pgfscope}%
\begin{pgfscope}%
\pgfsetrectcap%
\pgfsetmiterjoin%
\pgfsetlinewidth{0.803000pt}%
\definecolor{currentstroke}{rgb}{0.000000,0.000000,0.000000}%
\pgfsetstrokecolor{currentstroke}%
\pgfsetdash{}{0pt}%
\pgfpathmoveto{\pgfqpoint{0.692593in}{3.100926in}}%
\pgfpathlineto{\pgfqpoint{3.300000in}{3.100926in}}%
\pgfusepath{stroke}%
\end{pgfscope}%
\begin{pgfscope}%
\definecolor{textcolor}{rgb}{0.000000,0.000000,0.000000}%
\pgfsetstrokecolor{textcolor}%
\pgfsetfillcolor{textcolor}%
\pgftext[x=1.996297in,y=3.184260in,,base]{\color{textcolor}\rmfamily\fontsize{12.000000}{14.400000}\selectfont TV regularization}%
\end{pgfscope}%
\begin{pgfscope}%
\pgfsetbuttcap%
\pgfsetmiterjoin%
\definecolor{currentfill}{rgb}{1.000000,1.000000,1.000000}%
\pgfsetfillcolor{currentfill}%
\pgfsetfillopacity{0.800000}%
\pgfsetlinewidth{1.003750pt}%
\definecolor{currentstroke}{rgb}{0.800000,0.800000,0.800000}%
\pgfsetstrokecolor{currentstroke}%
\pgfsetstrokeopacity{0.800000}%
\pgfsetdash{}{0pt}%
\pgfpathmoveto{\pgfqpoint{2.333951in}{0.569136in}}%
\pgfpathlineto{\pgfqpoint{3.202778in}{0.569136in}}%
\pgfpathquadraticcurveto{\pgfqpoint{3.230556in}{0.569136in}}{\pgfqpoint{3.230556in}{0.596913in}}%
\pgfpathlineto{\pgfqpoint{3.230556in}{1.357716in}}%
\pgfpathquadraticcurveto{\pgfqpoint{3.230556in}{1.385493in}}{\pgfqpoint{3.202778in}{1.385493in}}%
\pgfpathlineto{\pgfqpoint{2.333951in}{1.385493in}}%
\pgfpathquadraticcurveto{\pgfqpoint{2.306173in}{1.385493in}}{\pgfqpoint{2.306173in}{1.357716in}}%
\pgfpathlineto{\pgfqpoint{2.306173in}{0.596913in}}%
\pgfpathquadraticcurveto{\pgfqpoint{2.306173in}{0.569136in}}{\pgfqpoint{2.333951in}{0.569136in}}%
\pgfpathclose%
\pgfusepath{stroke,fill}%
\end{pgfscope}%
\begin{pgfscope}%
\pgfsetrectcap%
\pgfsetroundjoin%
\pgfsetlinewidth{1.505625pt}%
\definecolor{currentstroke}{rgb}{0.121569,0.466667,0.705882}%
\pgfsetstrokecolor{currentstroke}%
\pgfsetdash{}{0pt}%
\pgfpathmoveto{\pgfqpoint{2.361728in}{1.281327in}}%
\pgfpathlineto{\pgfqpoint{2.639506in}{1.281327in}}%
\pgfusepath{stroke}%
\end{pgfscope}%
\begin{pgfscope}%
\pgfsetbuttcap%
\pgfsetroundjoin%
\definecolor{currentfill}{rgb}{0.121569,0.466667,0.705882}%
\pgfsetfillcolor{currentfill}%
\pgfsetlinewidth{1.003750pt}%
\definecolor{currentstroke}{rgb}{0.121569,0.466667,0.705882}%
\pgfsetstrokecolor{currentstroke}%
\pgfsetdash{}{0pt}%
\pgfsys@defobject{currentmarker}{\pgfqpoint{-0.041667in}{-0.041667in}}{\pgfqpoint{0.041667in}{0.041667in}}{%
\pgfpathmoveto{\pgfqpoint{-0.041667in}{-0.041667in}}%
\pgfpathlineto{\pgfqpoint{0.041667in}{0.041667in}}%
\pgfpathmoveto{\pgfqpoint{-0.041667in}{0.041667in}}%
\pgfpathlineto{\pgfqpoint{0.041667in}{-0.041667in}}%
\pgfusepath{stroke,fill}%
}%
\begin{pgfscope}%
\pgfsys@transformshift{2.500617in}{1.281327in}%
\pgfsys@useobject{currentmarker}{}%
\end{pgfscope}%
\end{pgfscope}%
\begin{pgfscope}%
\definecolor{textcolor}{rgb}{0.000000,0.000000,0.000000}%
\pgfsetstrokecolor{textcolor}%
\pgfsetfillcolor{textcolor}%
\pgftext[x=2.750617in,y=1.232716in,left,base]{\color{textcolor}\rmfamily\fontsize{10.000000}{12.000000}\selectfont deg\(\displaystyle =1\)}%
\end{pgfscope}%
\begin{pgfscope}%
\pgfsetrectcap%
\pgfsetroundjoin%
\pgfsetlinewidth{1.505625pt}%
\definecolor{currentstroke}{rgb}{1.000000,0.498039,0.054902}%
\pgfsetstrokecolor{currentstroke}%
\pgfsetdash{}{0pt}%
\pgfpathmoveto{\pgfqpoint{2.361728in}{1.087654in}}%
\pgfpathlineto{\pgfqpoint{2.639506in}{1.087654in}}%
\pgfusepath{stroke}%
\end{pgfscope}%
\begin{pgfscope}%
\pgfsetbuttcap%
\pgfsetroundjoin%
\definecolor{currentfill}{rgb}{1.000000,0.498039,0.054902}%
\pgfsetfillcolor{currentfill}%
\pgfsetlinewidth{1.003750pt}%
\definecolor{currentstroke}{rgb}{1.000000,0.498039,0.054902}%
\pgfsetstrokecolor{currentstroke}%
\pgfsetdash{}{0pt}%
\pgfsys@defobject{currentmarker}{\pgfqpoint{-0.041667in}{-0.041667in}}{\pgfqpoint{0.041667in}{0.041667in}}{%
\pgfpathmoveto{\pgfqpoint{-0.041667in}{-0.041667in}}%
\pgfpathlineto{\pgfqpoint{0.041667in}{0.041667in}}%
\pgfpathmoveto{\pgfqpoint{-0.041667in}{0.041667in}}%
\pgfpathlineto{\pgfqpoint{0.041667in}{-0.041667in}}%
\pgfusepath{stroke,fill}%
}%
\begin{pgfscope}%
\pgfsys@transformshift{2.500617in}{1.087654in}%
\pgfsys@useobject{currentmarker}{}%
\end{pgfscope}%
\end{pgfscope}%
\begin{pgfscope}%
\definecolor{textcolor}{rgb}{0.000000,0.000000,0.000000}%
\pgfsetstrokecolor{textcolor}%
\pgfsetfillcolor{textcolor}%
\pgftext[x=2.750617in,y=1.039043in,left,base]{\color{textcolor}\rmfamily\fontsize{10.000000}{12.000000}\selectfont deg\(\displaystyle =2\)}%
\end{pgfscope}%
\begin{pgfscope}%
\pgfsetrectcap%
\pgfsetroundjoin%
\pgfsetlinewidth{1.505625pt}%
\definecolor{currentstroke}{rgb}{0.172549,0.627451,0.172549}%
\pgfsetstrokecolor{currentstroke}%
\pgfsetdash{}{0pt}%
\pgfpathmoveto{\pgfqpoint{2.361728in}{0.893981in}}%
\pgfpathlineto{\pgfqpoint{2.639506in}{0.893981in}}%
\pgfusepath{stroke}%
\end{pgfscope}%
\begin{pgfscope}%
\pgfsetbuttcap%
\pgfsetroundjoin%
\definecolor{currentfill}{rgb}{0.172549,0.627451,0.172549}%
\pgfsetfillcolor{currentfill}%
\pgfsetlinewidth{1.003750pt}%
\definecolor{currentstroke}{rgb}{0.172549,0.627451,0.172549}%
\pgfsetstrokecolor{currentstroke}%
\pgfsetdash{}{0pt}%
\pgfsys@defobject{currentmarker}{\pgfqpoint{-0.041667in}{-0.041667in}}{\pgfqpoint{0.041667in}{0.041667in}}{%
\pgfpathmoveto{\pgfqpoint{-0.041667in}{-0.041667in}}%
\pgfpathlineto{\pgfqpoint{0.041667in}{0.041667in}}%
\pgfpathmoveto{\pgfqpoint{-0.041667in}{0.041667in}}%
\pgfpathlineto{\pgfqpoint{0.041667in}{-0.041667in}}%
\pgfusepath{stroke,fill}%
}%
\begin{pgfscope}%
\pgfsys@transformshift{2.500617in}{0.893981in}%
\pgfsys@useobject{currentmarker}{}%
\end{pgfscope}%
\end{pgfscope}%
\begin{pgfscope}%
\definecolor{textcolor}{rgb}{0.000000,0.000000,0.000000}%
\pgfsetstrokecolor{textcolor}%
\pgfsetfillcolor{textcolor}%
\pgftext[x=2.750617in,y=0.845370in,left,base]{\color{textcolor}\rmfamily\fontsize{10.000000}{12.000000}\selectfont deg\(\displaystyle =3\)}%
\end{pgfscope}%
\begin{pgfscope}%
\pgfsetrectcap%
\pgfsetroundjoin%
\pgfsetlinewidth{1.505625pt}%
\definecolor{currentstroke}{rgb}{0.839216,0.152941,0.156863}%
\pgfsetstrokecolor{currentstroke}%
\pgfsetdash{}{0pt}%
\pgfpathmoveto{\pgfqpoint{2.361728in}{0.700308in}}%
\pgfpathlineto{\pgfqpoint{2.639506in}{0.700308in}}%
\pgfusepath{stroke}%
\end{pgfscope}%
\begin{pgfscope}%
\pgfsetbuttcap%
\pgfsetroundjoin%
\definecolor{currentfill}{rgb}{0.839216,0.152941,0.156863}%
\pgfsetfillcolor{currentfill}%
\pgfsetlinewidth{1.003750pt}%
\definecolor{currentstroke}{rgb}{0.839216,0.152941,0.156863}%
\pgfsetstrokecolor{currentstroke}%
\pgfsetdash{}{0pt}%
\pgfsys@defobject{currentmarker}{\pgfqpoint{-0.041667in}{-0.041667in}}{\pgfqpoint{0.041667in}{0.041667in}}{%
\pgfpathmoveto{\pgfqpoint{-0.041667in}{-0.041667in}}%
\pgfpathlineto{\pgfqpoint{0.041667in}{0.041667in}}%
\pgfpathmoveto{\pgfqpoint{-0.041667in}{0.041667in}}%
\pgfpathlineto{\pgfqpoint{0.041667in}{-0.041667in}}%
\pgfusepath{stroke,fill}%
}%
\begin{pgfscope}%
\pgfsys@transformshift{2.500617in}{0.700308in}%
\pgfsys@useobject{currentmarker}{}%
\end{pgfscope}%
\end{pgfscope}%
\begin{pgfscope}%
\definecolor{textcolor}{rgb}{0.000000,0.000000,0.000000}%
\pgfsetstrokecolor{textcolor}%
\pgfsetfillcolor{textcolor}%
\pgftext[x=2.750617in,y=0.651697in,left,base]{\color{textcolor}\rmfamily\fontsize{10.000000}{12.000000}\selectfont deg\(\displaystyle =4\)}%
\end{pgfscope}%
\end{pgfpicture}%
\makeatother%
\endgroup%

%% file: figures/energies_pd.pgf
\begingroup%
\makeatletter%
\begin{pgfpicture}%
\pgfpathrectangle{\pgfpointorigin}{\pgfqpoint{3.400000in}{3.400000in}}%
\pgfusepath{use as bounding box, clip}%
\begin{pgfscope}%
\pgfsetbuttcap%
\pgfsetmiterjoin%
\definecolor{currentfill}{rgb}{1.000000,1.000000,1.000000}%
\pgfsetfillcolor{currentfill}%
\pgfsetlinewidth{0.000000pt}%
\definecolor{currentstroke}{rgb}{1.000000,1.000000,1.000000}%
\pgfsetstrokecolor{currentstroke}%
\pgfsetdash{}{0pt}%
\pgfpathmoveto{\pgfqpoint{0.000000in}{0.000000in}}%
\pgfpathlineto{\pgfqpoint{3.400000in}{0.000000in}}%
\pgfpathlineto{\pgfqpoint{3.400000in}{3.400000in}}%
\pgfpathlineto{\pgfqpoint{0.000000in}{3.400000in}}%
\pgfpathclose%
\pgfusepath{fill}%
\end{pgfscope}%
\begin{pgfscope}%
\pgfsetbuttcap%
\pgfsetmiterjoin%
\definecolor{currentfill}{rgb}{1.000000,1.000000,1.000000}%
\pgfsetfillcolor{currentfill}%
\pgfsetlinewidth{0.000000pt}%
\definecolor{currentstroke}{rgb}{0.000000,0.000000,0.000000}%
\pgfsetstrokecolor{currentstroke}%
\pgfsetstrokeopacity{0.000000}%
\pgfsetdash{}{0pt}%
\pgfpathmoveto{\pgfqpoint{0.692593in}{0.499691in}}%
\pgfpathlineto{\pgfqpoint{3.300000in}{0.499691in}}%
\pgfpathlineto{\pgfqpoint{3.300000in}{3.100926in}}%
\pgfpathlineto{\pgfqpoint{0.692593in}{3.100926in}}%
\pgfpathclose%
\pgfusepath{fill}%
\end{pgfscope}%
\begin{pgfscope}%
\pgfsetbuttcap%
\pgfsetroundjoin%
\definecolor{currentfill}{rgb}{0.000000,0.000000,0.000000}%
\pgfsetfillcolor{currentfill}%
\pgfsetlinewidth{0.803000pt}%
\definecolor{currentstroke}{rgb}{0.000000,0.000000,0.000000}%
\pgfsetstrokecolor{currentstroke}%
\pgfsetdash{}{0pt}%
\pgfsys@defobject{currentmarker}{\pgfqpoint{0.000000in}{-0.048611in}}{\pgfqpoint{0.000000in}{0.000000in}}{%
\pgfpathmoveto{\pgfqpoint{0.000000in}{0.000000in}}%
\pgfpathlineto{\pgfqpoint{0.000000in}{-0.048611in}}%
\pgfusepath{stroke,fill}%
}%
\begin{pgfscope}%
\pgfsys@transformshift{0.811112in}{0.499691in}%
\pgfsys@useobject{currentmarker}{}%
\end{pgfscope}%
\end{pgfscope}%
\begin{pgfscope}%
\definecolor{textcolor}{rgb}{0.000000,0.000000,0.000000}%
\pgfsetstrokecolor{textcolor}%
\pgfsetfillcolor{textcolor}%
\pgftext[x=0.811112in,y=0.402469in,,top]{\color{textcolor}\rmfamily\fontsize{10.000000}{12.000000}\selectfont \(\displaystyle 1\)}%
\end{pgfscope}%
\begin{pgfscope}%
\pgfsetbuttcap%
\pgfsetroundjoin%
\definecolor{currentfill}{rgb}{0.000000,0.000000,0.000000}%
\pgfsetfillcolor{currentfill}%
\pgfsetlinewidth{0.803000pt}%
\definecolor{currentstroke}{rgb}{0.000000,0.000000,0.000000}%
\pgfsetstrokecolor{currentstroke}%
\pgfsetdash{}{0pt}%
\pgfsys@defobject{currentmarker}{\pgfqpoint{0.000000in}{-0.048611in}}{\pgfqpoint{0.000000in}{0.000000in}}{%
\pgfpathmoveto{\pgfqpoint{0.000000in}{0.000000in}}%
\pgfpathlineto{\pgfqpoint{0.000000in}{-0.048611in}}%
\pgfusepath{stroke,fill}%
}%
\begin{pgfscope}%
\pgfsys@transformshift{1.601235in}{0.499691in}%
\pgfsys@useobject{currentmarker}{}%
\end{pgfscope}%
\end{pgfscope}%
\begin{pgfscope}%
\definecolor{textcolor}{rgb}{0.000000,0.000000,0.000000}%
\pgfsetstrokecolor{textcolor}%
\pgfsetfillcolor{textcolor}%
\pgftext[x=1.601235in,y=0.402469in,,top]{\color{textcolor}\rmfamily\fontsize{10.000000}{12.000000}\selectfont \(\displaystyle 2\)}%
\end{pgfscope}%
\begin{pgfscope}%
\pgfsetbuttcap%
\pgfsetroundjoin%
\definecolor{currentfill}{rgb}{0.000000,0.000000,0.000000}%
\pgfsetfillcolor{currentfill}%
\pgfsetlinewidth{0.803000pt}%
\definecolor{currentstroke}{rgb}{0.000000,0.000000,0.000000}%
\pgfsetstrokecolor{currentstroke}%
\pgfsetdash{}{0pt}%
\pgfsys@defobject{currentmarker}{\pgfqpoint{0.000000in}{-0.048611in}}{\pgfqpoint{0.000000in}{0.000000in}}{%
\pgfpathmoveto{\pgfqpoint{0.000000in}{0.000000in}}%
\pgfpathlineto{\pgfqpoint{0.000000in}{-0.048611in}}%
\pgfusepath{stroke,fill}%
}%
\begin{pgfscope}%
\pgfsys@transformshift{2.391358in}{0.499691in}%
\pgfsys@useobject{currentmarker}{}%
\end{pgfscope}%
\end{pgfscope}%
\begin{pgfscope}%
\definecolor{textcolor}{rgb}{0.000000,0.000000,0.000000}%
\pgfsetstrokecolor{textcolor}%
\pgfsetfillcolor{textcolor}%
\pgftext[x=2.391358in,y=0.402469in,,top]{\color{textcolor}\rmfamily\fontsize{10.000000}{12.000000}\selectfont \(\displaystyle 3\)}%
\end{pgfscope}%
\begin{pgfscope}%
\pgfsetbuttcap%
\pgfsetroundjoin%
\definecolor{currentfill}{rgb}{0.000000,0.000000,0.000000}%
\pgfsetfillcolor{currentfill}%
\pgfsetlinewidth{0.803000pt}%
\definecolor{currentstroke}{rgb}{0.000000,0.000000,0.000000}%
\pgfsetstrokecolor{currentstroke}%
\pgfsetdash{}{0pt}%
\pgfsys@defobject{currentmarker}{\pgfqpoint{0.000000in}{-0.048611in}}{\pgfqpoint{0.000000in}{0.000000in}}{%
\pgfpathmoveto{\pgfqpoint{0.000000in}{0.000000in}}%
\pgfpathlineto{\pgfqpoint{0.000000in}{-0.048611in}}%
\pgfusepath{stroke,fill}%
}%
\begin{pgfscope}%
\pgfsys@transformshift{3.181482in}{0.499691in}%
\pgfsys@useobject{currentmarker}{}%
\end{pgfscope}%
\end{pgfscope}%
\begin{pgfscope}%
\definecolor{textcolor}{rgb}{0.000000,0.000000,0.000000}%
\pgfsetstrokecolor{textcolor}%
\pgfsetfillcolor{textcolor}%
\pgftext[x=3.181482in,y=0.402469in,,top]{\color{textcolor}\rmfamily\fontsize{10.000000}{12.000000}\selectfont \(\displaystyle 4\)}%
\end{pgfscope}%
\begin{pgfscope}%
\definecolor{textcolor}{rgb}{0.000000,0.000000,0.000000}%
\pgfsetstrokecolor{textcolor}%
\pgfsetfillcolor{textcolor}%
\pgftext[x=1.996297in,y=0.223457in,,top]{\color{textcolor}\rmfamily\fontsize{10.000000}{12.000000}\selectfont Number of pieces \(\displaystyle K\)}%
\end{pgfscope}%
\begin{pgfscope}%
\pgfsetbuttcap%
\pgfsetroundjoin%
\definecolor{currentfill}{rgb}{0.000000,0.000000,0.000000}%
\pgfsetfillcolor{currentfill}%
\pgfsetlinewidth{0.803000pt}%
\definecolor{currentstroke}{rgb}{0.000000,0.000000,0.000000}%
\pgfsetstrokecolor{currentstroke}%
\pgfsetdash{}{0pt}%
\pgfsys@defobject{currentmarker}{\pgfqpoint{-0.048611in}{0.000000in}}{\pgfqpoint{0.000000in}{0.000000in}}{%
\pgfpathmoveto{\pgfqpoint{0.000000in}{0.000000in}}%
\pgfpathlineto{\pgfqpoint{-0.048611in}{0.000000in}}%
\pgfusepath{stroke,fill}%
}%
\begin{pgfscope}%
\pgfsys@transformshift{0.692593in}{0.653676in}%
\pgfsys@useobject{currentmarker}{}%
\end{pgfscope}%
\end{pgfscope}%
\begin{pgfscope}%
\definecolor{textcolor}{rgb}{0.000000,0.000000,0.000000}%
\pgfsetstrokecolor{textcolor}%
\pgfsetfillcolor{textcolor}%
\pgftext[x=0.279012in,y=0.605451in,left,base]{\color{textcolor}\rmfamily\fontsize{10.000000}{12.000000}\selectfont \(\displaystyle -250\)}%
\end{pgfscope}%
\begin{pgfscope}%
\pgfsetbuttcap%
\pgfsetroundjoin%
\definecolor{currentfill}{rgb}{0.000000,0.000000,0.000000}%
\pgfsetfillcolor{currentfill}%
\pgfsetlinewidth{0.803000pt}%
\definecolor{currentstroke}{rgb}{0.000000,0.000000,0.000000}%
\pgfsetstrokecolor{currentstroke}%
\pgfsetdash{}{0pt}%
\pgfsys@defobject{currentmarker}{\pgfqpoint{-0.048611in}{0.000000in}}{\pgfqpoint{0.000000in}{0.000000in}}{%
\pgfpathmoveto{\pgfqpoint{0.000000in}{0.000000in}}%
\pgfpathlineto{\pgfqpoint{-0.048611in}{0.000000in}}%
\pgfusepath{stroke,fill}%
}%
\begin{pgfscope}%
\pgfsys@transformshift{0.692593in}{1.097334in}%
\pgfsys@useobject{currentmarker}{}%
\end{pgfscope}%
\end{pgfscope}%
\begin{pgfscope}%
\definecolor{textcolor}{rgb}{0.000000,0.000000,0.000000}%
\pgfsetstrokecolor{textcolor}%
\pgfsetfillcolor{textcolor}%
\pgftext[x=0.279012in,y=1.049109in,left,base]{\color{textcolor}\rmfamily\fontsize{10.000000}{12.000000}\selectfont \(\displaystyle -200\)}%
\end{pgfscope}%
\begin{pgfscope}%
\pgfsetbuttcap%
\pgfsetroundjoin%
\definecolor{currentfill}{rgb}{0.000000,0.000000,0.000000}%
\pgfsetfillcolor{currentfill}%
\pgfsetlinewidth{0.803000pt}%
\definecolor{currentstroke}{rgb}{0.000000,0.000000,0.000000}%
\pgfsetstrokecolor{currentstroke}%
\pgfsetdash{}{0pt}%
\pgfsys@defobject{currentmarker}{\pgfqpoint{-0.048611in}{0.000000in}}{\pgfqpoint{0.000000in}{0.000000in}}{%
\pgfpathmoveto{\pgfqpoint{0.000000in}{0.000000in}}%
\pgfpathlineto{\pgfqpoint{-0.048611in}{0.000000in}}%
\pgfusepath{stroke,fill}%
}%
\begin{pgfscope}%
\pgfsys@transformshift{0.692593in}{1.540993in}%
\pgfsys@useobject{currentmarker}{}%
\end{pgfscope}%
\end{pgfscope}%
\begin{pgfscope}%
\definecolor{textcolor}{rgb}{0.000000,0.000000,0.000000}%
\pgfsetstrokecolor{textcolor}%
\pgfsetfillcolor{textcolor}%
\pgftext[x=0.279012in,y=1.492767in,left,base]{\color{textcolor}\rmfamily\fontsize{10.000000}{12.000000}\selectfont \(\displaystyle -150\)}%
\end{pgfscope}%
\begin{pgfscope}%
\pgfsetbuttcap%
\pgfsetroundjoin%
\definecolor{currentfill}{rgb}{0.000000,0.000000,0.000000}%
\pgfsetfillcolor{currentfill}%
\pgfsetlinewidth{0.803000pt}%
\definecolor{currentstroke}{rgb}{0.000000,0.000000,0.000000}%
\pgfsetstrokecolor{currentstroke}%
\pgfsetdash{}{0pt}%
\pgfsys@defobject{currentmarker}{\pgfqpoint{-0.048611in}{0.000000in}}{\pgfqpoint{0.000000in}{0.000000in}}{%
\pgfpathmoveto{\pgfqpoint{0.000000in}{0.000000in}}%
\pgfpathlineto{\pgfqpoint{-0.048611in}{0.000000in}}%
\pgfusepath{stroke,fill}%
}%
\begin{pgfscope}%
\pgfsys@transformshift{0.692593in}{1.984651in}%
\pgfsys@useobject{currentmarker}{}%
\end{pgfscope}%
\end{pgfscope}%
\begin{pgfscope}%
\definecolor{textcolor}{rgb}{0.000000,0.000000,0.000000}%
\pgfsetstrokecolor{textcolor}%
\pgfsetfillcolor{textcolor}%
\pgftext[x=0.279012in,y=1.936426in,left,base]{\color{textcolor}\rmfamily\fontsize{10.000000}{12.000000}\selectfont \(\displaystyle -100\)}%
\end{pgfscope}%
\begin{pgfscope}%
\pgfsetbuttcap%
\pgfsetroundjoin%
\definecolor{currentfill}{rgb}{0.000000,0.000000,0.000000}%
\pgfsetfillcolor{currentfill}%
\pgfsetlinewidth{0.803000pt}%
\definecolor{currentstroke}{rgb}{0.000000,0.000000,0.000000}%
\pgfsetstrokecolor{currentstroke}%
\pgfsetdash{}{0pt}%
\pgfsys@defobject{currentmarker}{\pgfqpoint{-0.048611in}{0.000000in}}{\pgfqpoint{0.000000in}{0.000000in}}{%
\pgfpathmoveto{\pgfqpoint{0.000000in}{0.000000in}}%
\pgfpathlineto{\pgfqpoint{-0.048611in}{0.000000in}}%
\pgfusepath{stroke,fill}%
}%
\begin{pgfscope}%
\pgfsys@transformshift{0.692593in}{2.428310in}%
\pgfsys@useobject{currentmarker}{}%
\end{pgfscope}%
\end{pgfscope}%
\begin{pgfscope}%
\definecolor{textcolor}{rgb}{0.000000,0.000000,0.000000}%
\pgfsetstrokecolor{textcolor}%
\pgfsetfillcolor{textcolor}%
\pgftext[x=0.348457in,y=2.380084in,left,base]{\color{textcolor}\rmfamily\fontsize{10.000000}{12.000000}\selectfont \(\displaystyle -50\)}%
\end{pgfscope}%
\begin{pgfscope}%
\pgfsetbuttcap%
\pgfsetroundjoin%
\definecolor{currentfill}{rgb}{0.000000,0.000000,0.000000}%
\pgfsetfillcolor{currentfill}%
\pgfsetlinewidth{0.803000pt}%
\definecolor{currentstroke}{rgb}{0.000000,0.000000,0.000000}%
\pgfsetstrokecolor{currentstroke}%
\pgfsetdash{}{0pt}%
\pgfsys@defobject{currentmarker}{\pgfqpoint{-0.048611in}{0.000000in}}{\pgfqpoint{0.000000in}{0.000000in}}{%
\pgfpathmoveto{\pgfqpoint{0.000000in}{0.000000in}}%
\pgfpathlineto{\pgfqpoint{-0.048611in}{0.000000in}}%
\pgfusepath{stroke,fill}%
}%
\begin{pgfscope}%
\pgfsys@transformshift{0.692593in}{2.871968in}%
\pgfsys@useobject{currentmarker}{}%
\end{pgfscope}%
\end{pgfscope}%
\begin{pgfscope}%
\definecolor{textcolor}{rgb}{0.000000,0.000000,0.000000}%
\pgfsetstrokecolor{textcolor}%
\pgfsetfillcolor{textcolor}%
\pgftext[x=0.525927in,y=2.823743in,left,base]{\color{textcolor}\rmfamily\fontsize{10.000000}{12.000000}\selectfont \(\displaystyle 0\)}%
\end{pgfscope}%
\begin{pgfscope}%
\definecolor{textcolor}{rgb}{0.000000,0.000000,0.000000}%
\pgfsetstrokecolor{textcolor}%
\pgfsetfillcolor{textcolor}%
\pgftext[x=0.223457in,y=1.800309in,,bottom,rotate=90.000000]{\color{textcolor}\rmfamily\fontsize{10.000000}{12.000000}\selectfont Energy}%
\end{pgfscope}%
\begin{pgfscope}%
\pgfpathrectangle{\pgfqpoint{0.692593in}{0.499691in}}{\pgfqpoint{2.607407in}{2.601235in}}%
\pgfusepath{clip}%
\pgfsetrectcap%
\pgfsetroundjoin%
\pgfsetlinewidth{1.505625pt}%
\definecolor{currentstroke}{rgb}{0.121569,0.466667,0.705882}%
\pgfsetstrokecolor{currentstroke}%
\pgfsetdash{}{0pt}%
\pgfpathmoveto{\pgfqpoint{0.811112in}{0.617929in}}%
\pgfpathlineto{\pgfqpoint{1.601235in}{1.136541in}}%
\pgfpathlineto{\pgfqpoint{2.391358in}{1.287739in}}%
\pgfpathlineto{\pgfqpoint{3.181482in}{1.484953in}}%
\pgfusepath{stroke}%
\end{pgfscope}%
\begin{pgfscope}%
\pgfpathrectangle{\pgfqpoint{0.692593in}{0.499691in}}{\pgfqpoint{2.607407in}{2.601235in}}%
\pgfusepath{clip}%
\pgfsetbuttcap%
\pgfsetroundjoin%
\definecolor{currentfill}{rgb}{0.121569,0.466667,0.705882}%
\pgfsetfillcolor{currentfill}%
\pgfsetlinewidth{1.003750pt}%
\definecolor{currentstroke}{rgb}{0.121569,0.466667,0.705882}%
\pgfsetstrokecolor{currentstroke}%
\pgfsetdash{}{0pt}%
\pgfsys@defobject{currentmarker}{\pgfqpoint{-0.041667in}{-0.041667in}}{\pgfqpoint{0.041667in}{0.041667in}}{%
\pgfpathmoveto{\pgfqpoint{-0.041667in}{-0.041667in}}%
\pgfpathlineto{\pgfqpoint{0.041667in}{0.041667in}}%
\pgfpathmoveto{\pgfqpoint{-0.041667in}{0.041667in}}%
\pgfpathlineto{\pgfqpoint{0.041667in}{-0.041667in}}%
\pgfusepath{stroke,fill}%
}%
\begin{pgfscope}%
\pgfsys@transformshift{0.811112in}{0.617929in}%
\pgfsys@useobject{currentmarker}{}%
\end{pgfscope}%
\begin{pgfscope}%
\pgfsys@transformshift{1.601235in}{1.136541in}%
\pgfsys@useobject{currentmarker}{}%
\end{pgfscope}%
\begin{pgfscope}%
\pgfsys@transformshift{2.391358in}{1.287739in}%
\pgfsys@useobject{currentmarker}{}%
\end{pgfscope}%
\begin{pgfscope}%
\pgfsys@transformshift{3.181482in}{1.484953in}%
\pgfsys@useobject{currentmarker}{}%
\end{pgfscope}%
\end{pgfscope}%
\begin{pgfscope}%
\pgfpathrectangle{\pgfqpoint{0.692593in}{0.499691in}}{\pgfqpoint{2.607407in}{2.601235in}}%
\pgfusepath{clip}%
\pgfsetrectcap%
\pgfsetroundjoin%
\pgfsetlinewidth{1.505625pt}%
\definecolor{currentstroke}{rgb}{1.000000,0.498039,0.054902}%
\pgfsetstrokecolor{currentstroke}%
\pgfsetdash{}{0pt}%
\pgfpathmoveto{\pgfqpoint{0.811112in}{0.971664in}}%
\pgfpathlineto{\pgfqpoint{1.601235in}{1.387406in}}%
\pgfpathlineto{\pgfqpoint{2.391358in}{1.497579in}}%
\pgfpathlineto{\pgfqpoint{3.181482in}{1.587113in}}%
\pgfusepath{stroke}%
\end{pgfscope}%
\begin{pgfscope}%
\pgfpathrectangle{\pgfqpoint{0.692593in}{0.499691in}}{\pgfqpoint{2.607407in}{2.601235in}}%
\pgfusepath{clip}%
\pgfsetbuttcap%
\pgfsetroundjoin%
\definecolor{currentfill}{rgb}{1.000000,0.498039,0.054902}%
\pgfsetfillcolor{currentfill}%
\pgfsetlinewidth{1.003750pt}%
\definecolor{currentstroke}{rgb}{1.000000,0.498039,0.054902}%
\pgfsetstrokecolor{currentstroke}%
\pgfsetdash{}{0pt}%
\pgfsys@defobject{currentmarker}{\pgfqpoint{-0.041667in}{-0.041667in}}{\pgfqpoint{0.041667in}{0.041667in}}{%
\pgfpathmoveto{\pgfqpoint{-0.041667in}{-0.041667in}}%
\pgfpathlineto{\pgfqpoint{0.041667in}{0.041667in}}%
\pgfpathmoveto{\pgfqpoint{-0.041667in}{0.041667in}}%
\pgfpathlineto{\pgfqpoint{0.041667in}{-0.041667in}}%
\pgfusepath{stroke,fill}%
}%
\begin{pgfscope}%
\pgfsys@transformshift{0.811112in}{0.971664in}%
\pgfsys@useobject{currentmarker}{}%
\end{pgfscope}%
\begin{pgfscope}%
\pgfsys@transformshift{1.601235in}{1.387406in}%
\pgfsys@useobject{currentmarker}{}%
\end{pgfscope}%
\begin{pgfscope}%
\pgfsys@transformshift{2.391358in}{1.497579in}%
\pgfsys@useobject{currentmarker}{}%
\end{pgfscope}%
\begin{pgfscope}%
\pgfsys@transformshift{3.181482in}{1.587113in}%
\pgfsys@useobject{currentmarker}{}%
\end{pgfscope}%
\end{pgfscope}%
\begin{pgfscope}%
\pgfpathrectangle{\pgfqpoint{0.692593in}{0.499691in}}{\pgfqpoint{2.607407in}{2.601235in}}%
\pgfusepath{clip}%
\pgfsetrectcap%
\pgfsetroundjoin%
\pgfsetlinewidth{1.505625pt}%
\definecolor{currentstroke}{rgb}{0.172549,0.627451,0.172549}%
\pgfsetstrokecolor{currentstroke}%
\pgfsetdash{}{0pt}%
\pgfpathmoveto{\pgfqpoint{0.811112in}{1.169005in}}%
\pgfpathlineto{\pgfqpoint{1.601235in}{1.523687in}}%
\pgfpathlineto{\pgfqpoint{2.391358in}{1.571662in}}%
\pgfpathlineto{\pgfqpoint{3.181482in}{1.627379in}}%
\pgfusepath{stroke}%
\end{pgfscope}%
\begin{pgfscope}%
\pgfpathrectangle{\pgfqpoint{0.692593in}{0.499691in}}{\pgfqpoint{2.607407in}{2.601235in}}%
\pgfusepath{clip}%
\pgfsetbuttcap%
\pgfsetroundjoin%
\definecolor{currentfill}{rgb}{0.172549,0.627451,0.172549}%
\pgfsetfillcolor{currentfill}%
\pgfsetlinewidth{1.003750pt}%
\definecolor{currentstroke}{rgb}{0.172549,0.627451,0.172549}%
\pgfsetstrokecolor{currentstroke}%
\pgfsetdash{}{0pt}%
\pgfsys@defobject{currentmarker}{\pgfqpoint{-0.041667in}{-0.041667in}}{\pgfqpoint{0.041667in}{0.041667in}}{%
\pgfpathmoveto{\pgfqpoint{-0.041667in}{-0.041667in}}%
\pgfpathlineto{\pgfqpoint{0.041667in}{0.041667in}}%
\pgfpathmoveto{\pgfqpoint{-0.041667in}{0.041667in}}%
\pgfpathlineto{\pgfqpoint{0.041667in}{-0.041667in}}%
\pgfusepath{stroke,fill}%
}%
\begin{pgfscope}%
\pgfsys@transformshift{0.811112in}{1.169005in}%
\pgfsys@useobject{currentmarker}{}%
\end{pgfscope}%
\begin{pgfscope}%
\pgfsys@transformshift{1.601235in}{1.523687in}%
\pgfsys@useobject{currentmarker}{}%
\end{pgfscope}%
\begin{pgfscope}%
\pgfsys@transformshift{2.391358in}{1.571662in}%
\pgfsys@useobject{currentmarker}{}%
\end{pgfscope}%
\begin{pgfscope}%
\pgfsys@transformshift{3.181482in}{1.627379in}%
\pgfsys@useobject{currentmarker}{}%
\end{pgfscope}%
\end{pgfscope}%
\begin{pgfscope}%
\pgfpathrectangle{\pgfqpoint{0.692593in}{0.499691in}}{\pgfqpoint{2.607407in}{2.601235in}}%
\pgfusepath{clip}%
\pgfsetrectcap%
\pgfsetroundjoin%
\pgfsetlinewidth{1.505625pt}%
\definecolor{currentstroke}{rgb}{0.839216,0.152941,0.156863}%
\pgfsetstrokecolor{currentstroke}%
\pgfsetdash{}{0pt}%
\pgfpathmoveto{\pgfqpoint{0.811112in}{1.412379in}}%
\pgfpathlineto{\pgfqpoint{1.601235in}{1.583997in}}%
\pgfpathlineto{\pgfqpoint{2.391358in}{1.617357in}}%
\pgfpathlineto{\pgfqpoint{3.181482in}{1.643790in}}%
\pgfusepath{stroke}%
\end{pgfscope}%
\begin{pgfscope}%
\pgfpathrectangle{\pgfqpoint{0.692593in}{0.499691in}}{\pgfqpoint{2.607407in}{2.601235in}}%
\pgfusepath{clip}%
\pgfsetbuttcap%
\pgfsetroundjoin%
\definecolor{currentfill}{rgb}{0.839216,0.152941,0.156863}%
\pgfsetfillcolor{currentfill}%
\pgfsetlinewidth{1.003750pt}%
\definecolor{currentstroke}{rgb}{0.839216,0.152941,0.156863}%
\pgfsetstrokecolor{currentstroke}%
\pgfsetdash{}{0pt}%
\pgfsys@defobject{currentmarker}{\pgfqpoint{-0.041667in}{-0.041667in}}{\pgfqpoint{0.041667in}{0.041667in}}{%
\pgfpathmoveto{\pgfqpoint{-0.041667in}{-0.041667in}}%
\pgfpathlineto{\pgfqpoint{0.041667in}{0.041667in}}%
\pgfpathmoveto{\pgfqpoint{-0.041667in}{0.041667in}}%
\pgfpathlineto{\pgfqpoint{0.041667in}{-0.041667in}}%
\pgfusepath{stroke,fill}%
}%
\begin{pgfscope}%
\pgfsys@transformshift{0.811112in}{1.412379in}%
\pgfsys@useobject{currentmarker}{}%
\end{pgfscope}%
\begin{pgfscope}%
\pgfsys@transformshift{1.601235in}{1.583997in}%
\pgfsys@useobject{currentmarker}{}%
\end{pgfscope}%
\begin{pgfscope}%
\pgfsys@transformshift{2.391358in}{1.617357in}%
\pgfsys@useobject{currentmarker}{}%
\end{pgfscope}%
\begin{pgfscope}%
\pgfsys@transformshift{3.181482in}{1.643790in}%
\pgfsys@useobject{currentmarker}{}%
\end{pgfscope}%
\end{pgfscope}%
\begin{pgfscope}%
\pgfpathrectangle{\pgfqpoint{0.692593in}{0.499691in}}{\pgfqpoint{2.607407in}{2.601235in}}%
\pgfusepath{clip}%
\pgfsetbuttcap%
\pgfsetroundjoin%
\pgfsetlinewidth{1.505625pt}%
\definecolor{currentstroke}{rgb}{0.121569,0.466667,0.705882}%
\pgfsetstrokecolor{currentstroke}%
\pgfsetdash{{5.550000pt}{2.400000pt}}{0.000000pt}%
\pgfpathmoveto{\pgfqpoint{0.811112in}{2.796428in}}%
\pgfpathlineto{\pgfqpoint{1.601235in}{2.982688in}}%
\pgfpathlineto{\pgfqpoint{2.391358in}{2.159842in}}%
\pgfpathlineto{\pgfqpoint{3.181482in}{2.208573in}}%
\pgfusepath{stroke}%
\end{pgfscope}%
\begin{pgfscope}%
\pgfpathrectangle{\pgfqpoint{0.692593in}{0.499691in}}{\pgfqpoint{2.607407in}{2.601235in}}%
\pgfusepath{clip}%
\pgfsetbuttcap%
\pgfsetroundjoin%
\definecolor{currentfill}{rgb}{0.121569,0.466667,0.705882}%
\pgfsetfillcolor{currentfill}%
\pgfsetlinewidth{1.003750pt}%
\definecolor{currentstroke}{rgb}{0.121569,0.466667,0.705882}%
\pgfsetstrokecolor{currentstroke}%
\pgfsetdash{}{0pt}%
\pgfsys@defobject{currentmarker}{\pgfqpoint{-0.041667in}{-0.041667in}}{\pgfqpoint{0.041667in}{0.041667in}}{%
\pgfpathmoveto{\pgfqpoint{-0.041667in}{-0.041667in}}%
\pgfpathlineto{\pgfqpoint{0.041667in}{0.041667in}}%
\pgfpathmoveto{\pgfqpoint{-0.041667in}{0.041667in}}%
\pgfpathlineto{\pgfqpoint{0.041667in}{-0.041667in}}%
\pgfusepath{stroke,fill}%
}%
\begin{pgfscope}%
\pgfsys@transformshift{0.811112in}{2.796428in}%
\pgfsys@useobject{currentmarker}{}%
\end{pgfscope}%
\begin{pgfscope}%
\pgfsys@transformshift{1.601235in}{2.982688in}%
\pgfsys@useobject{currentmarker}{}%
\end{pgfscope}%
\begin{pgfscope}%
\pgfsys@transformshift{2.391358in}{2.159842in}%
\pgfsys@useobject{currentmarker}{}%
\end{pgfscope}%
\begin{pgfscope}%
\pgfsys@transformshift{3.181482in}{2.208573in}%
\pgfsys@useobject{currentmarker}{}%
\end{pgfscope}%
\end{pgfscope}%
\begin{pgfscope}%
\pgfpathrectangle{\pgfqpoint{0.692593in}{0.499691in}}{\pgfqpoint{2.607407in}{2.601235in}}%
\pgfusepath{clip}%
\pgfsetbuttcap%
\pgfsetroundjoin%
\pgfsetlinewidth{1.505625pt}%
\definecolor{currentstroke}{rgb}{1.000000,0.498039,0.054902}%
\pgfsetstrokecolor{currentstroke}%
\pgfsetdash{{5.550000pt}{2.400000pt}}{0.000000pt}%
\pgfpathmoveto{\pgfqpoint{0.811112in}{2.702519in}}%
\pgfpathlineto{\pgfqpoint{1.601235in}{2.319349in}}%
\pgfpathlineto{\pgfqpoint{2.391358in}{1.922752in}}%
\pgfpathlineto{\pgfqpoint{3.181482in}{1.912024in}}%
\pgfusepath{stroke}%
\end{pgfscope}%
\begin{pgfscope}%
\pgfpathrectangle{\pgfqpoint{0.692593in}{0.499691in}}{\pgfqpoint{2.607407in}{2.601235in}}%
\pgfusepath{clip}%
\pgfsetbuttcap%
\pgfsetroundjoin%
\definecolor{currentfill}{rgb}{1.000000,0.498039,0.054902}%
\pgfsetfillcolor{currentfill}%
\pgfsetlinewidth{1.003750pt}%
\definecolor{currentstroke}{rgb}{1.000000,0.498039,0.054902}%
\pgfsetstrokecolor{currentstroke}%
\pgfsetdash{}{0pt}%
\pgfsys@defobject{currentmarker}{\pgfqpoint{-0.041667in}{-0.041667in}}{\pgfqpoint{0.041667in}{0.041667in}}{%
\pgfpathmoveto{\pgfqpoint{-0.041667in}{-0.041667in}}%
\pgfpathlineto{\pgfqpoint{0.041667in}{0.041667in}}%
\pgfpathmoveto{\pgfqpoint{-0.041667in}{0.041667in}}%
\pgfpathlineto{\pgfqpoint{0.041667in}{-0.041667in}}%
\pgfusepath{stroke,fill}%
}%
\begin{pgfscope}%
\pgfsys@transformshift{0.811112in}{2.702519in}%
\pgfsys@useobject{currentmarker}{}%
\end{pgfscope}%
\begin{pgfscope}%
\pgfsys@transformshift{1.601235in}{2.319349in}%
\pgfsys@useobject{currentmarker}{}%
\end{pgfscope}%
\begin{pgfscope}%
\pgfsys@transformshift{2.391358in}{1.922752in}%
\pgfsys@useobject{currentmarker}{}%
\end{pgfscope}%
\begin{pgfscope}%
\pgfsys@transformshift{3.181482in}{1.912024in}%
\pgfsys@useobject{currentmarker}{}%
\end{pgfscope}%
\end{pgfscope}%
\begin{pgfscope}%
\pgfpathrectangle{\pgfqpoint{0.692593in}{0.499691in}}{\pgfqpoint{2.607407in}{2.601235in}}%
\pgfusepath{clip}%
\pgfsetbuttcap%
\pgfsetroundjoin%
\pgfsetlinewidth{1.505625pt}%
\definecolor{currentstroke}{rgb}{0.172549,0.627451,0.172549}%
\pgfsetstrokecolor{currentstroke}%
\pgfsetdash{{5.550000pt}{2.400000pt}}{0.000000pt}%
\pgfpathmoveto{\pgfqpoint{0.811112in}{2.885484in}}%
\pgfpathlineto{\pgfqpoint{1.601235in}{2.103929in}}%
\pgfpathlineto{\pgfqpoint{2.391358in}{1.919916in}}%
\pgfpathlineto{\pgfqpoint{3.181482in}{1.762542in}}%
\pgfusepath{stroke}%
\end{pgfscope}%
\begin{pgfscope}%
\pgfpathrectangle{\pgfqpoint{0.692593in}{0.499691in}}{\pgfqpoint{2.607407in}{2.601235in}}%
\pgfusepath{clip}%
\pgfsetbuttcap%
\pgfsetroundjoin%
\definecolor{currentfill}{rgb}{0.172549,0.627451,0.172549}%
\pgfsetfillcolor{currentfill}%
\pgfsetlinewidth{1.003750pt}%
\definecolor{currentstroke}{rgb}{0.172549,0.627451,0.172549}%
\pgfsetstrokecolor{currentstroke}%
\pgfsetdash{}{0pt}%
\pgfsys@defobject{currentmarker}{\pgfqpoint{-0.041667in}{-0.041667in}}{\pgfqpoint{0.041667in}{0.041667in}}{%
\pgfpathmoveto{\pgfqpoint{-0.041667in}{-0.041667in}}%
\pgfpathlineto{\pgfqpoint{0.041667in}{0.041667in}}%
\pgfpathmoveto{\pgfqpoint{-0.041667in}{0.041667in}}%
\pgfpathlineto{\pgfqpoint{0.041667in}{-0.041667in}}%
\pgfusepath{stroke,fill}%
}%
\begin{pgfscope}%
\pgfsys@transformshift{0.811112in}{2.885484in}%
\pgfsys@useobject{currentmarker}{}%
\end{pgfscope}%
\begin{pgfscope}%
\pgfsys@transformshift{1.601235in}{2.103929in}%
\pgfsys@useobject{currentmarker}{}%
\end{pgfscope}%
\begin{pgfscope}%
\pgfsys@transformshift{2.391358in}{1.919916in}%
\pgfsys@useobject{currentmarker}{}%
\end{pgfscope}%
\begin{pgfscope}%
\pgfsys@transformshift{3.181482in}{1.762542in}%
\pgfsys@useobject{currentmarker}{}%
\end{pgfscope}%
\end{pgfscope}%
\begin{pgfscope}%
\pgfpathrectangle{\pgfqpoint{0.692593in}{0.499691in}}{\pgfqpoint{2.607407in}{2.601235in}}%
\pgfusepath{clip}%
\pgfsetbuttcap%
\pgfsetroundjoin%
\pgfsetlinewidth{1.505625pt}%
\definecolor{currentstroke}{rgb}{0.839216,0.152941,0.156863}%
\pgfsetstrokecolor{currentstroke}%
\pgfsetdash{{5.550000pt}{2.400000pt}}{0.000000pt}%
\pgfpathmoveto{\pgfqpoint{0.811112in}{2.705167in}}%
\pgfpathlineto{\pgfqpoint{1.601235in}{2.051229in}}%
\pgfpathlineto{\pgfqpoint{2.391358in}{1.885114in}}%
\pgfpathlineto{\pgfqpoint{3.181482in}{1.683701in}}%
\pgfusepath{stroke}%
\end{pgfscope}%
\begin{pgfscope}%
\pgfpathrectangle{\pgfqpoint{0.692593in}{0.499691in}}{\pgfqpoint{2.607407in}{2.601235in}}%
\pgfusepath{clip}%
\pgfsetbuttcap%
\pgfsetroundjoin%
\definecolor{currentfill}{rgb}{0.839216,0.152941,0.156863}%
\pgfsetfillcolor{currentfill}%
\pgfsetlinewidth{1.003750pt}%
\definecolor{currentstroke}{rgb}{0.839216,0.152941,0.156863}%
\pgfsetstrokecolor{currentstroke}%
\pgfsetdash{}{0pt}%
\pgfsys@defobject{currentmarker}{\pgfqpoint{-0.041667in}{-0.041667in}}{\pgfqpoint{0.041667in}{0.041667in}}{%
\pgfpathmoveto{\pgfqpoint{-0.041667in}{-0.041667in}}%
\pgfpathlineto{\pgfqpoint{0.041667in}{0.041667in}}%
\pgfpathmoveto{\pgfqpoint{-0.041667in}{0.041667in}}%
\pgfpathlineto{\pgfqpoint{0.041667in}{-0.041667in}}%
\pgfusepath{stroke,fill}%
}%
\begin{pgfscope}%
\pgfsys@transformshift{0.811112in}{2.705167in}%
\pgfsys@useobject{currentmarker}{}%
\end{pgfscope}%
\begin{pgfscope}%
\pgfsys@transformshift{1.601235in}{2.051229in}%
\pgfsys@useobject{currentmarker}{}%
\end{pgfscope}%
\begin{pgfscope}%
\pgfsys@transformshift{2.391358in}{1.885114in}%
\pgfsys@useobject{currentmarker}{}%
\end{pgfscope}%
\begin{pgfscope}%
\pgfsys@transformshift{3.181482in}{1.683701in}%
\pgfsys@useobject{currentmarker}{}%
\end{pgfscope}%
\end{pgfscope}%
\begin{pgfscope}%
\pgfsetrectcap%
\pgfsetmiterjoin%
\pgfsetlinewidth{0.803000pt}%
\definecolor{currentstroke}{rgb}{0.000000,0.000000,0.000000}%
\pgfsetstrokecolor{currentstroke}%
\pgfsetdash{}{0pt}%
\pgfpathmoveto{\pgfqpoint{0.692593in}{0.499691in}}%
\pgfpathlineto{\pgfqpoint{0.692593in}{3.100926in}}%
\pgfusepath{stroke}%
\end{pgfscope}%
\begin{pgfscope}%
\pgfsetrectcap%
\pgfsetmiterjoin%
\pgfsetlinewidth{0.803000pt}%
\definecolor{currentstroke}{rgb}{0.000000,0.000000,0.000000}%
\pgfsetstrokecolor{currentstroke}%
\pgfsetdash{}{0pt}%
\pgfpathmoveto{\pgfqpoint{3.300000in}{0.499691in}}%
\pgfpathlineto{\pgfqpoint{3.300000in}{3.100926in}}%
\pgfusepath{stroke}%
\end{pgfscope}%
\begin{pgfscope}%
\pgfsetrectcap%
\pgfsetmiterjoin%
\pgfsetlinewidth{0.803000pt}%
\definecolor{currentstroke}{rgb}{0.000000,0.000000,0.000000}%
\pgfsetstrokecolor{currentstroke}%
\pgfsetdash{}{0pt}%
\pgfpathmoveto{\pgfqpoint{0.692593in}{0.499691in}}%
\pgfpathlineto{\pgfqpoint{3.300000in}{0.499691in}}%
\pgfusepath{stroke}%
\end{pgfscope}%
\begin{pgfscope}%
\pgfsetrectcap%
\pgfsetmiterjoin%
\pgfsetlinewidth{0.803000pt}%
\definecolor{currentstroke}{rgb}{0.000000,0.000000,0.000000}%
\pgfsetstrokecolor{currentstroke}%
\pgfsetdash{}{0pt}%
\pgfpathmoveto{\pgfqpoint{0.692593in}{3.100926in}}%
\pgfpathlineto{\pgfqpoint{3.300000in}{3.100926in}}%
\pgfusepath{stroke}%
\end{pgfscope}%
\begin{pgfscope}%
\definecolor{textcolor}{rgb}{0.000000,0.000000,0.000000}%
\pgfsetstrokecolor{textcolor}%
\pgfsetfillcolor{textcolor}%
\pgftext[x=1.996297in,y=3.184260in,,base]{\color{textcolor}\rmfamily\fontsize{12.000000}{14.400000}\selectfont TV regularization}%
\end{pgfscope}%
\begin{pgfscope}%
\pgfsetbuttcap%
\pgfsetmiterjoin%
\definecolor{currentfill}{rgb}{1.000000,1.000000,1.000000}%
\pgfsetfillcolor{currentfill}%
\pgfsetfillopacity{0.800000}%
\pgfsetlinewidth{1.003750pt}%
\definecolor{currentstroke}{rgb}{0.800000,0.800000,0.800000}%
\pgfsetstrokecolor{currentstroke}%
\pgfsetstrokeopacity{0.800000}%
\pgfsetdash{}{0pt}%
\pgfpathmoveto{\pgfqpoint{2.333951in}{2.215124in}}%
\pgfpathlineto{\pgfqpoint{3.202778in}{2.215124in}}%
\pgfpathquadraticcurveto{\pgfqpoint{3.230556in}{2.215124in}}{\pgfqpoint{3.230556in}{2.242902in}}%
\pgfpathlineto{\pgfqpoint{3.230556in}{3.003704in}}%
\pgfpathquadraticcurveto{\pgfqpoint{3.230556in}{3.031482in}}{\pgfqpoint{3.202778in}{3.031482in}}%
\pgfpathlineto{\pgfqpoint{2.333951in}{3.031482in}}%
\pgfpathquadraticcurveto{\pgfqpoint{2.306173in}{3.031482in}}{\pgfqpoint{2.306173in}{3.003704in}}%
\pgfpathlineto{\pgfqpoint{2.306173in}{2.242902in}}%
\pgfpathquadraticcurveto{\pgfqpoint{2.306173in}{2.215124in}}{\pgfqpoint{2.333951in}{2.215124in}}%
\pgfpathclose%
\pgfusepath{stroke,fill}%
\end{pgfscope}%
\begin{pgfscope}%
\pgfsetrectcap%
\pgfsetroundjoin%
\pgfsetlinewidth{1.505625pt}%
\definecolor{currentstroke}{rgb}{0.121569,0.466667,0.705882}%
\pgfsetstrokecolor{currentstroke}%
\pgfsetdash{}{0pt}%
\pgfpathmoveto{\pgfqpoint{2.361728in}{2.927315in}}%
\pgfpathlineto{\pgfqpoint{2.639506in}{2.927315in}}%
\pgfusepath{stroke}%
\end{pgfscope}%
\begin{pgfscope}%
\pgfsetbuttcap%
\pgfsetroundjoin%
\definecolor{currentfill}{rgb}{0.121569,0.466667,0.705882}%
\pgfsetfillcolor{currentfill}%
\pgfsetlinewidth{1.003750pt}%
\definecolor{currentstroke}{rgb}{0.121569,0.466667,0.705882}%
\pgfsetstrokecolor{currentstroke}%
\pgfsetdash{}{0pt}%
\pgfsys@defobject{currentmarker}{\pgfqpoint{-0.041667in}{-0.041667in}}{\pgfqpoint{0.041667in}{0.041667in}}{%
\pgfpathmoveto{\pgfqpoint{-0.041667in}{-0.041667in}}%
\pgfpathlineto{\pgfqpoint{0.041667in}{0.041667in}}%
\pgfpathmoveto{\pgfqpoint{-0.041667in}{0.041667in}}%
\pgfpathlineto{\pgfqpoint{0.041667in}{-0.041667in}}%
\pgfusepath{stroke,fill}%
}%
\begin{pgfscope}%
\pgfsys@transformshift{2.500617in}{2.927315in}%
\pgfsys@useobject{currentmarker}{}%
\end{pgfscope}%
\end{pgfscope}%
\begin{pgfscope}%
\definecolor{textcolor}{rgb}{0.000000,0.000000,0.000000}%
\pgfsetstrokecolor{textcolor}%
\pgfsetfillcolor{textcolor}%
\pgftext[x=2.750617in,y=2.878704in,left,base]{\color{textcolor}\rmfamily\fontsize{10.000000}{12.000000}\selectfont deg\(\displaystyle =1\)}%
\end{pgfscope}%
\begin{pgfscope}%
\pgfsetrectcap%
\pgfsetroundjoin%
\pgfsetlinewidth{1.505625pt}%
\definecolor{currentstroke}{rgb}{1.000000,0.498039,0.054902}%
\pgfsetstrokecolor{currentstroke}%
\pgfsetdash{}{0pt}%
\pgfpathmoveto{\pgfqpoint{2.361728in}{2.733642in}}%
\pgfpathlineto{\pgfqpoint{2.639506in}{2.733642in}}%
\pgfusepath{stroke}%
\end{pgfscope}%
\begin{pgfscope}%
\pgfsetbuttcap%
\pgfsetroundjoin%
\definecolor{currentfill}{rgb}{1.000000,0.498039,0.054902}%
\pgfsetfillcolor{currentfill}%
\pgfsetlinewidth{1.003750pt}%
\definecolor{currentstroke}{rgb}{1.000000,0.498039,0.054902}%
\pgfsetstrokecolor{currentstroke}%
\pgfsetdash{}{0pt}%
\pgfsys@defobject{currentmarker}{\pgfqpoint{-0.041667in}{-0.041667in}}{\pgfqpoint{0.041667in}{0.041667in}}{%
\pgfpathmoveto{\pgfqpoint{-0.041667in}{-0.041667in}}%
\pgfpathlineto{\pgfqpoint{0.041667in}{0.041667in}}%
\pgfpathmoveto{\pgfqpoint{-0.041667in}{0.041667in}}%
\pgfpathlineto{\pgfqpoint{0.041667in}{-0.041667in}}%
\pgfusepath{stroke,fill}%
}%
\begin{pgfscope}%
\pgfsys@transformshift{2.500617in}{2.733642in}%
\pgfsys@useobject{currentmarker}{}%
\end{pgfscope}%
\end{pgfscope}%
\begin{pgfscope}%
\definecolor{textcolor}{rgb}{0.000000,0.000000,0.000000}%
\pgfsetstrokecolor{textcolor}%
\pgfsetfillcolor{textcolor}%
\pgftext[x=2.750617in,y=2.685031in,left,base]{\color{textcolor}\rmfamily\fontsize{10.000000}{12.000000}\selectfont deg\(\displaystyle =2\)}%
\end{pgfscope}%
\begin{pgfscope}%
\pgfsetrectcap%
\pgfsetroundjoin%
\pgfsetlinewidth{1.505625pt}%
\definecolor{currentstroke}{rgb}{0.172549,0.627451,0.172549}%
\pgfsetstrokecolor{currentstroke}%
\pgfsetdash{}{0pt}%
\pgfpathmoveto{\pgfqpoint{2.361728in}{2.539970in}}%
\pgfpathlineto{\pgfqpoint{2.639506in}{2.539970in}}%
\pgfusepath{stroke}%
\end{pgfscope}%
\begin{pgfscope}%
\pgfsetbuttcap%
\pgfsetroundjoin%
\definecolor{currentfill}{rgb}{0.172549,0.627451,0.172549}%
\pgfsetfillcolor{currentfill}%
\pgfsetlinewidth{1.003750pt}%
\definecolor{currentstroke}{rgb}{0.172549,0.627451,0.172549}%
\pgfsetstrokecolor{currentstroke}%
\pgfsetdash{}{0pt}%
\pgfsys@defobject{currentmarker}{\pgfqpoint{-0.041667in}{-0.041667in}}{\pgfqpoint{0.041667in}{0.041667in}}{%
\pgfpathmoveto{\pgfqpoint{-0.041667in}{-0.041667in}}%
\pgfpathlineto{\pgfqpoint{0.041667in}{0.041667in}}%
\pgfpathmoveto{\pgfqpoint{-0.041667in}{0.041667in}}%
\pgfpathlineto{\pgfqpoint{0.041667in}{-0.041667in}}%
\pgfusepath{stroke,fill}%
}%
\begin{pgfscope}%
\pgfsys@transformshift{2.500617in}{2.539970in}%
\pgfsys@useobject{currentmarker}{}%
\end{pgfscope}%
\end{pgfscope}%
\begin{pgfscope}%
\definecolor{textcolor}{rgb}{0.000000,0.000000,0.000000}%
\pgfsetstrokecolor{textcolor}%
\pgfsetfillcolor{textcolor}%
\pgftext[x=2.750617in,y=2.491358in,left,base]{\color{textcolor}\rmfamily\fontsize{10.000000}{12.000000}\selectfont deg\(\displaystyle =3\)}%
\end{pgfscope}%
\begin{pgfscope}%
\pgfsetrectcap%
\pgfsetroundjoin%
\pgfsetlinewidth{1.505625pt}%
\definecolor{currentstroke}{rgb}{0.839216,0.152941,0.156863}%
\pgfsetstrokecolor{currentstroke}%
\pgfsetdash{}{0pt}%
\pgfpathmoveto{\pgfqpoint{2.361728in}{2.346297in}}%
\pgfpathlineto{\pgfqpoint{2.639506in}{2.346297in}}%
\pgfusepath{stroke}%
\end{pgfscope}%
\begin{pgfscope}%
\pgfsetbuttcap%
\pgfsetroundjoin%
\definecolor{currentfill}{rgb}{0.839216,0.152941,0.156863}%
\pgfsetfillcolor{currentfill}%
\pgfsetlinewidth{1.003750pt}%
\definecolor{currentstroke}{rgb}{0.839216,0.152941,0.156863}%
\pgfsetstrokecolor{currentstroke}%
\pgfsetdash{}{0pt}%
\pgfsys@defobject{currentmarker}{\pgfqpoint{-0.041667in}{-0.041667in}}{\pgfqpoint{0.041667in}{0.041667in}}{%
\pgfpathmoveto{\pgfqpoint{-0.041667in}{-0.041667in}}%
\pgfpathlineto{\pgfqpoint{0.041667in}{0.041667in}}%
\pgfpathmoveto{\pgfqpoint{-0.041667in}{0.041667in}}%
\pgfpathlineto{\pgfqpoint{0.041667in}{-0.041667in}}%
\pgfusepath{stroke,fill}%
}%
\begin{pgfscope}%
\pgfsys@transformshift{2.500617in}{2.346297in}%
\pgfsys@useobject{currentmarker}{}%
\end{pgfscope}%
\end{pgfscope}%
\begin{pgfscope}%
\definecolor{textcolor}{rgb}{0.000000,0.000000,0.000000}%
\pgfsetstrokecolor{textcolor}%
\pgfsetfillcolor{textcolor}%
\pgftext[x=2.750617in,y=2.297686in,left,base]{\color{textcolor}\rmfamily\fontsize{10.000000}{12.000000}\selectfont deg\(\displaystyle =4\)}%
\end{pgfscope}%
\end{pgfpicture}%
\makeatother%
\endgroup%

%% file: figures/energies_pd_gap.pgf
\begingroup%
\makeatletter%
\begin{pgfpicture}%
\pgfpathrectangle{\pgfpointorigin}{\pgfqpoint{3.400000in}{3.400000in}}%
\pgfusepath{use as bounding box, clip}%
\begin{pgfscope}%
\pgfsetbuttcap%
\pgfsetmiterjoin%
\definecolor{currentfill}{rgb}{1.000000,1.000000,1.000000}%
\pgfsetfillcolor{currentfill}%
\pgfsetlinewidth{0.000000pt}%
\definecolor{currentstroke}{rgb}{1.000000,1.000000,1.000000}%
\pgfsetstrokecolor{currentstroke}%
\pgfsetdash{}{0pt}%
\pgfpathmoveto{\pgfqpoint{0.000000in}{0.000000in}}%
\pgfpathlineto{\pgfqpoint{3.400000in}{0.000000in}}%
\pgfpathlineto{\pgfqpoint{3.400000in}{3.400000in}}%
\pgfpathlineto{\pgfqpoint{0.000000in}{3.400000in}}%
\pgfpathclose%
\pgfusepath{fill}%
\end{pgfscope}%
\begin{pgfscope}%
\pgfsetbuttcap%
\pgfsetmiterjoin%
\definecolor{currentfill}{rgb}{1.000000,1.000000,1.000000}%
\pgfsetfillcolor{currentfill}%
\pgfsetlinewidth{0.000000pt}%
\definecolor{currentstroke}{rgb}{0.000000,0.000000,0.000000}%
\pgfsetstrokecolor{currentstroke}%
\pgfsetstrokeopacity{0.000000}%
\pgfsetdash{}{0pt}%
\pgfpathmoveto{\pgfqpoint{0.577431in}{0.499691in}}%
\pgfpathlineto{\pgfqpoint{3.300000in}{0.499691in}}%
\pgfpathlineto{\pgfqpoint{3.300000in}{3.100926in}}%
\pgfpathlineto{\pgfqpoint{0.577431in}{3.100926in}}%
\pgfpathclose%
\pgfusepath{fill}%
\end{pgfscope}%
\begin{pgfscope}%
\pgfsetbuttcap%
\pgfsetroundjoin%
\definecolor{currentfill}{rgb}{0.000000,0.000000,0.000000}%
\pgfsetfillcolor{currentfill}%
\pgfsetlinewidth{0.803000pt}%
\definecolor{currentstroke}{rgb}{0.000000,0.000000,0.000000}%
\pgfsetstrokecolor{currentstroke}%
\pgfsetdash{}{0pt}%
\pgfsys@defobject{currentmarker}{\pgfqpoint{0.000000in}{-0.048611in}}{\pgfqpoint{0.000000in}{0.000000in}}{%
\pgfpathmoveto{\pgfqpoint{0.000000in}{0.000000in}}%
\pgfpathlineto{\pgfqpoint{0.000000in}{-0.048611in}}%
\pgfusepath{stroke,fill}%
}%
\begin{pgfscope}%
\pgfsys@transformshift{0.701184in}{0.499691in}%
\pgfsys@useobject{currentmarker}{}%
\end{pgfscope}%
\end{pgfscope}%
\begin{pgfscope}%
\definecolor{textcolor}{rgb}{0.000000,0.000000,0.000000}%
\pgfsetstrokecolor{textcolor}%
\pgfsetfillcolor{textcolor}%
\pgftext[x=0.701184in,y=0.402469in,,top]{\color{textcolor}\rmfamily\fontsize{10.000000}{12.000000}\selectfont \(\displaystyle 1\)}%
\end{pgfscope}%
\begin{pgfscope}%
\pgfsetbuttcap%
\pgfsetroundjoin%
\definecolor{currentfill}{rgb}{0.000000,0.000000,0.000000}%
\pgfsetfillcolor{currentfill}%
\pgfsetlinewidth{0.803000pt}%
\definecolor{currentstroke}{rgb}{0.000000,0.000000,0.000000}%
\pgfsetstrokecolor{currentstroke}%
\pgfsetdash{}{0pt}%
\pgfsys@defobject{currentmarker}{\pgfqpoint{0.000000in}{-0.048611in}}{\pgfqpoint{0.000000in}{0.000000in}}{%
\pgfpathmoveto{\pgfqpoint{0.000000in}{0.000000in}}%
\pgfpathlineto{\pgfqpoint{0.000000in}{-0.048611in}}%
\pgfusepath{stroke,fill}%
}%
\begin{pgfscope}%
\pgfsys@transformshift{1.526205in}{0.499691in}%
\pgfsys@useobject{currentmarker}{}%
\end{pgfscope}%
\end{pgfscope}%
\begin{pgfscope}%
\definecolor{textcolor}{rgb}{0.000000,0.000000,0.000000}%
\pgfsetstrokecolor{textcolor}%
\pgfsetfillcolor{textcolor}%
\pgftext[x=1.526205in,y=0.402469in,,top]{\color{textcolor}\rmfamily\fontsize{10.000000}{12.000000}\selectfont \(\displaystyle 2\)}%
\end{pgfscope}%
\begin{pgfscope}%
\pgfsetbuttcap%
\pgfsetroundjoin%
\definecolor{currentfill}{rgb}{0.000000,0.000000,0.000000}%
\pgfsetfillcolor{currentfill}%
\pgfsetlinewidth{0.803000pt}%
\definecolor{currentstroke}{rgb}{0.000000,0.000000,0.000000}%
\pgfsetstrokecolor{currentstroke}%
\pgfsetdash{}{0pt}%
\pgfsys@defobject{currentmarker}{\pgfqpoint{0.000000in}{-0.048611in}}{\pgfqpoint{0.000000in}{0.000000in}}{%
\pgfpathmoveto{\pgfqpoint{0.000000in}{0.000000in}}%
\pgfpathlineto{\pgfqpoint{0.000000in}{-0.048611in}}%
\pgfusepath{stroke,fill}%
}%
\begin{pgfscope}%
\pgfsys@transformshift{2.351226in}{0.499691in}%
\pgfsys@useobject{currentmarker}{}%
\end{pgfscope}%
\end{pgfscope}%
\begin{pgfscope}%
\definecolor{textcolor}{rgb}{0.000000,0.000000,0.000000}%
\pgfsetstrokecolor{textcolor}%
\pgfsetfillcolor{textcolor}%
\pgftext[x=2.351226in,y=0.402469in,,top]{\color{textcolor}\rmfamily\fontsize{10.000000}{12.000000}\selectfont \(\displaystyle 3\)}%
\end{pgfscope}%
\begin{pgfscope}%
\pgfsetbuttcap%
\pgfsetroundjoin%
\definecolor{currentfill}{rgb}{0.000000,0.000000,0.000000}%
\pgfsetfillcolor{currentfill}%
\pgfsetlinewidth{0.803000pt}%
\definecolor{currentstroke}{rgb}{0.000000,0.000000,0.000000}%
\pgfsetstrokecolor{currentstroke}%
\pgfsetdash{}{0pt}%
\pgfsys@defobject{currentmarker}{\pgfqpoint{0.000000in}{-0.048611in}}{\pgfqpoint{0.000000in}{0.000000in}}{%
\pgfpathmoveto{\pgfqpoint{0.000000in}{0.000000in}}%
\pgfpathlineto{\pgfqpoint{0.000000in}{-0.048611in}}%
\pgfusepath{stroke,fill}%
}%
\begin{pgfscope}%
\pgfsys@transformshift{3.176247in}{0.499691in}%
\pgfsys@useobject{currentmarker}{}%
\end{pgfscope}%
\end{pgfscope}%
\begin{pgfscope}%
\definecolor{textcolor}{rgb}{0.000000,0.000000,0.000000}%
\pgfsetstrokecolor{textcolor}%
\pgfsetfillcolor{textcolor}%
\pgftext[x=3.176247in,y=0.402469in,,top]{\color{textcolor}\rmfamily\fontsize{10.000000}{12.000000}\selectfont \(\displaystyle 4\)}%
\end{pgfscope}%
\begin{pgfscope}%
\definecolor{textcolor}{rgb}{0.000000,0.000000,0.000000}%
\pgfsetstrokecolor{textcolor}%
\pgfsetfillcolor{textcolor}%
\pgftext[x=1.938716in,y=0.223457in,,top]{\color{textcolor}\rmfamily\fontsize{10.000000}{12.000000}\selectfont Number of pieces \(\displaystyle K\)}%
\end{pgfscope}%
\begin{pgfscope}%
\pgfsetbuttcap%
\pgfsetroundjoin%
\definecolor{currentfill}{rgb}{0.000000,0.000000,0.000000}%
\pgfsetfillcolor{currentfill}%
\pgfsetlinewidth{0.803000pt}%
\definecolor{currentstroke}{rgb}{0.000000,0.000000,0.000000}%
\pgfsetstrokecolor{currentstroke}%
\pgfsetdash{}{0pt}%
\pgfsys@defobject{currentmarker}{\pgfqpoint{-0.048611in}{0.000000in}}{\pgfqpoint{0.000000in}{0.000000in}}{%
\pgfpathmoveto{\pgfqpoint{0.000000in}{0.000000in}}%
\pgfpathlineto{\pgfqpoint{-0.048611in}{0.000000in}}%
\pgfusepath{stroke,fill}%
}%
\begin{pgfscope}%
\pgfsys@transformshift{0.577431in}{1.090308in}%
\pgfsys@useobject{currentmarker}{}%
\end{pgfscope}%
\end{pgfscope}%
\begin{pgfscope}%
\definecolor{textcolor}{rgb}{0.000000,0.000000,0.000000}%
\pgfsetstrokecolor{textcolor}%
\pgfsetfillcolor{textcolor}%
\pgftext[x=0.279012in,y=1.042082in,left,base]{\color{textcolor}\rmfamily\fontsize{10.000000}{12.000000}\selectfont \(\displaystyle 10^{1}\)}%
\end{pgfscope}%
\begin{pgfscope}%
\pgfsetbuttcap%
\pgfsetroundjoin%
\definecolor{currentfill}{rgb}{0.000000,0.000000,0.000000}%
\pgfsetfillcolor{currentfill}%
\pgfsetlinewidth{0.803000pt}%
\definecolor{currentstroke}{rgb}{0.000000,0.000000,0.000000}%
\pgfsetstrokecolor{currentstroke}%
\pgfsetdash{}{0pt}%
\pgfsys@defobject{currentmarker}{\pgfqpoint{-0.048611in}{0.000000in}}{\pgfqpoint{0.000000in}{0.000000in}}{%
\pgfpathmoveto{\pgfqpoint{0.000000in}{0.000000in}}%
\pgfpathlineto{\pgfqpoint{-0.048611in}{0.000000in}}%
\pgfusepath{stroke,fill}%
}%
\begin{pgfscope}%
\pgfsys@transformshift{0.577431in}{2.451656in}%
\pgfsys@useobject{currentmarker}{}%
\end{pgfscope}%
\end{pgfscope}%
\begin{pgfscope}%
\definecolor{textcolor}{rgb}{0.000000,0.000000,0.000000}%
\pgfsetstrokecolor{textcolor}%
\pgfsetfillcolor{textcolor}%
\pgftext[x=0.279012in,y=2.403430in,left,base]{\color{textcolor}\rmfamily\fontsize{10.000000}{12.000000}\selectfont \(\displaystyle 10^{2}\)}%
\end{pgfscope}%
\begin{pgfscope}%
\pgfsetbuttcap%
\pgfsetroundjoin%
\definecolor{currentfill}{rgb}{0.000000,0.000000,0.000000}%
\pgfsetfillcolor{currentfill}%
\pgfsetlinewidth{0.602250pt}%
\definecolor{currentstroke}{rgb}{0.000000,0.000000,0.000000}%
\pgfsetstrokecolor{currentstroke}%
\pgfsetdash{}{0pt}%
\pgfsys@defobject{currentmarker}{\pgfqpoint{-0.027778in}{0.000000in}}{\pgfqpoint{0.000000in}{0.000000in}}{%
\pgfpathmoveto{\pgfqpoint{0.000000in}{0.000000in}}%
\pgfpathlineto{\pgfqpoint{-0.027778in}{0.000000in}}%
\pgfusepath{stroke,fill}%
}%
\begin{pgfscope}%
\pgfsys@transformshift{0.577431in}{0.548573in}%
\pgfsys@useobject{currentmarker}{}%
\end{pgfscope}%
\end{pgfscope}%
\begin{pgfscope}%
\pgfsetbuttcap%
\pgfsetroundjoin%
\definecolor{currentfill}{rgb}{0.000000,0.000000,0.000000}%
\pgfsetfillcolor{currentfill}%
\pgfsetlinewidth{0.602250pt}%
\definecolor{currentstroke}{rgb}{0.000000,0.000000,0.000000}%
\pgfsetstrokecolor{currentstroke}%
\pgfsetdash{}{0pt}%
\pgfsys@defobject{currentmarker}{\pgfqpoint{-0.027778in}{0.000000in}}{\pgfqpoint{0.000000in}{0.000000in}}{%
\pgfpathmoveto{\pgfqpoint{0.000000in}{0.000000in}}%
\pgfpathlineto{\pgfqpoint{-0.027778in}{0.000000in}}%
\pgfusepath{stroke,fill}%
}%
\begin{pgfscope}%
\pgfsys@transformshift{0.577431in}{0.680501in}%
\pgfsys@useobject{currentmarker}{}%
\end{pgfscope}%
\end{pgfscope}%
\begin{pgfscope}%
\pgfsetbuttcap%
\pgfsetroundjoin%
\definecolor{currentfill}{rgb}{0.000000,0.000000,0.000000}%
\pgfsetfillcolor{currentfill}%
\pgfsetlinewidth{0.602250pt}%
\definecolor{currentstroke}{rgb}{0.000000,0.000000,0.000000}%
\pgfsetstrokecolor{currentstroke}%
\pgfsetdash{}{0pt}%
\pgfsys@defobject{currentmarker}{\pgfqpoint{-0.027778in}{0.000000in}}{\pgfqpoint{0.000000in}{0.000000in}}{%
\pgfpathmoveto{\pgfqpoint{0.000000in}{0.000000in}}%
\pgfpathlineto{\pgfqpoint{-0.027778in}{0.000000in}}%
\pgfusepath{stroke,fill}%
}%
\begin{pgfscope}%
\pgfsys@transformshift{0.577431in}{0.788294in}%
\pgfsys@useobject{currentmarker}{}%
\end{pgfscope}%
\end{pgfscope}%
\begin{pgfscope}%
\pgfsetbuttcap%
\pgfsetroundjoin%
\definecolor{currentfill}{rgb}{0.000000,0.000000,0.000000}%
\pgfsetfillcolor{currentfill}%
\pgfsetlinewidth{0.602250pt}%
\definecolor{currentstroke}{rgb}{0.000000,0.000000,0.000000}%
\pgfsetstrokecolor{currentstroke}%
\pgfsetdash{}{0pt}%
\pgfsys@defobject{currentmarker}{\pgfqpoint{-0.027778in}{0.000000in}}{\pgfqpoint{0.000000in}{0.000000in}}{%
\pgfpathmoveto{\pgfqpoint{0.000000in}{0.000000in}}%
\pgfpathlineto{\pgfqpoint{-0.027778in}{0.000000in}}%
\pgfusepath{stroke,fill}%
}%
\begin{pgfscope}%
\pgfsys@transformshift{0.577431in}{0.879432in}%
\pgfsys@useobject{currentmarker}{}%
\end{pgfscope}%
\end{pgfscope}%
\begin{pgfscope}%
\pgfsetbuttcap%
\pgfsetroundjoin%
\definecolor{currentfill}{rgb}{0.000000,0.000000,0.000000}%
\pgfsetfillcolor{currentfill}%
\pgfsetlinewidth{0.602250pt}%
\definecolor{currentstroke}{rgb}{0.000000,0.000000,0.000000}%
\pgfsetstrokecolor{currentstroke}%
\pgfsetdash{}{0pt}%
\pgfsys@defobject{currentmarker}{\pgfqpoint{-0.027778in}{0.000000in}}{\pgfqpoint{0.000000in}{0.000000in}}{%
\pgfpathmoveto{\pgfqpoint{0.000000in}{0.000000in}}%
\pgfpathlineto{\pgfqpoint{-0.027778in}{0.000000in}}%
\pgfusepath{stroke,fill}%
}%
\begin{pgfscope}%
\pgfsys@transformshift{0.577431in}{0.958379in}%
\pgfsys@useobject{currentmarker}{}%
\end{pgfscope}%
\end{pgfscope}%
\begin{pgfscope}%
\pgfsetbuttcap%
\pgfsetroundjoin%
\definecolor{currentfill}{rgb}{0.000000,0.000000,0.000000}%
\pgfsetfillcolor{currentfill}%
\pgfsetlinewidth{0.602250pt}%
\definecolor{currentstroke}{rgb}{0.000000,0.000000,0.000000}%
\pgfsetstrokecolor{currentstroke}%
\pgfsetdash{}{0pt}%
\pgfsys@defobject{currentmarker}{\pgfqpoint{-0.027778in}{0.000000in}}{\pgfqpoint{0.000000in}{0.000000in}}{%
\pgfpathmoveto{\pgfqpoint{0.000000in}{0.000000in}}%
\pgfpathlineto{\pgfqpoint{-0.027778in}{0.000000in}}%
\pgfusepath{stroke,fill}%
}%
\begin{pgfscope}%
\pgfsys@transformshift{0.577431in}{1.028016in}%
\pgfsys@useobject{currentmarker}{}%
\end{pgfscope}%
\end{pgfscope}%
\begin{pgfscope}%
\pgfsetbuttcap%
\pgfsetroundjoin%
\definecolor{currentfill}{rgb}{0.000000,0.000000,0.000000}%
\pgfsetfillcolor{currentfill}%
\pgfsetlinewidth{0.602250pt}%
\definecolor{currentstroke}{rgb}{0.000000,0.000000,0.000000}%
\pgfsetstrokecolor{currentstroke}%
\pgfsetdash{}{0pt}%
\pgfsys@defobject{currentmarker}{\pgfqpoint{-0.027778in}{0.000000in}}{\pgfqpoint{0.000000in}{0.000000in}}{%
\pgfpathmoveto{\pgfqpoint{0.000000in}{0.000000in}}%
\pgfpathlineto{\pgfqpoint{-0.027778in}{0.000000in}}%
\pgfusepath{stroke,fill}%
}%
\begin{pgfscope}%
\pgfsys@transformshift{0.577431in}{1.500114in}%
\pgfsys@useobject{currentmarker}{}%
\end{pgfscope}%
\end{pgfscope}%
\begin{pgfscope}%
\pgfsetbuttcap%
\pgfsetroundjoin%
\definecolor{currentfill}{rgb}{0.000000,0.000000,0.000000}%
\pgfsetfillcolor{currentfill}%
\pgfsetlinewidth{0.602250pt}%
\definecolor{currentstroke}{rgb}{0.000000,0.000000,0.000000}%
\pgfsetstrokecolor{currentstroke}%
\pgfsetdash{}{0pt}%
\pgfsys@defobject{currentmarker}{\pgfqpoint{-0.027778in}{0.000000in}}{\pgfqpoint{0.000000in}{0.000000in}}{%
\pgfpathmoveto{\pgfqpoint{0.000000in}{0.000000in}}%
\pgfpathlineto{\pgfqpoint{-0.027778in}{0.000000in}}%
\pgfusepath{stroke,fill}%
}%
\begin{pgfscope}%
\pgfsys@transformshift{0.577431in}{1.739836in}%
\pgfsys@useobject{currentmarker}{}%
\end{pgfscope}%
\end{pgfscope}%
\begin{pgfscope}%
\pgfsetbuttcap%
\pgfsetroundjoin%
\definecolor{currentfill}{rgb}{0.000000,0.000000,0.000000}%
\pgfsetfillcolor{currentfill}%
\pgfsetlinewidth{0.602250pt}%
\definecolor{currentstroke}{rgb}{0.000000,0.000000,0.000000}%
\pgfsetstrokecolor{currentstroke}%
\pgfsetdash{}{0pt}%
\pgfsys@defobject{currentmarker}{\pgfqpoint{-0.027778in}{0.000000in}}{\pgfqpoint{0.000000in}{0.000000in}}{%
\pgfpathmoveto{\pgfqpoint{0.000000in}{0.000000in}}%
\pgfpathlineto{\pgfqpoint{-0.027778in}{0.000000in}}%
\pgfusepath{stroke,fill}%
}%
\begin{pgfscope}%
\pgfsys@transformshift{0.577431in}{1.909921in}%
\pgfsys@useobject{currentmarker}{}%
\end{pgfscope}%
\end{pgfscope}%
\begin{pgfscope}%
\pgfsetbuttcap%
\pgfsetroundjoin%
\definecolor{currentfill}{rgb}{0.000000,0.000000,0.000000}%
\pgfsetfillcolor{currentfill}%
\pgfsetlinewidth{0.602250pt}%
\definecolor{currentstroke}{rgb}{0.000000,0.000000,0.000000}%
\pgfsetstrokecolor{currentstroke}%
\pgfsetdash{}{0pt}%
\pgfsys@defobject{currentmarker}{\pgfqpoint{-0.027778in}{0.000000in}}{\pgfqpoint{0.000000in}{0.000000in}}{%
\pgfpathmoveto{\pgfqpoint{0.000000in}{0.000000in}}%
\pgfpathlineto{\pgfqpoint{-0.027778in}{0.000000in}}%
\pgfusepath{stroke,fill}%
}%
\begin{pgfscope}%
\pgfsys@transformshift{0.577431in}{2.041849in}%
\pgfsys@useobject{currentmarker}{}%
\end{pgfscope}%
\end{pgfscope}%
\begin{pgfscope}%
\pgfsetbuttcap%
\pgfsetroundjoin%
\definecolor{currentfill}{rgb}{0.000000,0.000000,0.000000}%
\pgfsetfillcolor{currentfill}%
\pgfsetlinewidth{0.602250pt}%
\definecolor{currentstroke}{rgb}{0.000000,0.000000,0.000000}%
\pgfsetstrokecolor{currentstroke}%
\pgfsetdash{}{0pt}%
\pgfsys@defobject{currentmarker}{\pgfqpoint{-0.027778in}{0.000000in}}{\pgfqpoint{0.000000in}{0.000000in}}{%
\pgfpathmoveto{\pgfqpoint{0.000000in}{0.000000in}}%
\pgfpathlineto{\pgfqpoint{-0.027778in}{0.000000in}}%
\pgfusepath{stroke,fill}%
}%
\begin{pgfscope}%
\pgfsys@transformshift{0.577431in}{2.149642in}%
\pgfsys@useobject{currentmarker}{}%
\end{pgfscope}%
\end{pgfscope}%
\begin{pgfscope}%
\pgfsetbuttcap%
\pgfsetroundjoin%
\definecolor{currentfill}{rgb}{0.000000,0.000000,0.000000}%
\pgfsetfillcolor{currentfill}%
\pgfsetlinewidth{0.602250pt}%
\definecolor{currentstroke}{rgb}{0.000000,0.000000,0.000000}%
\pgfsetstrokecolor{currentstroke}%
\pgfsetdash{}{0pt}%
\pgfsys@defobject{currentmarker}{\pgfqpoint{-0.027778in}{0.000000in}}{\pgfqpoint{0.000000in}{0.000000in}}{%
\pgfpathmoveto{\pgfqpoint{0.000000in}{0.000000in}}%
\pgfpathlineto{\pgfqpoint{-0.027778in}{0.000000in}}%
\pgfusepath{stroke,fill}%
}%
\begin{pgfscope}%
\pgfsys@transformshift{0.577431in}{2.240780in}%
\pgfsys@useobject{currentmarker}{}%
\end{pgfscope}%
\end{pgfscope}%
\begin{pgfscope}%
\pgfsetbuttcap%
\pgfsetroundjoin%
\definecolor{currentfill}{rgb}{0.000000,0.000000,0.000000}%
\pgfsetfillcolor{currentfill}%
\pgfsetlinewidth{0.602250pt}%
\definecolor{currentstroke}{rgb}{0.000000,0.000000,0.000000}%
\pgfsetstrokecolor{currentstroke}%
\pgfsetdash{}{0pt}%
\pgfsys@defobject{currentmarker}{\pgfqpoint{-0.027778in}{0.000000in}}{\pgfqpoint{0.000000in}{0.000000in}}{%
\pgfpathmoveto{\pgfqpoint{0.000000in}{0.000000in}}%
\pgfpathlineto{\pgfqpoint{-0.027778in}{0.000000in}}%
\pgfusepath{stroke,fill}%
}%
\begin{pgfscope}%
\pgfsys@transformshift{0.577431in}{2.319727in}%
\pgfsys@useobject{currentmarker}{}%
\end{pgfscope}%
\end{pgfscope}%
\begin{pgfscope}%
\pgfsetbuttcap%
\pgfsetroundjoin%
\definecolor{currentfill}{rgb}{0.000000,0.000000,0.000000}%
\pgfsetfillcolor{currentfill}%
\pgfsetlinewidth{0.602250pt}%
\definecolor{currentstroke}{rgb}{0.000000,0.000000,0.000000}%
\pgfsetstrokecolor{currentstroke}%
\pgfsetdash{}{0pt}%
\pgfsys@defobject{currentmarker}{\pgfqpoint{-0.027778in}{0.000000in}}{\pgfqpoint{0.000000in}{0.000000in}}{%
\pgfpathmoveto{\pgfqpoint{0.000000in}{0.000000in}}%
\pgfpathlineto{\pgfqpoint{-0.027778in}{0.000000in}}%
\pgfusepath{stroke,fill}%
}%
\begin{pgfscope}%
\pgfsys@transformshift{0.577431in}{2.389364in}%
\pgfsys@useobject{currentmarker}{}%
\end{pgfscope}%
\end{pgfscope}%
\begin{pgfscope}%
\pgfsetbuttcap%
\pgfsetroundjoin%
\definecolor{currentfill}{rgb}{0.000000,0.000000,0.000000}%
\pgfsetfillcolor{currentfill}%
\pgfsetlinewidth{0.602250pt}%
\definecolor{currentstroke}{rgb}{0.000000,0.000000,0.000000}%
\pgfsetstrokecolor{currentstroke}%
\pgfsetdash{}{0pt}%
\pgfsys@defobject{currentmarker}{\pgfqpoint{-0.027778in}{0.000000in}}{\pgfqpoint{0.000000in}{0.000000in}}{%
\pgfpathmoveto{\pgfqpoint{0.000000in}{0.000000in}}%
\pgfpathlineto{\pgfqpoint{-0.027778in}{0.000000in}}%
\pgfusepath{stroke,fill}%
}%
\begin{pgfscope}%
\pgfsys@transformshift{0.577431in}{2.861462in}%
\pgfsys@useobject{currentmarker}{}%
\end{pgfscope}%
\end{pgfscope}%
\begin{pgfscope}%
\definecolor{textcolor}{rgb}{0.000000,0.000000,0.000000}%
\pgfsetstrokecolor{textcolor}%
\pgfsetfillcolor{textcolor}%
\pgftext[x=0.223457in,y=1.800309in,,bottom,rotate=90.000000]{\color{textcolor}\rmfamily\fontsize{10.000000}{12.000000}\selectfont Primal-Dual gap}%
\end{pgfscope}%
\begin{pgfscope}%
\pgfpathrectangle{\pgfqpoint{0.577431in}{0.499691in}}{\pgfqpoint{2.722569in}{2.601235in}}%
\pgfusepath{clip}%
\pgfsetrectcap%
\pgfsetroundjoin%
\pgfsetlinewidth{1.505625pt}%
\definecolor{currentstroke}{rgb}{0.121569,0.466667,0.705882}%
\pgfsetstrokecolor{currentstroke}%
\pgfsetdash{}{0pt}%
\pgfpathmoveto{\pgfqpoint{0.701184in}{2.982688in}}%
\pgfpathlineto{\pgfqpoint{1.526205in}{2.884820in}}%
\pgfpathlineto{\pgfqpoint{2.351226in}{2.441431in}}%
\pgfpathlineto{\pgfqpoint{3.176247in}{2.331084in}}%
\pgfusepath{stroke}%
\end{pgfscope}%
\begin{pgfscope}%
\pgfpathrectangle{\pgfqpoint{0.577431in}{0.499691in}}{\pgfqpoint{2.722569in}{2.601235in}}%
\pgfusepath{clip}%
\pgfsetbuttcap%
\pgfsetroundjoin%
\definecolor{currentfill}{rgb}{0.121569,0.466667,0.705882}%
\pgfsetfillcolor{currentfill}%
\pgfsetlinewidth{1.003750pt}%
\definecolor{currentstroke}{rgb}{0.121569,0.466667,0.705882}%
\pgfsetstrokecolor{currentstroke}%
\pgfsetdash{}{0pt}%
\pgfsys@defobject{currentmarker}{\pgfqpoint{-0.041667in}{-0.041667in}}{\pgfqpoint{0.041667in}{0.041667in}}{%
\pgfpathmoveto{\pgfqpoint{-0.041667in}{-0.041667in}}%
\pgfpathlineto{\pgfqpoint{0.041667in}{0.041667in}}%
\pgfpathmoveto{\pgfqpoint{-0.041667in}{0.041667in}}%
\pgfpathlineto{\pgfqpoint{0.041667in}{-0.041667in}}%
\pgfusepath{stroke,fill}%
}%
\begin{pgfscope}%
\pgfsys@transformshift{0.701184in}{2.982688in}%
\pgfsys@useobject{currentmarker}{}%
\end{pgfscope}%
\begin{pgfscope}%
\pgfsys@transformshift{1.526205in}{2.884820in}%
\pgfsys@useobject{currentmarker}{}%
\end{pgfscope}%
\begin{pgfscope}%
\pgfsys@transformshift{2.351226in}{2.441431in}%
\pgfsys@useobject{currentmarker}{}%
\end{pgfscope}%
\begin{pgfscope}%
\pgfsys@transformshift{3.176247in}{2.331084in}%
\pgfsys@useobject{currentmarker}{}%
\end{pgfscope}%
\end{pgfscope}%
\begin{pgfscope}%
\pgfpathrectangle{\pgfqpoint{0.577431in}{0.499691in}}{\pgfqpoint{2.722569in}{2.601235in}}%
\pgfusepath{clip}%
\pgfsetrectcap%
\pgfsetroundjoin%
\pgfsetlinewidth{1.505625pt}%
\definecolor{currentstroke}{rgb}{1.000000,0.498039,0.054902}%
\pgfsetstrokecolor{currentstroke}%
\pgfsetdash{}{0pt}%
\pgfpathmoveto{\pgfqpoint{0.701184in}{2.846694in}}%
\pgfpathlineto{\pgfqpoint{1.526205in}{2.480667in}}%
\pgfpathlineto{\pgfqpoint{2.351226in}{2.016686in}}%
\pgfpathlineto{\pgfqpoint{3.176247in}{1.857681in}}%
\pgfusepath{stroke}%
\end{pgfscope}%
\begin{pgfscope}%
\pgfpathrectangle{\pgfqpoint{0.577431in}{0.499691in}}{\pgfqpoint{2.722569in}{2.601235in}}%
\pgfusepath{clip}%
\pgfsetbuttcap%
\pgfsetroundjoin%
\definecolor{currentfill}{rgb}{1.000000,0.498039,0.054902}%
\pgfsetfillcolor{currentfill}%
\pgfsetlinewidth{1.003750pt}%
\definecolor{currentstroke}{rgb}{1.000000,0.498039,0.054902}%
\pgfsetstrokecolor{currentstroke}%
\pgfsetdash{}{0pt}%
\pgfsys@defobject{currentmarker}{\pgfqpoint{-0.041667in}{-0.041667in}}{\pgfqpoint{0.041667in}{0.041667in}}{%
\pgfpathmoveto{\pgfqpoint{-0.041667in}{-0.041667in}}%
\pgfpathlineto{\pgfqpoint{0.041667in}{0.041667in}}%
\pgfpathmoveto{\pgfqpoint{-0.041667in}{0.041667in}}%
\pgfpathlineto{\pgfqpoint{0.041667in}{-0.041667in}}%
\pgfusepath{stroke,fill}%
}%
\begin{pgfscope}%
\pgfsys@transformshift{0.701184in}{2.846694in}%
\pgfsys@useobject{currentmarker}{}%
\end{pgfscope}%
\begin{pgfscope}%
\pgfsys@transformshift{1.526205in}{2.480667in}%
\pgfsys@useobject{currentmarker}{}%
\end{pgfscope}%
\begin{pgfscope}%
\pgfsys@transformshift{2.351226in}{2.016686in}%
\pgfsys@useobject{currentmarker}{}%
\end{pgfscope}%
\begin{pgfscope}%
\pgfsys@transformshift{3.176247in}{1.857681in}%
\pgfsys@useobject{currentmarker}{}%
\end{pgfscope}%
\end{pgfscope}%
\begin{pgfscope}%
\pgfpathrectangle{\pgfqpoint{0.577431in}{0.499691in}}{\pgfqpoint{2.722569in}{2.601235in}}%
\pgfusepath{clip}%
\pgfsetrectcap%
\pgfsetroundjoin%
\pgfsetlinewidth{1.505625pt}%
\definecolor{currentstroke}{rgb}{0.172549,0.627451,0.172549}%
\pgfsetstrokecolor{currentstroke}%
\pgfsetdash{}{0pt}%
\pgfpathmoveto{\pgfqpoint{0.701184in}{2.841763in}}%
\pgfpathlineto{\pgfqpoint{1.526205in}{2.200529in}}%
\pgfpathlineto{\pgfqpoint{2.351226in}{1.898700in}}%
\pgfpathlineto{\pgfqpoint{3.176247in}{1.339133in}}%
\pgfusepath{stroke}%
\end{pgfscope}%
\begin{pgfscope}%
\pgfpathrectangle{\pgfqpoint{0.577431in}{0.499691in}}{\pgfqpoint{2.722569in}{2.601235in}}%
\pgfusepath{clip}%
\pgfsetbuttcap%
\pgfsetroundjoin%
\definecolor{currentfill}{rgb}{0.172549,0.627451,0.172549}%
\pgfsetfillcolor{currentfill}%
\pgfsetlinewidth{1.003750pt}%
\definecolor{currentstroke}{rgb}{0.172549,0.627451,0.172549}%
\pgfsetstrokecolor{currentstroke}%
\pgfsetdash{}{0pt}%
\pgfsys@defobject{currentmarker}{\pgfqpoint{-0.041667in}{-0.041667in}}{\pgfqpoint{0.041667in}{0.041667in}}{%
\pgfpathmoveto{\pgfqpoint{-0.041667in}{-0.041667in}}%
\pgfpathlineto{\pgfqpoint{0.041667in}{0.041667in}}%
\pgfpathmoveto{\pgfqpoint{-0.041667in}{0.041667in}}%
\pgfpathlineto{\pgfqpoint{0.041667in}{-0.041667in}}%
\pgfusepath{stroke,fill}%
}%
\begin{pgfscope}%
\pgfsys@transformshift{0.701184in}{2.841763in}%
\pgfsys@useobject{currentmarker}{}%
\end{pgfscope}%
\begin{pgfscope}%
\pgfsys@transformshift{1.526205in}{2.200529in}%
\pgfsys@useobject{currentmarker}{}%
\end{pgfscope}%
\begin{pgfscope}%
\pgfsys@transformshift{2.351226in}{1.898700in}%
\pgfsys@useobject{currentmarker}{}%
\end{pgfscope}%
\begin{pgfscope}%
\pgfsys@transformshift{3.176247in}{1.339133in}%
\pgfsys@useobject{currentmarker}{}%
\end{pgfscope}%
\end{pgfscope}%
\begin{pgfscope}%
\pgfpathrectangle{\pgfqpoint{0.577431in}{0.499691in}}{\pgfqpoint{2.722569in}{2.601235in}}%
\pgfusepath{clip}%
\pgfsetrectcap%
\pgfsetroundjoin%
\pgfsetlinewidth{1.505625pt}%
\definecolor{currentstroke}{rgb}{0.839216,0.152941,0.156863}%
\pgfsetstrokecolor{currentstroke}%
\pgfsetdash{}{0pt}%
\pgfpathmoveto{\pgfqpoint{0.701184in}{2.674166in}}%
\pgfpathlineto{\pgfqpoint{1.526205in}{2.072458in}}%
\pgfpathlineto{\pgfqpoint{2.351226in}{1.743295in}}%
\pgfpathlineto{\pgfqpoint{3.176247in}{0.617929in}}%
\pgfusepath{stroke}%
\end{pgfscope}%
\begin{pgfscope}%
\pgfpathrectangle{\pgfqpoint{0.577431in}{0.499691in}}{\pgfqpoint{2.722569in}{2.601235in}}%
\pgfusepath{clip}%
\pgfsetbuttcap%
\pgfsetroundjoin%
\definecolor{currentfill}{rgb}{0.839216,0.152941,0.156863}%
\pgfsetfillcolor{currentfill}%
\pgfsetlinewidth{1.003750pt}%
\definecolor{currentstroke}{rgb}{0.839216,0.152941,0.156863}%
\pgfsetstrokecolor{currentstroke}%
\pgfsetdash{}{0pt}%
\pgfsys@defobject{currentmarker}{\pgfqpoint{-0.041667in}{-0.041667in}}{\pgfqpoint{0.041667in}{0.041667in}}{%
\pgfpathmoveto{\pgfqpoint{-0.041667in}{-0.041667in}}%
\pgfpathlineto{\pgfqpoint{0.041667in}{0.041667in}}%
\pgfpathmoveto{\pgfqpoint{-0.041667in}{0.041667in}}%
\pgfpathlineto{\pgfqpoint{0.041667in}{-0.041667in}}%
\pgfusepath{stroke,fill}%
}%
\begin{pgfscope}%
\pgfsys@transformshift{0.701184in}{2.674166in}%
\pgfsys@useobject{currentmarker}{}%
\end{pgfscope}%
\begin{pgfscope}%
\pgfsys@transformshift{1.526205in}{2.072458in}%
\pgfsys@useobject{currentmarker}{}%
\end{pgfscope}%
\begin{pgfscope}%
\pgfsys@transformshift{2.351226in}{1.743295in}%
\pgfsys@useobject{currentmarker}{}%
\end{pgfscope}%
\begin{pgfscope}%
\pgfsys@transformshift{3.176247in}{0.617929in}%
\pgfsys@useobject{currentmarker}{}%
\end{pgfscope}%
\end{pgfscope}%
\begin{pgfscope}%
\pgfsetrectcap%
\pgfsetmiterjoin%
\pgfsetlinewidth{0.803000pt}%
\definecolor{currentstroke}{rgb}{0.000000,0.000000,0.000000}%
\pgfsetstrokecolor{currentstroke}%
\pgfsetdash{}{0pt}%
\pgfpathmoveto{\pgfqpoint{0.577431in}{0.499691in}}%
\pgfpathlineto{\pgfqpoint{0.577431in}{3.100926in}}%
\pgfusepath{stroke}%
\end{pgfscope}%
\begin{pgfscope}%
\pgfsetrectcap%
\pgfsetmiterjoin%
\pgfsetlinewidth{0.803000pt}%
\definecolor{currentstroke}{rgb}{0.000000,0.000000,0.000000}%
\pgfsetstrokecolor{currentstroke}%
\pgfsetdash{}{0pt}%
\pgfpathmoveto{\pgfqpoint{3.300000in}{0.499691in}}%
\pgfpathlineto{\pgfqpoint{3.300000in}{3.100926in}}%
\pgfusepath{stroke}%
\end{pgfscope}%
\begin{pgfscope}%
\pgfsetrectcap%
\pgfsetmiterjoin%
\pgfsetlinewidth{0.803000pt}%
\definecolor{currentstroke}{rgb}{0.000000,0.000000,0.000000}%
\pgfsetstrokecolor{currentstroke}%
\pgfsetdash{}{0pt}%
\pgfpathmoveto{\pgfqpoint{0.577431in}{0.499691in}}%
\pgfpathlineto{\pgfqpoint{3.300000in}{0.499691in}}%
\pgfusepath{stroke}%
\end{pgfscope}%
\begin{pgfscope}%
\pgfsetrectcap%
\pgfsetmiterjoin%
\pgfsetlinewidth{0.803000pt}%
\definecolor{currentstroke}{rgb}{0.000000,0.000000,0.000000}%
\pgfsetstrokecolor{currentstroke}%
\pgfsetdash{}{0pt}%
\pgfpathmoveto{\pgfqpoint{0.577431in}{3.100926in}}%
\pgfpathlineto{\pgfqpoint{3.300000in}{3.100926in}}%
\pgfusepath{stroke}%
\end{pgfscope}%
\begin{pgfscope}%
\definecolor{textcolor}{rgb}{0.000000,0.000000,0.000000}%
\pgfsetstrokecolor{textcolor}%
\pgfsetfillcolor{textcolor}%
\pgftext[x=1.938716in,y=3.184260in,,base]{\color{textcolor}\rmfamily\fontsize{12.000000}{14.400000}\selectfont TV regularization}%
\end{pgfscope}%
\begin{pgfscope}%
\pgfsetbuttcap%
\pgfsetmiterjoin%
\definecolor{currentfill}{rgb}{1.000000,1.000000,1.000000}%
\pgfsetfillcolor{currentfill}%
\pgfsetfillopacity{0.800000}%
\pgfsetlinewidth{1.003750pt}%
\definecolor{currentstroke}{rgb}{0.800000,0.800000,0.800000}%
\pgfsetstrokecolor{currentstroke}%
\pgfsetstrokeopacity{0.800000}%
\pgfsetdash{}{0pt}%
\pgfpathmoveto{\pgfqpoint{0.674653in}{0.569136in}}%
\pgfpathlineto{\pgfqpoint{1.543481in}{0.569136in}}%
\pgfpathquadraticcurveto{\pgfqpoint{1.571258in}{0.569136in}}{\pgfqpoint{1.571258in}{0.596913in}}%
\pgfpathlineto{\pgfqpoint{1.571258in}{1.357716in}}%
\pgfpathquadraticcurveto{\pgfqpoint{1.571258in}{1.385493in}}{\pgfqpoint{1.543481in}{1.385493in}}%
\pgfpathlineto{\pgfqpoint{0.674653in}{1.385493in}}%
\pgfpathquadraticcurveto{\pgfqpoint{0.646876in}{1.385493in}}{\pgfqpoint{0.646876in}{1.357716in}}%
\pgfpathlineto{\pgfqpoint{0.646876in}{0.596913in}}%
\pgfpathquadraticcurveto{\pgfqpoint{0.646876in}{0.569136in}}{\pgfqpoint{0.674653in}{0.569136in}}%
\pgfpathclose%
\pgfusepath{stroke,fill}%
\end{pgfscope}%
\begin{pgfscope}%
\pgfsetrectcap%
\pgfsetroundjoin%
\pgfsetlinewidth{1.505625pt}%
\definecolor{currentstroke}{rgb}{0.121569,0.466667,0.705882}%
\pgfsetstrokecolor{currentstroke}%
\pgfsetdash{}{0pt}%
\pgfpathmoveto{\pgfqpoint{0.702431in}{1.281327in}}%
\pgfpathlineto{\pgfqpoint{0.980209in}{1.281327in}}%
\pgfusepath{stroke}%
\end{pgfscope}%
\begin{pgfscope}%
\pgfsetbuttcap%
\pgfsetroundjoin%
\definecolor{currentfill}{rgb}{0.121569,0.466667,0.705882}%
\pgfsetfillcolor{currentfill}%
\pgfsetlinewidth{1.003750pt}%
\definecolor{currentstroke}{rgb}{0.121569,0.466667,0.705882}%
\pgfsetstrokecolor{currentstroke}%
\pgfsetdash{}{0pt}%
\pgfsys@defobject{currentmarker}{\pgfqpoint{-0.041667in}{-0.041667in}}{\pgfqpoint{0.041667in}{0.041667in}}{%
\pgfpathmoveto{\pgfqpoint{-0.041667in}{-0.041667in}}%
\pgfpathlineto{\pgfqpoint{0.041667in}{0.041667in}}%
\pgfpathmoveto{\pgfqpoint{-0.041667in}{0.041667in}}%
\pgfpathlineto{\pgfqpoint{0.041667in}{-0.041667in}}%
\pgfusepath{stroke,fill}%
}%
\begin{pgfscope}%
\pgfsys@transformshift{0.841320in}{1.281327in}%
\pgfsys@useobject{currentmarker}{}%
\end{pgfscope}%
\end{pgfscope}%
\begin{pgfscope}%
\definecolor{textcolor}{rgb}{0.000000,0.000000,0.000000}%
\pgfsetstrokecolor{textcolor}%
\pgfsetfillcolor{textcolor}%
\pgftext[x=1.091320in,y=1.232716in,left,base]{\color{textcolor}\rmfamily\fontsize{10.000000}{12.000000}\selectfont deg\(\displaystyle =1\)}%
\end{pgfscope}%
\begin{pgfscope}%
\pgfsetrectcap%
\pgfsetroundjoin%
\pgfsetlinewidth{1.505625pt}%
\definecolor{currentstroke}{rgb}{1.000000,0.498039,0.054902}%
\pgfsetstrokecolor{currentstroke}%
\pgfsetdash{}{0pt}%
\pgfpathmoveto{\pgfqpoint{0.702431in}{1.087654in}}%
\pgfpathlineto{\pgfqpoint{0.980209in}{1.087654in}}%
\pgfusepath{stroke}%
\end{pgfscope}%
\begin{pgfscope}%
\pgfsetbuttcap%
\pgfsetroundjoin%
\definecolor{currentfill}{rgb}{1.000000,0.498039,0.054902}%
\pgfsetfillcolor{currentfill}%
\pgfsetlinewidth{1.003750pt}%
\definecolor{currentstroke}{rgb}{1.000000,0.498039,0.054902}%
\pgfsetstrokecolor{currentstroke}%
\pgfsetdash{}{0pt}%
\pgfsys@defobject{currentmarker}{\pgfqpoint{-0.041667in}{-0.041667in}}{\pgfqpoint{0.041667in}{0.041667in}}{%
\pgfpathmoveto{\pgfqpoint{-0.041667in}{-0.041667in}}%
\pgfpathlineto{\pgfqpoint{0.041667in}{0.041667in}}%
\pgfpathmoveto{\pgfqpoint{-0.041667in}{0.041667in}}%
\pgfpathlineto{\pgfqpoint{0.041667in}{-0.041667in}}%
\pgfusepath{stroke,fill}%
}%
\begin{pgfscope}%
\pgfsys@transformshift{0.841320in}{1.087654in}%
\pgfsys@useobject{currentmarker}{}%
\end{pgfscope}%
\end{pgfscope}%
\begin{pgfscope}%
\definecolor{textcolor}{rgb}{0.000000,0.000000,0.000000}%
\pgfsetstrokecolor{textcolor}%
\pgfsetfillcolor{textcolor}%
\pgftext[x=1.091320in,y=1.039043in,left,base]{\color{textcolor}\rmfamily\fontsize{10.000000}{12.000000}\selectfont deg\(\displaystyle =2\)}%
\end{pgfscope}%
\begin{pgfscope}%
\pgfsetrectcap%
\pgfsetroundjoin%
\pgfsetlinewidth{1.505625pt}%
\definecolor{currentstroke}{rgb}{0.172549,0.627451,0.172549}%
\pgfsetstrokecolor{currentstroke}%
\pgfsetdash{}{0pt}%
\pgfpathmoveto{\pgfqpoint{0.702431in}{0.893981in}}%
\pgfpathlineto{\pgfqpoint{0.980209in}{0.893981in}}%
\pgfusepath{stroke}%
\end{pgfscope}%
\begin{pgfscope}%
\pgfsetbuttcap%
\pgfsetroundjoin%
\definecolor{currentfill}{rgb}{0.172549,0.627451,0.172549}%
\pgfsetfillcolor{currentfill}%
\pgfsetlinewidth{1.003750pt}%
\definecolor{currentstroke}{rgb}{0.172549,0.627451,0.172549}%
\pgfsetstrokecolor{currentstroke}%
\pgfsetdash{}{0pt}%
\pgfsys@defobject{currentmarker}{\pgfqpoint{-0.041667in}{-0.041667in}}{\pgfqpoint{0.041667in}{0.041667in}}{%
\pgfpathmoveto{\pgfqpoint{-0.041667in}{-0.041667in}}%
\pgfpathlineto{\pgfqpoint{0.041667in}{0.041667in}}%
\pgfpathmoveto{\pgfqpoint{-0.041667in}{0.041667in}}%
\pgfpathlineto{\pgfqpoint{0.041667in}{-0.041667in}}%
\pgfusepath{stroke,fill}%
}%
\begin{pgfscope}%
\pgfsys@transformshift{0.841320in}{0.893981in}%
\pgfsys@useobject{currentmarker}{}%
\end{pgfscope}%
\end{pgfscope}%
\begin{pgfscope}%
\definecolor{textcolor}{rgb}{0.000000,0.000000,0.000000}%
\pgfsetstrokecolor{textcolor}%
\pgfsetfillcolor{textcolor}%
\pgftext[x=1.091320in,y=0.845370in,left,base]{\color{textcolor}\rmfamily\fontsize{10.000000}{12.000000}\selectfont deg\(\displaystyle =3\)}%
\end{pgfscope}%
\begin{pgfscope}%
\pgfsetrectcap%
\pgfsetroundjoin%
\pgfsetlinewidth{1.505625pt}%
\definecolor{currentstroke}{rgb}{0.839216,0.152941,0.156863}%
\pgfsetstrokecolor{currentstroke}%
\pgfsetdash{}{0pt}%
\pgfpathmoveto{\pgfqpoint{0.702431in}{0.700308in}}%
\pgfpathlineto{\pgfqpoint{0.980209in}{0.700308in}}%
\pgfusepath{stroke}%
\end{pgfscope}%
\begin{pgfscope}%
\pgfsetbuttcap%
\pgfsetroundjoin%
\definecolor{currentfill}{rgb}{0.839216,0.152941,0.156863}%
\pgfsetfillcolor{currentfill}%
\pgfsetlinewidth{1.003750pt}%
\definecolor{currentstroke}{rgb}{0.839216,0.152941,0.156863}%
\pgfsetstrokecolor{currentstroke}%
\pgfsetdash{}{0pt}%
\pgfsys@defobject{currentmarker}{\pgfqpoint{-0.041667in}{-0.041667in}}{\pgfqpoint{0.041667in}{0.041667in}}{%
\pgfpathmoveto{\pgfqpoint{-0.041667in}{-0.041667in}}%
\pgfpathlineto{\pgfqpoint{0.041667in}{0.041667in}}%
\pgfpathmoveto{\pgfqpoint{-0.041667in}{0.041667in}}%
\pgfpathlineto{\pgfqpoint{0.041667in}{-0.041667in}}%
\pgfusepath{stroke,fill}%
}%
\begin{pgfscope}%
\pgfsys@transformshift{0.841320in}{0.700308in}%
\pgfsys@useobject{currentmarker}{}%
\end{pgfscope}%
\end{pgfscope}%
\begin{pgfscope}%
\definecolor{textcolor}{rgb}{0.000000,0.000000,0.000000}%
\pgfsetstrokecolor{textcolor}%
\pgfsetfillcolor{textcolor}%
\pgftext[x=1.091320in,y=0.651697in,left,base]{\color{textcolor}\rmfamily\fontsize{10.000000}{12.000000}\selectfont deg\(\displaystyle =4\)}%
\end{pgfscope}%
\end{pgfpicture}%
\makeatother%
\endgroup%

%% file: figures/energiesl0.pgf
\begingroup%
\makeatletter%
\begin{pgfpicture}%
\pgfpathrectangle{\pgfpointorigin}{\pgfqpoint{3.400000in}{3.400000in}}%
\pgfusepath{use as bounding box, clip}%
\begin{pgfscope}%
\pgfsetbuttcap%
\pgfsetmiterjoin%
\definecolor{currentfill}{rgb}{1.000000,1.000000,1.000000}%
\pgfsetfillcolor{currentfill}%
\pgfsetlinewidth{0.000000pt}%
\definecolor{currentstroke}{rgb}{1.000000,1.000000,1.000000}%
\pgfsetstrokecolor{currentstroke}%
\pgfsetdash{}{0pt}%
\pgfpathmoveto{\pgfqpoint{0.000000in}{0.000000in}}%
\pgfpathlineto{\pgfqpoint{3.400000in}{0.000000in}}%
\pgfpathlineto{\pgfqpoint{3.400000in}{3.400000in}}%
\pgfpathlineto{\pgfqpoint{0.000000in}{3.400000in}}%
\pgfpathclose%
\pgfusepath{fill}%
\end{pgfscope}%
\begin{pgfscope}%
\pgfsetbuttcap%
\pgfsetmiterjoin%
\definecolor{currentfill}{rgb}{1.000000,1.000000,1.000000}%
\pgfsetfillcolor{currentfill}%
\pgfsetlinewidth{0.000000pt}%
\definecolor{currentstroke}{rgb}{0.000000,0.000000,0.000000}%
\pgfsetstrokecolor{currentstroke}%
\pgfsetstrokeopacity{0.000000}%
\pgfsetdash{}{0pt}%
\pgfpathmoveto{\pgfqpoint{0.692593in}{0.499691in}}%
\pgfpathlineto{\pgfqpoint{3.300000in}{0.499691in}}%
\pgfpathlineto{\pgfqpoint{3.300000in}{3.100926in}}%
\pgfpathlineto{\pgfqpoint{0.692593in}{3.100926in}}%
\pgfpathclose%
\pgfusepath{fill}%
\end{pgfscope}%
\begin{pgfscope}%
\pgfsetbuttcap%
\pgfsetroundjoin%
\definecolor{currentfill}{rgb}{0.000000,0.000000,0.000000}%
\pgfsetfillcolor{currentfill}%
\pgfsetlinewidth{0.803000pt}%
\definecolor{currentstroke}{rgb}{0.000000,0.000000,0.000000}%
\pgfsetstrokecolor{currentstroke}%
\pgfsetdash{}{0pt}%
\pgfsys@defobject{currentmarker}{\pgfqpoint{0.000000in}{-0.048611in}}{\pgfqpoint{0.000000in}{0.000000in}}{%
\pgfpathmoveto{\pgfqpoint{0.000000in}{0.000000in}}%
\pgfpathlineto{\pgfqpoint{0.000000in}{-0.048611in}}%
\pgfusepath{stroke,fill}%
}%
\begin{pgfscope}%
\pgfsys@transformshift{0.811112in}{0.499691in}%
\pgfsys@useobject{currentmarker}{}%
\end{pgfscope}%
\end{pgfscope}%
\begin{pgfscope}%
\definecolor{textcolor}{rgb}{0.000000,0.000000,0.000000}%
\pgfsetstrokecolor{textcolor}%
\pgfsetfillcolor{textcolor}%
\pgftext[x=0.811112in,y=0.402469in,,top]{\color{textcolor}\rmfamily\fontsize{10.000000}{12.000000}\selectfont \(\displaystyle 1\)}%
\end{pgfscope}%
\begin{pgfscope}%
\pgfsetbuttcap%
\pgfsetroundjoin%
\definecolor{currentfill}{rgb}{0.000000,0.000000,0.000000}%
\pgfsetfillcolor{currentfill}%
\pgfsetlinewidth{0.803000pt}%
\definecolor{currentstroke}{rgb}{0.000000,0.000000,0.000000}%
\pgfsetstrokecolor{currentstroke}%
\pgfsetdash{}{0pt}%
\pgfsys@defobject{currentmarker}{\pgfqpoint{0.000000in}{-0.048611in}}{\pgfqpoint{0.000000in}{0.000000in}}{%
\pgfpathmoveto{\pgfqpoint{0.000000in}{0.000000in}}%
\pgfpathlineto{\pgfqpoint{0.000000in}{-0.048611in}}%
\pgfusepath{stroke,fill}%
}%
\begin{pgfscope}%
\pgfsys@transformshift{1.601235in}{0.499691in}%
\pgfsys@useobject{currentmarker}{}%
\end{pgfscope}%
\end{pgfscope}%
\begin{pgfscope}%
\definecolor{textcolor}{rgb}{0.000000,0.000000,0.000000}%
\pgfsetstrokecolor{textcolor}%
\pgfsetfillcolor{textcolor}%
\pgftext[x=1.601235in,y=0.402469in,,top]{\color{textcolor}\rmfamily\fontsize{10.000000}{12.000000}\selectfont \(\displaystyle 2\)}%
\end{pgfscope}%
\begin{pgfscope}%
\pgfsetbuttcap%
\pgfsetroundjoin%
\definecolor{currentfill}{rgb}{0.000000,0.000000,0.000000}%
\pgfsetfillcolor{currentfill}%
\pgfsetlinewidth{0.803000pt}%
\definecolor{currentstroke}{rgb}{0.000000,0.000000,0.000000}%
\pgfsetstrokecolor{currentstroke}%
\pgfsetdash{}{0pt}%
\pgfsys@defobject{currentmarker}{\pgfqpoint{0.000000in}{-0.048611in}}{\pgfqpoint{0.000000in}{0.000000in}}{%
\pgfpathmoveto{\pgfqpoint{0.000000in}{0.000000in}}%
\pgfpathlineto{\pgfqpoint{0.000000in}{-0.048611in}}%
\pgfusepath{stroke,fill}%
}%
\begin{pgfscope}%
\pgfsys@transformshift{2.391358in}{0.499691in}%
\pgfsys@useobject{currentmarker}{}%
\end{pgfscope}%
\end{pgfscope}%
\begin{pgfscope}%
\definecolor{textcolor}{rgb}{0.000000,0.000000,0.000000}%
\pgfsetstrokecolor{textcolor}%
\pgfsetfillcolor{textcolor}%
\pgftext[x=2.391358in,y=0.402469in,,top]{\color{textcolor}\rmfamily\fontsize{10.000000}{12.000000}\selectfont \(\displaystyle 3\)}%
\end{pgfscope}%
\begin{pgfscope}%
\pgfsetbuttcap%
\pgfsetroundjoin%
\definecolor{currentfill}{rgb}{0.000000,0.000000,0.000000}%
\pgfsetfillcolor{currentfill}%
\pgfsetlinewidth{0.803000pt}%
\definecolor{currentstroke}{rgb}{0.000000,0.000000,0.000000}%
\pgfsetstrokecolor{currentstroke}%
\pgfsetdash{}{0pt}%
\pgfsys@defobject{currentmarker}{\pgfqpoint{0.000000in}{-0.048611in}}{\pgfqpoint{0.000000in}{0.000000in}}{%
\pgfpathmoveto{\pgfqpoint{0.000000in}{0.000000in}}%
\pgfpathlineto{\pgfqpoint{0.000000in}{-0.048611in}}%
\pgfusepath{stroke,fill}%
}%
\begin{pgfscope}%
\pgfsys@transformshift{3.181482in}{0.499691in}%
\pgfsys@useobject{currentmarker}{}%
\end{pgfscope}%
\end{pgfscope}%
\begin{pgfscope}%
\definecolor{textcolor}{rgb}{0.000000,0.000000,0.000000}%
\pgfsetstrokecolor{textcolor}%
\pgfsetfillcolor{textcolor}%
\pgftext[x=3.181482in,y=0.402469in,,top]{\color{textcolor}\rmfamily\fontsize{10.000000}{12.000000}\selectfont \(\displaystyle 4\)}%
\end{pgfscope}%
\begin{pgfscope}%
\definecolor{textcolor}{rgb}{0.000000,0.000000,0.000000}%
\pgfsetstrokecolor{textcolor}%
\pgfsetfillcolor{textcolor}%
\pgftext[x=1.996297in,y=0.223457in,,top]{\color{textcolor}\rmfamily\fontsize{10.000000}{12.000000}\selectfont Number of pieces \(\displaystyle K\)}%
\end{pgfscope}%
\begin{pgfscope}%
\pgfsetbuttcap%
\pgfsetroundjoin%
\definecolor{currentfill}{rgb}{0.000000,0.000000,0.000000}%
\pgfsetfillcolor{currentfill}%
\pgfsetlinewidth{0.803000pt}%
\definecolor{currentstroke}{rgb}{0.000000,0.000000,0.000000}%
\pgfsetstrokecolor{currentstroke}%
\pgfsetdash{}{0pt}%
\pgfsys@defobject{currentmarker}{\pgfqpoint{-0.048611in}{0.000000in}}{\pgfqpoint{0.000000in}{0.000000in}}{%
\pgfpathmoveto{\pgfqpoint{0.000000in}{0.000000in}}%
\pgfpathlineto{\pgfqpoint{-0.048611in}{0.000000in}}%
\pgfusepath{stroke,fill}%
}%
\begin{pgfscope}%
\pgfsys@transformshift{0.692593in}{0.755915in}%
\pgfsys@useobject{currentmarker}{}%
\end{pgfscope}%
\end{pgfscope}%
\begin{pgfscope}%
\definecolor{textcolor}{rgb}{0.000000,0.000000,0.000000}%
\pgfsetstrokecolor{textcolor}%
\pgfsetfillcolor{textcolor}%
\pgftext[x=0.279012in,y=0.707690in,left,base]{\color{textcolor}\rmfamily\fontsize{10.000000}{12.000000}\selectfont \(\displaystyle -250\)}%
\end{pgfscope}%
\begin{pgfscope}%
\pgfsetbuttcap%
\pgfsetroundjoin%
\definecolor{currentfill}{rgb}{0.000000,0.000000,0.000000}%
\pgfsetfillcolor{currentfill}%
\pgfsetlinewidth{0.803000pt}%
\definecolor{currentstroke}{rgb}{0.000000,0.000000,0.000000}%
\pgfsetstrokecolor{currentstroke}%
\pgfsetdash{}{0pt}%
\pgfsys@defobject{currentmarker}{\pgfqpoint{-0.048611in}{0.000000in}}{\pgfqpoint{0.000000in}{0.000000in}}{%
\pgfpathmoveto{\pgfqpoint{0.000000in}{0.000000in}}%
\pgfpathlineto{\pgfqpoint{-0.048611in}{0.000000in}}%
\pgfusepath{stroke,fill}%
}%
\begin{pgfscope}%
\pgfsys@transformshift{0.692593in}{1.083461in}%
\pgfsys@useobject{currentmarker}{}%
\end{pgfscope}%
\end{pgfscope}%
\begin{pgfscope}%
\definecolor{textcolor}{rgb}{0.000000,0.000000,0.000000}%
\pgfsetstrokecolor{textcolor}%
\pgfsetfillcolor{textcolor}%
\pgftext[x=0.279012in,y=1.035235in,left,base]{\color{textcolor}\rmfamily\fontsize{10.000000}{12.000000}\selectfont \(\displaystyle -225\)}%
\end{pgfscope}%
\begin{pgfscope}%
\pgfsetbuttcap%
\pgfsetroundjoin%
\definecolor{currentfill}{rgb}{0.000000,0.000000,0.000000}%
\pgfsetfillcolor{currentfill}%
\pgfsetlinewidth{0.803000pt}%
\definecolor{currentstroke}{rgb}{0.000000,0.000000,0.000000}%
\pgfsetstrokecolor{currentstroke}%
\pgfsetdash{}{0pt}%
\pgfsys@defobject{currentmarker}{\pgfqpoint{-0.048611in}{0.000000in}}{\pgfqpoint{0.000000in}{0.000000in}}{%
\pgfpathmoveto{\pgfqpoint{0.000000in}{0.000000in}}%
\pgfpathlineto{\pgfqpoint{-0.048611in}{0.000000in}}%
\pgfusepath{stroke,fill}%
}%
\begin{pgfscope}%
\pgfsys@transformshift{0.692593in}{1.411006in}%
\pgfsys@useobject{currentmarker}{}%
\end{pgfscope}%
\end{pgfscope}%
\begin{pgfscope}%
\definecolor{textcolor}{rgb}{0.000000,0.000000,0.000000}%
\pgfsetstrokecolor{textcolor}%
\pgfsetfillcolor{textcolor}%
\pgftext[x=0.279012in,y=1.362781in,left,base]{\color{textcolor}\rmfamily\fontsize{10.000000}{12.000000}\selectfont \(\displaystyle -200\)}%
\end{pgfscope}%
\begin{pgfscope}%
\pgfsetbuttcap%
\pgfsetroundjoin%
\definecolor{currentfill}{rgb}{0.000000,0.000000,0.000000}%
\pgfsetfillcolor{currentfill}%
\pgfsetlinewidth{0.803000pt}%
\definecolor{currentstroke}{rgb}{0.000000,0.000000,0.000000}%
\pgfsetstrokecolor{currentstroke}%
\pgfsetdash{}{0pt}%
\pgfsys@defobject{currentmarker}{\pgfqpoint{-0.048611in}{0.000000in}}{\pgfqpoint{0.000000in}{0.000000in}}{%
\pgfpathmoveto{\pgfqpoint{0.000000in}{0.000000in}}%
\pgfpathlineto{\pgfqpoint{-0.048611in}{0.000000in}}%
\pgfusepath{stroke,fill}%
}%
\begin{pgfscope}%
\pgfsys@transformshift{0.692593in}{1.738552in}%
\pgfsys@useobject{currentmarker}{}%
\end{pgfscope}%
\end{pgfscope}%
\begin{pgfscope}%
\definecolor{textcolor}{rgb}{0.000000,0.000000,0.000000}%
\pgfsetstrokecolor{textcolor}%
\pgfsetfillcolor{textcolor}%
\pgftext[x=0.279012in,y=1.690327in,left,base]{\color{textcolor}\rmfamily\fontsize{10.000000}{12.000000}\selectfont \(\displaystyle -175\)}%
\end{pgfscope}%
\begin{pgfscope}%
\pgfsetbuttcap%
\pgfsetroundjoin%
\definecolor{currentfill}{rgb}{0.000000,0.000000,0.000000}%
\pgfsetfillcolor{currentfill}%
\pgfsetlinewidth{0.803000pt}%
\definecolor{currentstroke}{rgb}{0.000000,0.000000,0.000000}%
\pgfsetstrokecolor{currentstroke}%
\pgfsetdash{}{0pt}%
\pgfsys@defobject{currentmarker}{\pgfqpoint{-0.048611in}{0.000000in}}{\pgfqpoint{0.000000in}{0.000000in}}{%
\pgfpathmoveto{\pgfqpoint{0.000000in}{0.000000in}}%
\pgfpathlineto{\pgfqpoint{-0.048611in}{0.000000in}}%
\pgfusepath{stroke,fill}%
}%
\begin{pgfscope}%
\pgfsys@transformshift{0.692593in}{2.066097in}%
\pgfsys@useobject{currentmarker}{}%
\end{pgfscope}%
\end{pgfscope}%
\begin{pgfscope}%
\definecolor{textcolor}{rgb}{0.000000,0.000000,0.000000}%
\pgfsetstrokecolor{textcolor}%
\pgfsetfillcolor{textcolor}%
\pgftext[x=0.279012in,y=2.017872in,left,base]{\color{textcolor}\rmfamily\fontsize{10.000000}{12.000000}\selectfont \(\displaystyle -150\)}%
\end{pgfscope}%
\begin{pgfscope}%
\pgfsetbuttcap%
\pgfsetroundjoin%
\definecolor{currentfill}{rgb}{0.000000,0.000000,0.000000}%
\pgfsetfillcolor{currentfill}%
\pgfsetlinewidth{0.803000pt}%
\definecolor{currentstroke}{rgb}{0.000000,0.000000,0.000000}%
\pgfsetstrokecolor{currentstroke}%
\pgfsetdash{}{0pt}%
\pgfsys@defobject{currentmarker}{\pgfqpoint{-0.048611in}{0.000000in}}{\pgfqpoint{0.000000in}{0.000000in}}{%
\pgfpathmoveto{\pgfqpoint{0.000000in}{0.000000in}}%
\pgfpathlineto{\pgfqpoint{-0.048611in}{0.000000in}}%
\pgfusepath{stroke,fill}%
}%
\begin{pgfscope}%
\pgfsys@transformshift{0.692593in}{2.393643in}%
\pgfsys@useobject{currentmarker}{}%
\end{pgfscope}%
\end{pgfscope}%
\begin{pgfscope}%
\definecolor{textcolor}{rgb}{0.000000,0.000000,0.000000}%
\pgfsetstrokecolor{textcolor}%
\pgfsetfillcolor{textcolor}%
\pgftext[x=0.279012in,y=2.345418in,left,base]{\color{textcolor}\rmfamily\fontsize{10.000000}{12.000000}\selectfont \(\displaystyle -125\)}%
\end{pgfscope}%
\begin{pgfscope}%
\pgfsetbuttcap%
\pgfsetroundjoin%
\definecolor{currentfill}{rgb}{0.000000,0.000000,0.000000}%
\pgfsetfillcolor{currentfill}%
\pgfsetlinewidth{0.803000pt}%
\definecolor{currentstroke}{rgb}{0.000000,0.000000,0.000000}%
\pgfsetstrokecolor{currentstroke}%
\pgfsetdash{}{0pt}%
\pgfsys@defobject{currentmarker}{\pgfqpoint{-0.048611in}{0.000000in}}{\pgfqpoint{0.000000in}{0.000000in}}{%
\pgfpathmoveto{\pgfqpoint{0.000000in}{0.000000in}}%
\pgfpathlineto{\pgfqpoint{-0.048611in}{0.000000in}}%
\pgfusepath{stroke,fill}%
}%
\begin{pgfscope}%
\pgfsys@transformshift{0.692593in}{2.721189in}%
\pgfsys@useobject{currentmarker}{}%
\end{pgfscope}%
\end{pgfscope}%
\begin{pgfscope}%
\definecolor{textcolor}{rgb}{0.000000,0.000000,0.000000}%
\pgfsetstrokecolor{textcolor}%
\pgfsetfillcolor{textcolor}%
\pgftext[x=0.279012in,y=2.672963in,left,base]{\color{textcolor}\rmfamily\fontsize{10.000000}{12.000000}\selectfont \(\displaystyle -100\)}%
\end{pgfscope}%
\begin{pgfscope}%
\pgfsetbuttcap%
\pgfsetroundjoin%
\definecolor{currentfill}{rgb}{0.000000,0.000000,0.000000}%
\pgfsetfillcolor{currentfill}%
\pgfsetlinewidth{0.803000pt}%
\definecolor{currentstroke}{rgb}{0.000000,0.000000,0.000000}%
\pgfsetstrokecolor{currentstroke}%
\pgfsetdash{}{0pt}%
\pgfsys@defobject{currentmarker}{\pgfqpoint{-0.048611in}{0.000000in}}{\pgfqpoint{0.000000in}{0.000000in}}{%
\pgfpathmoveto{\pgfqpoint{0.000000in}{0.000000in}}%
\pgfpathlineto{\pgfqpoint{-0.048611in}{0.000000in}}%
\pgfusepath{stroke,fill}%
}%
\begin{pgfscope}%
\pgfsys@transformshift{0.692593in}{3.048734in}%
\pgfsys@useobject{currentmarker}{}%
\end{pgfscope}%
\end{pgfscope}%
\begin{pgfscope}%
\definecolor{textcolor}{rgb}{0.000000,0.000000,0.000000}%
\pgfsetstrokecolor{textcolor}%
\pgfsetfillcolor{textcolor}%
\pgftext[x=0.348457in,y=3.000509in,left,base]{\color{textcolor}\rmfamily\fontsize{10.000000}{12.000000}\selectfont \(\displaystyle -75\)}%
\end{pgfscope}%
\begin{pgfscope}%
\definecolor{textcolor}{rgb}{0.000000,0.000000,0.000000}%
\pgfsetstrokecolor{textcolor}%
\pgfsetfillcolor{textcolor}%
\pgftext[x=0.223457in,y=1.800309in,,bottom,rotate=90.000000]{\color{textcolor}\rmfamily\fontsize{10.000000}{12.000000}\selectfont Energy}%
\end{pgfscope}%
\begin{pgfscope}%
\pgfpathrectangle{\pgfqpoint{0.692593in}{0.499691in}}{\pgfqpoint{2.607407in}{2.601235in}}%
\pgfusepath{clip}%
\pgfsetrectcap%
\pgfsetroundjoin%
\pgfsetlinewidth{1.505625pt}%
\definecolor{currentstroke}{rgb}{0.121569,0.466667,0.705882}%
\pgfsetstrokecolor{currentstroke}%
\pgfsetdash{}{0pt}%
\pgfpathmoveto{\pgfqpoint{0.811112in}{0.617929in}}%
\pgfpathlineto{\pgfqpoint{1.601235in}{1.584255in}}%
\pgfpathlineto{\pgfqpoint{2.391358in}{2.039567in}}%
\pgfpathlineto{\pgfqpoint{3.181482in}{2.366453in}}%
\pgfusepath{stroke}%
\end{pgfscope}%
\begin{pgfscope}%
\pgfpathrectangle{\pgfqpoint{0.692593in}{0.499691in}}{\pgfqpoint{2.607407in}{2.601235in}}%
\pgfusepath{clip}%
\pgfsetbuttcap%
\pgfsetroundjoin%
\definecolor{currentfill}{rgb}{0.121569,0.466667,0.705882}%
\pgfsetfillcolor{currentfill}%
\pgfsetlinewidth{1.003750pt}%
\definecolor{currentstroke}{rgb}{0.121569,0.466667,0.705882}%
\pgfsetstrokecolor{currentstroke}%
\pgfsetdash{}{0pt}%
\pgfsys@defobject{currentmarker}{\pgfqpoint{-0.041667in}{-0.041667in}}{\pgfqpoint{0.041667in}{0.041667in}}{%
\pgfpathmoveto{\pgfqpoint{-0.041667in}{-0.041667in}}%
\pgfpathlineto{\pgfqpoint{0.041667in}{0.041667in}}%
\pgfpathmoveto{\pgfqpoint{-0.041667in}{0.041667in}}%
\pgfpathlineto{\pgfqpoint{0.041667in}{-0.041667in}}%
\pgfusepath{stroke,fill}%
}%
\begin{pgfscope}%
\pgfsys@transformshift{0.811112in}{0.617929in}%
\pgfsys@useobject{currentmarker}{}%
\end{pgfscope}%
\begin{pgfscope}%
\pgfsys@transformshift{1.601235in}{1.584255in}%
\pgfsys@useobject{currentmarker}{}%
\end{pgfscope}%
\begin{pgfscope}%
\pgfsys@transformshift{2.391358in}{2.039567in}%
\pgfsys@useobject{currentmarker}{}%
\end{pgfscope}%
\begin{pgfscope}%
\pgfsys@transformshift{3.181482in}{2.366453in}%
\pgfsys@useobject{currentmarker}{}%
\end{pgfscope}%
\end{pgfscope}%
\begin{pgfscope}%
\pgfpathrectangle{\pgfqpoint{0.692593in}{0.499691in}}{\pgfqpoint{2.607407in}{2.601235in}}%
\pgfusepath{clip}%
\pgfsetrectcap%
\pgfsetroundjoin%
\pgfsetlinewidth{1.505625pt}%
\definecolor{currentstroke}{rgb}{1.000000,0.498039,0.054902}%
\pgfsetstrokecolor{currentstroke}%
\pgfsetdash{}{0pt}%
\pgfpathmoveto{\pgfqpoint{0.811112in}{1.287992in}}%
\pgfpathlineto{\pgfqpoint{1.601235in}{2.456899in}}%
\pgfpathlineto{\pgfqpoint{2.391358in}{2.898629in}}%
\pgfpathlineto{\pgfqpoint{3.181482in}{2.951991in}}%
\pgfusepath{stroke}%
\end{pgfscope}%
\begin{pgfscope}%
\pgfpathrectangle{\pgfqpoint{0.692593in}{0.499691in}}{\pgfqpoint{2.607407in}{2.601235in}}%
\pgfusepath{clip}%
\pgfsetbuttcap%
\pgfsetroundjoin%
\definecolor{currentfill}{rgb}{1.000000,0.498039,0.054902}%
\pgfsetfillcolor{currentfill}%
\pgfsetlinewidth{1.003750pt}%
\definecolor{currentstroke}{rgb}{1.000000,0.498039,0.054902}%
\pgfsetstrokecolor{currentstroke}%
\pgfsetdash{}{0pt}%
\pgfsys@defobject{currentmarker}{\pgfqpoint{-0.041667in}{-0.041667in}}{\pgfqpoint{0.041667in}{0.041667in}}{%
\pgfpathmoveto{\pgfqpoint{-0.041667in}{-0.041667in}}%
\pgfpathlineto{\pgfqpoint{0.041667in}{0.041667in}}%
\pgfpathmoveto{\pgfqpoint{-0.041667in}{0.041667in}}%
\pgfpathlineto{\pgfqpoint{0.041667in}{-0.041667in}}%
\pgfusepath{stroke,fill}%
}%
\begin{pgfscope}%
\pgfsys@transformshift{0.811112in}{1.287992in}%
\pgfsys@useobject{currentmarker}{}%
\end{pgfscope}%
\begin{pgfscope}%
\pgfsys@transformshift{1.601235in}{2.456899in}%
\pgfsys@useobject{currentmarker}{}%
\end{pgfscope}%
\begin{pgfscope}%
\pgfsys@transformshift{2.391358in}{2.898629in}%
\pgfsys@useobject{currentmarker}{}%
\end{pgfscope}%
\begin{pgfscope}%
\pgfsys@transformshift{3.181482in}{2.951991in}%
\pgfsys@useobject{currentmarker}{}%
\end{pgfscope}%
\end{pgfscope}%
\begin{pgfscope}%
\pgfpathrectangle{\pgfqpoint{0.692593in}{0.499691in}}{\pgfqpoint{2.607407in}{2.601235in}}%
\pgfusepath{clip}%
\pgfsetrectcap%
\pgfsetroundjoin%
\pgfsetlinewidth{1.505625pt}%
\definecolor{currentstroke}{rgb}{0.172549,0.627451,0.172549}%
\pgfsetstrokecolor{currentstroke}%
\pgfsetdash{}{0pt}%
\pgfpathmoveto{\pgfqpoint{0.811112in}{1.949438in}}%
\pgfpathlineto{\pgfqpoint{1.601235in}{2.934239in}}%
\pgfpathlineto{\pgfqpoint{2.391358in}{2.975562in}}%
\pgfpathlineto{\pgfqpoint{3.181482in}{2.981465in}}%
\pgfusepath{stroke}%
\end{pgfscope}%
\begin{pgfscope}%
\pgfpathrectangle{\pgfqpoint{0.692593in}{0.499691in}}{\pgfqpoint{2.607407in}{2.601235in}}%
\pgfusepath{clip}%
\pgfsetbuttcap%
\pgfsetroundjoin%
\definecolor{currentfill}{rgb}{0.172549,0.627451,0.172549}%
\pgfsetfillcolor{currentfill}%
\pgfsetlinewidth{1.003750pt}%
\definecolor{currentstroke}{rgb}{0.172549,0.627451,0.172549}%
\pgfsetstrokecolor{currentstroke}%
\pgfsetdash{}{0pt}%
\pgfsys@defobject{currentmarker}{\pgfqpoint{-0.041667in}{-0.041667in}}{\pgfqpoint{0.041667in}{0.041667in}}{%
\pgfpathmoveto{\pgfqpoint{-0.041667in}{-0.041667in}}%
\pgfpathlineto{\pgfqpoint{0.041667in}{0.041667in}}%
\pgfpathmoveto{\pgfqpoint{-0.041667in}{0.041667in}}%
\pgfpathlineto{\pgfqpoint{0.041667in}{-0.041667in}}%
\pgfusepath{stroke,fill}%
}%
\begin{pgfscope}%
\pgfsys@transformshift{0.811112in}{1.949438in}%
\pgfsys@useobject{currentmarker}{}%
\end{pgfscope}%
\begin{pgfscope}%
\pgfsys@transformshift{1.601235in}{2.934239in}%
\pgfsys@useobject{currentmarker}{}%
\end{pgfscope}%
\begin{pgfscope}%
\pgfsys@transformshift{2.391358in}{2.975562in}%
\pgfsys@useobject{currentmarker}{}%
\end{pgfscope}%
\begin{pgfscope}%
\pgfsys@transformshift{3.181482in}{2.981465in}%
\pgfsys@useobject{currentmarker}{}%
\end{pgfscope}%
\end{pgfscope}%
\begin{pgfscope}%
\pgfpathrectangle{\pgfqpoint{0.692593in}{0.499691in}}{\pgfqpoint{2.607407in}{2.601235in}}%
\pgfusepath{clip}%
\pgfsetrectcap%
\pgfsetroundjoin%
\pgfsetlinewidth{1.505625pt}%
\definecolor{currentstroke}{rgb}{0.839216,0.152941,0.156863}%
\pgfsetstrokecolor{currentstroke}%
\pgfsetdash{}{0pt}%
\pgfpathmoveto{\pgfqpoint{0.811112in}{2.863340in}}%
\pgfpathlineto{\pgfqpoint{1.601235in}{2.973602in}}%
\pgfpathlineto{\pgfqpoint{2.391358in}{2.982055in}}%
\pgfpathlineto{\pgfqpoint{3.181482in}{2.982688in}}%
\pgfusepath{stroke}%
\end{pgfscope}%
\begin{pgfscope}%
\pgfpathrectangle{\pgfqpoint{0.692593in}{0.499691in}}{\pgfqpoint{2.607407in}{2.601235in}}%
\pgfusepath{clip}%
\pgfsetbuttcap%
\pgfsetroundjoin%
\definecolor{currentfill}{rgb}{0.839216,0.152941,0.156863}%
\pgfsetfillcolor{currentfill}%
\pgfsetlinewidth{1.003750pt}%
\definecolor{currentstroke}{rgb}{0.839216,0.152941,0.156863}%
\pgfsetstrokecolor{currentstroke}%
\pgfsetdash{}{0pt}%
\pgfsys@defobject{currentmarker}{\pgfqpoint{-0.041667in}{-0.041667in}}{\pgfqpoint{0.041667in}{0.041667in}}{%
\pgfpathmoveto{\pgfqpoint{-0.041667in}{-0.041667in}}%
\pgfpathlineto{\pgfqpoint{0.041667in}{0.041667in}}%
\pgfpathmoveto{\pgfqpoint{-0.041667in}{0.041667in}}%
\pgfpathlineto{\pgfqpoint{0.041667in}{-0.041667in}}%
\pgfusepath{stroke,fill}%
}%
\begin{pgfscope}%
\pgfsys@transformshift{0.811112in}{2.863340in}%
\pgfsys@useobject{currentmarker}{}%
\end{pgfscope}%
\begin{pgfscope}%
\pgfsys@transformshift{1.601235in}{2.973602in}%
\pgfsys@useobject{currentmarker}{}%
\end{pgfscope}%
\begin{pgfscope}%
\pgfsys@transformshift{2.391358in}{2.982055in}%
\pgfsys@useobject{currentmarker}{}%
\end{pgfscope}%
\begin{pgfscope}%
\pgfsys@transformshift{3.181482in}{2.982688in}%
\pgfsys@useobject{currentmarker}{}%
\end{pgfscope}%
\end{pgfscope}%
\begin{pgfscope}%
\pgfsetrectcap%
\pgfsetmiterjoin%
\pgfsetlinewidth{0.803000pt}%
\definecolor{currentstroke}{rgb}{0.000000,0.000000,0.000000}%
\pgfsetstrokecolor{currentstroke}%
\pgfsetdash{}{0pt}%
\pgfpathmoveto{\pgfqpoint{0.692593in}{0.499691in}}%
\pgfpathlineto{\pgfqpoint{0.692593in}{3.100926in}}%
\pgfusepath{stroke}%
\end{pgfscope}%
\begin{pgfscope}%
\pgfsetrectcap%
\pgfsetmiterjoin%
\pgfsetlinewidth{0.803000pt}%
\definecolor{currentstroke}{rgb}{0.000000,0.000000,0.000000}%
\pgfsetstrokecolor{currentstroke}%
\pgfsetdash{}{0pt}%
\pgfpathmoveto{\pgfqpoint{3.300000in}{0.499691in}}%
\pgfpathlineto{\pgfqpoint{3.300000in}{3.100926in}}%
\pgfusepath{stroke}%
\end{pgfscope}%
\begin{pgfscope}%
\pgfsetrectcap%
\pgfsetmiterjoin%
\pgfsetlinewidth{0.803000pt}%
\definecolor{currentstroke}{rgb}{0.000000,0.000000,0.000000}%
\pgfsetstrokecolor{currentstroke}%
\pgfsetdash{}{0pt}%
\pgfpathmoveto{\pgfqpoint{0.692593in}{0.499691in}}%
\pgfpathlineto{\pgfqpoint{3.300000in}{0.499691in}}%
\pgfusepath{stroke}%
\end{pgfscope}%
\begin{pgfscope}%
\pgfsetrectcap%
\pgfsetmiterjoin%
\pgfsetlinewidth{0.803000pt}%
\definecolor{currentstroke}{rgb}{0.000000,0.000000,0.000000}%
\pgfsetstrokecolor{currentstroke}%
\pgfsetdash{}{0pt}%
\pgfpathmoveto{\pgfqpoint{0.692593in}{3.100926in}}%
\pgfpathlineto{\pgfqpoint{3.300000in}{3.100926in}}%
\pgfusepath{stroke}%
\end{pgfscope}%
\begin{pgfscope}%
\definecolor{textcolor}{rgb}{0.000000,0.000000,0.000000}%
\pgfsetstrokecolor{textcolor}%
\pgfsetfillcolor{textcolor}%
\pgftext[x=1.996297in,y=3.184260in,,base]{\color{textcolor}\rmfamily\fontsize{12.000000}{14.400000}\selectfont Potts regularization}%
\end{pgfscope}%
\begin{pgfscope}%
\pgfsetbuttcap%
\pgfsetmiterjoin%
\definecolor{currentfill}{rgb}{1.000000,1.000000,1.000000}%
\pgfsetfillcolor{currentfill}%
\pgfsetfillopacity{0.800000}%
\pgfsetlinewidth{1.003750pt}%
\definecolor{currentstroke}{rgb}{0.800000,0.800000,0.800000}%
\pgfsetstrokecolor{currentstroke}%
\pgfsetstrokeopacity{0.800000}%
\pgfsetdash{}{0pt}%
\pgfpathmoveto{\pgfqpoint{2.333951in}{0.569136in}}%
\pgfpathlineto{\pgfqpoint{3.202778in}{0.569136in}}%
\pgfpathquadraticcurveto{\pgfqpoint{3.230556in}{0.569136in}}{\pgfqpoint{3.230556in}{0.596913in}}%
\pgfpathlineto{\pgfqpoint{3.230556in}{1.357716in}}%
\pgfpathquadraticcurveto{\pgfqpoint{3.230556in}{1.385493in}}{\pgfqpoint{3.202778in}{1.385493in}}%
\pgfpathlineto{\pgfqpoint{2.333951in}{1.385493in}}%
\pgfpathquadraticcurveto{\pgfqpoint{2.306173in}{1.385493in}}{\pgfqpoint{2.306173in}{1.357716in}}%
\pgfpathlineto{\pgfqpoint{2.306173in}{0.596913in}}%
\pgfpathquadraticcurveto{\pgfqpoint{2.306173in}{0.569136in}}{\pgfqpoint{2.333951in}{0.569136in}}%
\pgfpathclose%
\pgfusepath{stroke,fill}%
\end{pgfscope}%
\begin{pgfscope}%
\pgfsetrectcap%
\pgfsetroundjoin%
\pgfsetlinewidth{1.505625pt}%
\definecolor{currentstroke}{rgb}{0.121569,0.466667,0.705882}%
\pgfsetstrokecolor{currentstroke}%
\pgfsetdash{}{0pt}%
\pgfpathmoveto{\pgfqpoint{2.361728in}{1.281327in}}%
\pgfpathlineto{\pgfqpoint{2.639506in}{1.281327in}}%
\pgfusepath{stroke}%
\end{pgfscope}%
\begin{pgfscope}%
\pgfsetbuttcap%
\pgfsetroundjoin%
\definecolor{currentfill}{rgb}{0.121569,0.466667,0.705882}%
\pgfsetfillcolor{currentfill}%
\pgfsetlinewidth{1.003750pt}%
\definecolor{currentstroke}{rgb}{0.121569,0.466667,0.705882}%
\pgfsetstrokecolor{currentstroke}%
\pgfsetdash{}{0pt}%
\pgfsys@defobject{currentmarker}{\pgfqpoint{-0.041667in}{-0.041667in}}{\pgfqpoint{0.041667in}{0.041667in}}{%
\pgfpathmoveto{\pgfqpoint{-0.041667in}{-0.041667in}}%
\pgfpathlineto{\pgfqpoint{0.041667in}{0.041667in}}%
\pgfpathmoveto{\pgfqpoint{-0.041667in}{0.041667in}}%
\pgfpathlineto{\pgfqpoint{0.041667in}{-0.041667in}}%
\pgfusepath{stroke,fill}%
}%
\begin{pgfscope}%
\pgfsys@transformshift{2.500617in}{1.281327in}%
\pgfsys@useobject{currentmarker}{}%
\end{pgfscope}%
\end{pgfscope}%
\begin{pgfscope}%
\definecolor{textcolor}{rgb}{0.000000,0.000000,0.000000}%
\pgfsetstrokecolor{textcolor}%
\pgfsetfillcolor{textcolor}%
\pgftext[x=2.750617in,y=1.232716in,left,base]{\color{textcolor}\rmfamily\fontsize{10.000000}{12.000000}\selectfont deg\(\displaystyle =1\)}%
\end{pgfscope}%
\begin{pgfscope}%
\pgfsetrectcap%
\pgfsetroundjoin%
\pgfsetlinewidth{1.505625pt}%
\definecolor{currentstroke}{rgb}{1.000000,0.498039,0.054902}%
\pgfsetstrokecolor{currentstroke}%
\pgfsetdash{}{0pt}%
\pgfpathmoveto{\pgfqpoint{2.361728in}{1.087654in}}%
\pgfpathlineto{\pgfqpoint{2.639506in}{1.087654in}}%
\pgfusepath{stroke}%
\end{pgfscope}%
\begin{pgfscope}%
\pgfsetbuttcap%
\pgfsetroundjoin%
\definecolor{currentfill}{rgb}{1.000000,0.498039,0.054902}%
\pgfsetfillcolor{currentfill}%
\pgfsetlinewidth{1.003750pt}%
\definecolor{currentstroke}{rgb}{1.000000,0.498039,0.054902}%
\pgfsetstrokecolor{currentstroke}%
\pgfsetdash{}{0pt}%
\pgfsys@defobject{currentmarker}{\pgfqpoint{-0.041667in}{-0.041667in}}{\pgfqpoint{0.041667in}{0.041667in}}{%
\pgfpathmoveto{\pgfqpoint{-0.041667in}{-0.041667in}}%
\pgfpathlineto{\pgfqpoint{0.041667in}{0.041667in}}%
\pgfpathmoveto{\pgfqpoint{-0.041667in}{0.041667in}}%
\pgfpathlineto{\pgfqpoint{0.041667in}{-0.041667in}}%
\pgfusepath{stroke,fill}%
}%
\begin{pgfscope}%
\pgfsys@transformshift{2.500617in}{1.087654in}%
\pgfsys@useobject{currentmarker}{}%
\end{pgfscope}%
\end{pgfscope}%
\begin{pgfscope}%
\definecolor{textcolor}{rgb}{0.000000,0.000000,0.000000}%
\pgfsetstrokecolor{textcolor}%
\pgfsetfillcolor{textcolor}%
\pgftext[x=2.750617in,y=1.039043in,left,base]{\color{textcolor}\rmfamily\fontsize{10.000000}{12.000000}\selectfont deg\(\displaystyle =2\)}%
\end{pgfscope}%
\begin{pgfscope}%
\pgfsetrectcap%
\pgfsetroundjoin%
\pgfsetlinewidth{1.505625pt}%
\definecolor{currentstroke}{rgb}{0.172549,0.627451,0.172549}%
\pgfsetstrokecolor{currentstroke}%
\pgfsetdash{}{0pt}%
\pgfpathmoveto{\pgfqpoint{2.361728in}{0.893981in}}%
\pgfpathlineto{\pgfqpoint{2.639506in}{0.893981in}}%
\pgfusepath{stroke}%
\end{pgfscope}%
\begin{pgfscope}%
\pgfsetbuttcap%
\pgfsetroundjoin%
\definecolor{currentfill}{rgb}{0.172549,0.627451,0.172549}%
\pgfsetfillcolor{currentfill}%
\pgfsetlinewidth{1.003750pt}%
\definecolor{currentstroke}{rgb}{0.172549,0.627451,0.172549}%
\pgfsetstrokecolor{currentstroke}%
\pgfsetdash{}{0pt}%
\pgfsys@defobject{currentmarker}{\pgfqpoint{-0.041667in}{-0.041667in}}{\pgfqpoint{0.041667in}{0.041667in}}{%
\pgfpathmoveto{\pgfqpoint{-0.041667in}{-0.041667in}}%
\pgfpathlineto{\pgfqpoint{0.041667in}{0.041667in}}%
\pgfpathmoveto{\pgfqpoint{-0.041667in}{0.041667in}}%
\pgfpathlineto{\pgfqpoint{0.041667in}{-0.041667in}}%
\pgfusepath{stroke,fill}%
}%
\begin{pgfscope}%
\pgfsys@transformshift{2.500617in}{0.893981in}%
\pgfsys@useobject{currentmarker}{}%
\end{pgfscope}%
\end{pgfscope}%
\begin{pgfscope}%
\definecolor{textcolor}{rgb}{0.000000,0.000000,0.000000}%
\pgfsetstrokecolor{textcolor}%
\pgfsetfillcolor{textcolor}%
\pgftext[x=2.750617in,y=0.845370in,left,base]{\color{textcolor}\rmfamily\fontsize{10.000000}{12.000000}\selectfont deg\(\displaystyle =3\)}%
\end{pgfscope}%
\begin{pgfscope}%
\pgfsetrectcap%
\pgfsetroundjoin%
\pgfsetlinewidth{1.505625pt}%
\definecolor{currentstroke}{rgb}{0.839216,0.152941,0.156863}%
\pgfsetstrokecolor{currentstroke}%
\pgfsetdash{}{0pt}%
\pgfpathmoveto{\pgfqpoint{2.361728in}{0.700308in}}%
\pgfpathlineto{\pgfqpoint{2.639506in}{0.700308in}}%
\pgfusepath{stroke}%
\end{pgfscope}%
\begin{pgfscope}%
\pgfsetbuttcap%
\pgfsetroundjoin%
\definecolor{currentfill}{rgb}{0.839216,0.152941,0.156863}%
\pgfsetfillcolor{currentfill}%
\pgfsetlinewidth{1.003750pt}%
\definecolor{currentstroke}{rgb}{0.839216,0.152941,0.156863}%
\pgfsetstrokecolor{currentstroke}%
\pgfsetdash{}{0pt}%
\pgfsys@defobject{currentmarker}{\pgfqpoint{-0.041667in}{-0.041667in}}{\pgfqpoint{0.041667in}{0.041667in}}{%
\pgfpathmoveto{\pgfqpoint{-0.041667in}{-0.041667in}}%
\pgfpathlineto{\pgfqpoint{0.041667in}{0.041667in}}%
\pgfpathmoveto{\pgfqpoint{-0.041667in}{0.041667in}}%
\pgfpathlineto{\pgfqpoint{0.041667in}{-0.041667in}}%
\pgfusepath{stroke,fill}%
}%
\begin{pgfscope}%
\pgfsys@transformshift{2.500617in}{0.700308in}%
\pgfsys@useobject{currentmarker}{}%
\end{pgfscope}%
\end{pgfscope}%
\begin{pgfscope}%
\definecolor{textcolor}{rgb}{0.000000,0.000000,0.000000}%
\pgfsetstrokecolor{textcolor}%
\pgfsetfillcolor{textcolor}%
\pgftext[x=2.750617in,y=0.651697in,left,base]{\color{textcolor}\rmfamily\fontsize{10.000000}{12.000000}\selectfont deg\(\displaystyle =4\)}%
\end{pgfscope}%
\end{pgfpicture}%
\makeatother%
\endgroup%